\documentclass{amsart}
\usepackage{latexsym,amssymb,amsmath}
\usepackage{tikz-cd}
\usetikzlibrary{matrix,arrows,decorations,positioning}

\theoremstyle{plain}
\newtheorem{lemma}{Lemma}[section]
\newtheorem{theorem}[lemma]{Theorem}
\newtheorem{corollary}[lemma]{Corollary}

\newtheorem{proposition}[lemma]{Proposition}

\theoremstyle{definition}

\newtheorem{example}[lemma]{Example}
\newtheorem{remark}[lemma]{Remark}
\newtheorem{definition}[lemma]{Definition}
\newtheorem{construction}[lemma]{Construction}

\begin{document}

\title[ ]{Lcm-lattice, Taylor Bases and Minimal Free Resolutions of a Monomial ideal}
\author[ ]{\text{Ri-Xiang Chen}}
\address{Department of Mathematics, Nanjing University of Science and Technology, Nanjing, Jiangsu, 210094, P.R.China}
\email{rc429@cornell.edu}
\date{}
\begin{abstract}
We use the lcm-lattice of a monomial ideal to study its minimal free resolutions. 
A new concept called a Taylor basis of a minimal free resolution is introduced 
and then used throughout the paper. We give a method of constructing minimal 
free resolutions of a monomial ideal from its lcm-lattice, which is called the atomic 
lattice resolution theory. Some applications of this theory is given. 
As the main application, we rewrite the theory of poset resolutions, and we obtain
an approximation formula for minimal free resolutions of \emph{all} monomial ideals. 
\end{abstract}
\maketitle
\raggedbottom

\section{Introduction}

\noindent Throughout the paper let $S=k[x_1, \ldots, x_n]$ be the polynomial
ring in $n$ variables over a field $k$, and let $M$ be a monomial ideal in
$S$ minimally generated by monomials $m_1, \ldots, m_r$. In this paper we
will try to construct a minimal free resolution of $M$, which is a problem
first posed by Kaplansky in the early 1960s. Note that sometimes when we say  
a minimal free resolution of $M$, what we actually mean is  a minimal free 
resolution of $S/M$.

There are several methods for computing the multigraded Betti numbers of $M$.
These methods often involve computing the homologies of some simplicial
complexes. However, not much is known about finding the differential maps
in a minimal free resolution of $M$. For some specific classes of monomial
ideals, one can use their combinatorial structures to get explicit formulas
for their minimal free resolutions. For example, we have the Eliahou-Kervaire
resolutions \cite{B:EK} for stable ideals. But it seems impossible to get
such an explicit formula for all monomial ideals. 

Let $\mathbf{F}$ be a minimal free resolution of $M$. In \cite{B:PV} Peeva and
Velasco introduce the concept of the frame of $\mathbf{F}$, which is a complex
of $k$-vector spaces and has the same amount of information as $\mathbf{F}$.
So the problem of constructing a minimal free resolution of $M$ is reduced to
the problem of building the frame of a minimal free resolution.

For example, if $M$ has a simplicial resolution $\mathbf{F}$, then there exists
a simplicial complex $\Delta$ such that the simplicial chain complex
$C(\Delta; k)$ is the frame of $\mathbf{F}$, or equivalently, $\mathbf{F}$
is the $M$-homogenization of $C(\Delta; k)$. In this case, we say that
$\mathbf{F}$ is supported on $\Delta$. Similarly, we can consider
minimal free resolutions supported on cell complexes or CW-complexes. However,
in \cite{B:Ve} Velasco gives an example of a monomial ideal such that none of
its minimal free resolutions can be supported on a CW-complex. Hence,
CW-complexes do not provide all the frames we need. Moreover, except for some
special classes of monomial ideals, say, Scarf ideals, it seems hard to first
get a CW-complex and then use it to obtain a minimal free resolution. Also,
a cellular resolution often requires a special choice of basis for the minimal
free resolution, whereas many important concepts in mathematics do not depend
on the choice of basis.

In \cite{B:GPW} Gasharov, Peeva and Welker introduce the concept of the
lcm-lattice $L_M$ of $M$. They show that minimal free resolutions of $M$ are
determined by $L_M$. So one may wonder if it is possible to build the frame of
a minimal free resolution of $M$ from $L_M$. Some important work has been done
in this direction. For example, in \cite{B:Cl} Clark proves that if $M$ is
lattice-linear then the poset resolution constructed from $L_M$ is a minimal
free resolution of $S/M$; in \cite{B:CM1} and \cite{B:CM2} Clark and Mapes
prove that if $M$ is a rigid monomial ideal then the poset resolution
constructed from the Betti poset of $M$, which is a subposet of $L_M$, is a
minimal free resolution of $S/M$; and in \cite{B:Wo} Wood generalizes Clark's
result to Betti-linear monomial ideals.

In this paper we will use $L_M$ in a different way to construct minimal free
resolutions for \emph{all} monomial ideals. This method is originated from
the concept of nearly Scarf monomial ideals introduced by Peeva and Velasco in
\cite{B:PV}. For convenience, a minimal free resolution of $S/M$ constructed
by this method is called an atomic lattice resolution. We show that every
minimal free resolution of $S/M$ is an atomic lattice resolution.

A major feature of this paper is that any complex of multigraded free $S$-modules or
any complex of $k$-vector spaces is always equipped with a fixed basis. Let
$\mathbf{F}$ be a minimal free resolution of $S/M$, we introduce a new concept
called a Taylor basis of $\mathbf{F}$, which will be frequently used in this
paper. Let $\mathbf{T}$ be the Taylor resolution \cite{B:Ta} of $S/M$. Let $\Omega$ be
the simplex with the vertex set $\{1,2,\ldots,r\}$. Then the frame of $\mathbf{T}$
is the augmented chain complex $\widetilde{C}(\Omega; k)[-1]$ 
with a shift in homological degree and a basis consisting of all
the faces of $\Omega$. This basis can be viewed as a multigraded basis of
$\mathbf{T}$. Since $\mathbf{F}$ is a multigraded free submodule of
$\mathbf{T}$, we can fix a multigraded basis of $\mathbf{F}$. Dehomogenizing
these basis elements, we get a basis for the frame of $\mathbf{F}$. For
convenience, we call this basis a Taylor basis of $\mathbf{F}$. The elements
in a Taylor basis are chains in $\Omega$. If we apply the boundary map $d$
of $\Omega$ to the Taylor basis elements, then we get the frame of
$\mathbf{F}$. So the problem of building a frame of a minimal free resolution
is equivalent to the problem of finding a Taylor basis of a minimal free
resolution.

It is well-known that any two minimal free resolutions of $S/M$ are isomorphic,
because of which we often call \emph{a} minimal free resolution of $S/M$ by
\emph{the} minimal free resolution of $S/M$. In this paper we will distinguish
different minimal free resolutions of $S/M$, namely, we will rigorously define
when two minimal free resolutions of $S/M$ are called being equal to each other.
As a result, $S/M$ will have many different minimal free resolutions, which
correspond to many different frames and many different Taylor bases.

This paper is organized as follows. 

In Section 2 we prepare definitions, notations
and basic results which will be used in later sections. This section has four subsections.
Subsection 2.1  begins with the automorphism group of a multigraded free $S$-module.
Then we show that an isomorphism between two complexes of multigraded free
$S$-modules is equivalent to a change of basis in one of the complexes. The
definition of two complexes being equal to each other and the definition of a
subcomplex are a little different from the traditional definitions, because in
this paper a complex is always equipped with a fixed basis.
In Subsection 2.2, we introduce the concept of a Taylor basis of a minimal free
resolution $\mathbf{F}$ of $S/M$, which characterizes how $\mathbf{F}$ can be
embedded in the Taylor resolution $\mathbf{T}$. $\mathbf{F}$ may have many
different Taylor bases, each of which determines a submodule of $\mathbf{T}$. Such
a submodule is called a Taylor submodule for $M$.
Subsection 2.3 is about consecutive cancellations, which is a technique
introduced by Peeva. In this subsection, from the Taylor
resolution of $S/M$ we use consecutive cancellations to get a minimal free
resolution with a Taylor basis. The idea of consecutive cancellations will 
be used in the proofs of some theorems and propositions in Section 3.
Subsection 2.4 is about the lcm-lattice $L_M$. Many results in this paper benefit
from our new labeling of the elements in $L_M$. Each element $m$ in $L_M$ is
labeled by a subset $A_m$ of $\{1,\ldots,r\}$ such that
$A_m=\{1\leq i \leq r|m_i \mbox{ divides } m\}$; and then each $m\neq 1$ in
$L_M$ is associated with a simplicial complex $\Delta_m$ with facets
$A_{\beta_1}, \ldots, A_{\beta_t}$ where $\beta_1, \ldots, \beta_t$ are all the
elements in $L_M$ covered by $m$. By the crosscut theorem, $\Delta_m$ is
homotopic to the order complex of $(1,m)$. This fact is used in the proof of
Theorem 58.8 in \cite{B:Pe}, in the proof of Theorem 3.3 in \cite{B:Cl} and in
the proof of Theorem 2.5 in \cite{B:Wo}. We feel that in $L_M$ it is more
natural and much simpler to work with $\Delta_m$ than the order complex of
$(1,m)$. Some basic properties about $A_m$ are proved. In Theorem \ref{T:basis}  an interesting
and useful relation between a Taylor basis and homology groups
$\widetilde{H}(\Delta_m;k)$ is revealed.

Section 3 contains the main theory of this paper, which has three subsections.
In Subsection 3.1, we introduce a new concept called the exact closure of a
complex of $k$-vector spaces, namely, given any complex $\mathbb{U}$ over
$k$ we construct an exact complex $\mathbb{V}$ such that $\mathbb{U}$ is a
subcomplex of $\mathbb{V}$ and $\mathbb{V}$ is in a sense the smallest such
complex. Some propositions about exact closures are proved.
In Subsection 3.2, we use $L_M$ and exact closures to construct a minimal
free resolution of $S/M$ with a Taylor basis. Specifically, we build the
frame of a minimal free resolution of $S/M$, which is a generalization of
Theorem 6.1 in \cite{B:PV} about nearly Scarf monomial ideals. Because
given any atomic lattice $L$ we can use this method to associate $L$ with
an exact complex of $k$-vector spaces, we call the minimal free resolution
constructed this way an atomic lattice resolution of $S/M$. Conversely, we
prove that every minimal free resolution of $S/M$ is an atomic lattice
resolution. Theorem \ref{T:main1} and Theorem \ref{T:main2} are the main 
theorems of this paper, and the results in Subsection 3.2 are called the 
atomic lattice resolution theory.
Subsection 3.3 discusses some applications of the atomic lattice resolution
theory. Besides the computation of minimal free resolutions, this theory
can be applied to some theoretical problems. First, converse to Theorem 
\ref{T:basis}, Theorem \ref{T:cbasis} gives a simple
criteria for a set of chains to form a Taylor basis, which is very handy
for finding minimal free resolutions in examples. Next, in Construction 
\ref{C:basis3} we develop a method to find a Taylor basis of a given 
minimal free resolution of $S/M$. Then we show that  the intersection 
of all Taylor submodules for $M$ is the submodule
of the Taylor resolution $\mathbf{T}$ generated by the faces corresponding
to the Scarf multidegrees, which was also proved by Mermin in \cite{B:Me}. 
After that, we prove that the Betti poset of $M$ determines minimal 
free resolutions of $S/M$, which is also proved in \cite{B:CM2} and 
\cite{B:TV}.  Finally, we use $L_M$ to get a bound for the projective 
dimension of $S/M$.

Section 4 can also be viewed as an application of the atomic lattice
resolution theory. This section has two similar subsections. The idea in both
subsections is to construct differential maps in a minimal free resolution by
using homology groups $\tilde{H}(\Delta_m,k)$ and the connecting
homomorphism in the Mayer-Vietoris sequence.
In Subsection 4.1, we use our language to rewrite the theory of poset
resolutions. And then the relations among rigid monomial ideals in 
\cite{B:CM1}, homologically monotonic monomial ideals in \cite{B:FMS},
lattice-linear monomial ideals in \cite{B:Cl} and Betti-linear monomial
ideals in \cite{B:Wo} are studied. After that, we give a new proof of
a rigid monomial ideal having a poset resolution, and this proof is
generalized to show that a monomial ideal $M$ has a poset resolution
if and only if $M$ is Betti-linear. Our proofs are different from those
in \cite{B:CM2}, \cite{B:Cl} and \cite{B:Wo}.
In Subsection 4.2, we give a construction (Definition \ref{D:RLM} ) which is very similar to the
poset construction, and we introduce a new concept called the maximal
approximation of a minimal free resolution of $S/M$. In Theorem \ref{T:approximation} we show
that any maximal approximation of a minimal free resolution can be
obtained by Definition \ref{D:RLM}, and conversely, any sequence of 
multigraded free $S$-modules and multigraded homomorphisms 
obtained by Definition \ref{D:RLM} is the maximal approximation of a minimal
free resolution of $S/M$. In other words, we have an approximation formula for
minimal free resolutions of \emph{all} monomial ideals. Then similar to Betti-linear monomial
ideals, we introduce a new class of momomial ideals called homology-linear monomial ideals. 
After that we introduce the class of strongly homology-linear monomial ideals and 
the class of nearly homologically monotonic monomial ideals. 
The latter is a generalization of homologically monotonic momomial 
ideals and nearly Scarf monomial ideals. The relations among these classes 
of monomial ideals are studied. 
\vspace{0.1cm}

\noindent \textbf{Acknowledgments}. The author learned monomial resolutions 
from Irena Peeva, and he wants to thank Irena 
for her help and encouragement during his graduate study at Cornell University.

\section{Preliminaries}
\noindent In this section we introduce definitions, notations and basic
results which will be used in later sections. This section has four
subsections.

\subsection{Complexes of  Multigraded Free $S$-modules}
In this paper the polynomial ring $S$ is multigraded. Let
$\mathbb{N}=\{0,1,2,\ldots\}$. For any $\mathbf{a}=(a_1,\ldots,a_n)\in \mathbb{N}^n$,
let $x^{\mathbf{a}}=x_1^{a_1}\cdots x_n^{a_n}$. The multigraded elements
in $S$ are scalar multiples of monomials. Let $\lambda\neq 0 \in k$ then
$\lambda x^{\mathbf{a}}$ is multigraded with multidegree $x^{\mathbf{a}}$,
written as $\mbox{mdeg}(\lambda x^{\mathbf{a}})=x^{\mathbf{a}}$. The free
$S$-module generated by one element $f$ in multidegree $x^{\mathbf{a}}$ is
denoted by $S(-x^{\mathbf{a}})$. Then for any $\lambda\neq 0 \in k$ and
$\mathbf{b}\in \mathbb{N}^n$, $\lambda x^{\mathbf{b}}f$ is multigraded in
$S(-x^{\mathbf{a}})$ with $\mbox{mdeg}(\lambda x^{\mathbf{b}}f)=x^{\mathbf{a}+\mathbf{b}}$.

Let $F=S(-x^{\mathbf{a}_1})\oplus \cdots \oplus S(-x^{\mathbf{a}_p})$ be a
multigraded free $S$-module. For any $1\leq i \leq p$ let $f_i$ be a basis
element of $S(-x^{\mathbf{a}_i})$, then $f_1, \ldots, f_p$ is a multigraded
basis of $F$. Let $f=u_1f_{i_1}+\cdots+u_qf_{i_q}\in F$ with $1\leq i_1<\cdots<i_q\leq p$
and $u_1,\ldots, u_q$ being nozero elements in $S$, then $f$ is called
\emph{multigraded} if $u_1,\ldots, u_q$ are multigraded elements $S$ and
$\mbox{mdeg}(u_1f_{i_1})=\cdots=\mbox{mdeg}(u_qf_{i_q})$. In this case,
we write $\mbox{mdeg}(f)=\mbox{mdeg}(u_1f_{i_1})$.

Let $\sigma:F \to F$ be a multigraded isomorphism, then $\sigma(f_1),\ldots, \sigma(f_p)$
is a multigraded basis of $F$ with
$\mbox{mdeg}(\sigma(f_1))=x^{\mathbf{a}_1}, \ldots, \mbox{mdeg}(\sigma(f_p))=x^{\mathbf{a}_p}$.
Let $A=(a_{ij})_{p\times p}$ be the matrix of $\sigma$ with respect to the basis $f_1,\ldots, f_p$ of $F$,
then for any $1\leq j \leq p$, we have that $\sigma(f_j)=\displaystyle\sum_{i=1}^p a_{ij}f_i$;
or in abbreviation, we can write them as $\sigma(f_1,\ldots,f_p)=(\sigma(f_1),\ldots,\sigma(f_p))=(f_1,\ldots,f_p)A$.
It is easy to see that $A$ is characterized by the following three properties:
   \begin{itemize}
   \item[(i)] $A$ is invertible, i.e., $\det(A)\neq 0\in k$;
   \item[(ii)] If $x^{\mathbf{a}_i}$ does not divide $x^{\mathbf{a}_j}$, then $a_{ij}=0$;
   \item[(iii)] If $x^{\mathbf{a}_i}$ divides $x^{\mathbf{a}_j}$, then either $a_{ij}=0$
   or $a_{ij}$ is a nonzero multigraded element in $S$ with $\mbox{mdeg}(a_{ij})=x^{\mathbf{a}_j-\mathbf{a}_i}$.
   \end{itemize}
By the above conditions, using the Laplace expansion along the columns of $A$,
we see that in every column of $A$ there is a nonzero entry belonging to $k$.
Let $\mbox{Aut}(F)$ be the set of all such $A$. Because the composition of
two multigraded isomorphisms is a multigraded isomorphism, it is easy to see
that $\mbox{Aut}(F)$ is a group under matrix multiplication. We call
$\mbox{Aut}(F)$ the \emph{automorphism group} of $F$. Note that $\mbox{Aut}(F)$ does
not depend on the basis of $F$. It only depends on the ordered multidegrees
$x^{\mathbf{a}_1},\ldots, x^{\mathbf{a}_p}$. 
Here the order of the multidegrees matters because for example, 
$\mbox{Aut}(S(-x_1)\oplus S(-x_1^2)) \neq \mbox{Aut}(S(-x_1^2)\oplus S(-x_1))$. 
If for any $i\neq j$, $x^{\mathbf{a}_i}$
does not divide $x^{\mathbf{a}_j}$, then $\mbox{Aut}(F)$ is the multiplicative
group of invertible diagonal matrices.

Let $g_1,\ldots, g_p$ be another multigraded basis of $F$ with multidegrees
$x^{\mathbf{a}_1},\ldots, x^{\mathbf{a}_p}$, respectively. For any $1\leq j \leq p$,
let $g_j=\displaystyle\sum_{i=1}^p b_{ij}f_i$, where $b_{ij}\in S$;
or in abbreviation, we can write $(g_1,\ldots,g_p)=(f_1,\ldots,f_p)B$,
where $B=(b_{ij})_{p\times p}$.
Then we have a multigraded isomorphism $\rho:F\to F$ such that
$\rho(f_1)=g_1,\ldots,\rho(f_p)=g_p$, and the matrix of $\rho$ under basis
$f_1,\ldots,f_p$ is $B\in \mbox{Aut}(F)$.
Hence, any change of basis in $F$ is corresponding to an automorphism of $F$
and is corresponding to a matrix in $\mbox{Aut}(F)$.
Therefore, if we fix a basis $f_1,\ldots,f_p$ of $F$ then there is a bijection
between $\mbox{Aut}(F)$ and the set of all multigraded bases of $F$ with
multidegrees $x^{\mathbf{a}_1},\ldots, x^{\mathbf{a}_p}$; in other words,
we can use $\mbox{Aut}(F)$ to get all the multigraded bases of $F$.

Let $t_1,\ldots, t_p$ be a permutation of $1, \ldots, p$, then $f_{t_1}, \ldots, f_{t_p}$ 
may not be a basis of $F$. However,  $f_{t_1}, \ldots, f_{t_p}$ is a basis of 
$\widehat{F}=S(-x^{\mathbf{a}_{t_1}})\oplus \cdots \oplus S(-x^{\mathbf{a}_{t_p}})$. 
Hence,  in the rest of the paper, when we say a permutation of the basis elements 
of $F$, $F$ is simultaneously changed to $\widehat{F}$. 

Let $\phi:F=S(-x^{\mathbf{a}_1})\oplus \cdots \oplus S(-x^{\mathbf{a}_p})
\to G=S(-x^{\mathbf{b}_1})\oplus \cdots \oplus S(-x^{\mathbf{b}_q})$
be a multigraded homomorphism.
Let $f_1,\ldots, f_p$ and $e_1,\ldots, e_p$ be two multigraded bases of $F$
such that $(e_1,\ldots,e_p)=(f_1,\ldots,f_p)U$ with $U\in \mbox{Aut}(F)$.
Let $g_1,\ldots, g_q$ and $h_1,\ldots, h_q$ be two multigraded bases of $G$
such that $(h_1,\ldots,h_q)=(g_1,\ldots,g_q)V$ with $V\in \mbox{Aut}(G)$.
Let $A$ be the $q\times p$ matrix of $\phi$ under the basis $f_1,\ldots,f_p$
of $F$ and the basis $g_1,\ldots, g_q$ of $G$, i.e., $\phi(f_1,\ldots,f_p)=(g_1,\ldots,g_q)A$.
Similarly, let $\phi(e_1,\ldots,e_p)=(h_1,\ldots,h_q)B$.
Then similar to the results in linear algebra, it is easy to prove that $B=V^{-1}AU$.

Let $\sigma_1:F\to F$ be the isomorphism defined by $\sigma_1(f_1,\ldots,f_p)=(e_1,\ldots,e_p)$.
Let $\sigma_2:G\to G$ be the isomorphism defined by $\sigma_2(g_1,\ldots,g_q)=(h_1,\ldots,h_q)$.
Let $\psi=\sigma_2^{-1}\phi\sigma_1:F\to G$. Then $\psi$ is a
multigraded homomorphism and it is easy to prove that
$\psi(f_1,\ldots,f_p)=(g_1,\ldots,g_q)V^{-1}AU=(g_1,\ldots,g_q)B$. Hence,
the change of bases in $F$ and $G$ is equivalent to the following
commutative diagram:
\[
\begin{tikzcd}[column sep =2cm]
  F \arrow[r,"\phi"] & G   \\
	F \arrow[u,"\sigma_1"] \arrow[r,"\psi=\sigma_2^{-1}\psi \sigma_1"] & G \arrow[u,"\sigma_2"]
\end{tikzcd}
\]
And under the basis $f_1,\ldots,f_p$ of $F$ and the basis $g_1,\ldots,g_q$ of $G$, 
this diagram can be written as:
\[
\begin{tikzcd}[column sep =2cm]
  F \arrow[r,"A"] & G   \\
	F \arrow[u,"U"] \arrow[r,"B=V^{-1}AU"] & G \arrow[u,"V"]
\end{tikzcd}
\]
Note that in this paper when we write $F \xrightarrow{\ A\  }  G$, we mean 
that there exists a multigraded homomorphism $\phi: F\to G$ such that under 
a basis of $F$ and a basis of $G$ the matrix of $\phi$ is $A$.

In this paper a complex $(\mathbf{F},d)$ of multigraded free $S$-modules is 
a finite sequence of multigraded homomorphisms:
\[
\begin{tikzcd}
\mathbf{F}:\  0 \arrow{r} & F_l \arrow[r,"d_l"] &F_{l-1} \arrow[r]& \cdots \arrow[r] 
& F_2 \arrow[r,"d_2"]  & F_1 \arrow[r,"d_1"] & F_0 \arrow[r] &0
\end{tikzcd}
\]
with a fixed multigraded basis for each $F_i$.
Let $A_i$ be the matrix of $d_i$ under the basis of $F_i$ and $F_{i-1}$, then 
$\mathbf{F}$ is often written as
\[
\begin{tikzcd}
\mathbf{F}:\  0 \arrow{r} & F_l \arrow[r,"A_l"] &F_{l-1} \arrow[r]& \cdots \arrow[r] 
& F_2 \arrow[r,"A_2"]  & F_1 \arrow[r,"A_1"] & F_0 \arrow[r] &0.
\end{tikzcd}
\]

\begin{definition}\label{D:isomorphism}
Let $(\mathbf{F},d)$ and $(\mathbf{G},\partial)$ be two complexes of multigraded 
free $S$-modules. $\varphi: \mathbf{F} \to \mathbf{G}$ is called an \emph{isomorphism} if 
for each $i$, $\varphi_i: F_i \to  G_i$  is a multigraded isomorphism and 
$\varphi_{i-1}d_i=\partial_{i}\varphi_{i}$.  If there exist an isomorphism between 
$\mathbf{F}$ and $\mathbf{G}$ then we say that $\mathbf{F}$ and $\mathbf{G}$ 
are \emph{isomorphic} and we write $\mathbf{F} \cong \mathbf{G}$.
\end{definition}

\begin{definition}\label{D:equal}
Let $(\mathbf{F},d)$ and $(\mathbf{G},\partial)$ be two complexes of multigraded 
free $S$-modules. If there exists an isomorphism $\varphi: \mathbf{F} \to \mathbf{G}$ 
such that $\varphi$ induces a bijection between the fixed bases of $\mathbf{F}$ and 
$\mathbf{G}$, then we say that $\mathbf{F}$ is equal to $\mathbf{G}$ and we 
write $\mathbf{F}=\mathbf{G}$.
\end{definition}

\begin{remark}\label{R:equal}
  \begin{itemize}
    \item[(1)] Let $F=S(-x_1) \oplus S(-x_2)$ and $G=S(-x_2) \oplus S(-x_1)$, then $F$ 
                    and $G$ are isomorphic, but $F\neq G$. However, by Definition \ref{D:equal} 
                   the complexes $0\to F \to 0$ and $0\to G \to 0$ are equal.  
   \item[(2)] In Definition \ref{D:equal} let $A_i$ be the matrix of $d_i$ under the basis of 
                   $\mathbf{F}$. Then it is easy to see that if $\mathbf{F}=\mathbf{G}$ then 
                   there exists a permutation of the basis elements of $\mathbf{G}$ such that 
                   under the new basis the  matrix of $\partial_i$ is $A_i$;  and conversely, if 
                   $F_i=G_i$ for all $i$ and the matrix of $\partial_i$ is $A_i$, then $\mathbf{F}=\mathbf{G}$.
  \item[(3)] Let $(\mathbf{F}, d)$ and $(\mathbf{G}, \partial)$ be two sequences of multigraded 
                  free $S$-modules and multigraded homomorphisms. Then  we have similar 
                  definitions for $\mathbf{F}\cong \mathbf{G}$ and $\mathbf{F}=\mathbf{G}$. 
  \end{itemize}
\end{remark}

Let $(\mathbf{F},d)$ be a complex of multigraded free $S$-modules and $A_i$ the 
matrix of $d_i$ under the basis of $\mathbf{F}$. 
Let $\sigma_i: F_i \to F_i$ be a change of basis isomorphism corresponding to 
$U_i\in \mbox{Aut}(F_i)$. Then under the new basis, the matrix of $d_i$ is 
$U_{i-1}^{-1}A_iU_i$. The new complex with the new basis is denoted 
by $\sigma(\mathbf{F})$, where $\sigma$ is the collection of isomorphisms 
$\sigma_i$ of $F_i$. We call $\sigma$ a \emph{change of basis map} of $\mathbf{F}$. 
Note that $\mathbf{F}$ and $\sigma(\mathbf{F})$ are isomorphic. Indeed, 
we have the following commutative diagram:
\[
 \begin{tikzcd}
 \mathbf{F}:\  \cdots \arrow{r} & F_i \arrow[r,"A_i"] &[2em] F_{i-1} \arrow[r]& \cdots \arrow[r] 
 & F_1 \arrow[r,"A_1"] &[2em]  F_0 \arrow[r] &0 \\
 \sigma(\mathbf{F}):\  \cdots \arrow{r} & F_i \arrow[r,"U_{i-1}^{-1}A_iU_i"] \arrow[u,"U_i"] 
 &F_{i-1} \arrow[r] \arrow[u,"U_{i-1}"]& \cdots \arrow[r] 
 & F_1 \arrow[r,"U_0^{-1}A_1U_1"]\arrow[u,"U_1"] & F_0 \arrow[r] \arrow[u,"U_0"]&0,
 \end{tikzcd}
\]
where $\mathbf{F}$ has the old basis and $\sigma(\mathbf{F})$ has the new basis. 

Conversely, if $\mathbf{F} \cong \mathbf{G}$ we will show that there exists a 
change of basis map $\sigma$ of $\mathbf{F}$ such that $\sigma(\mathbf{F})=\mathbf{G}$.
Indeed, let $\varphi: \mathbf{F}\to \mathbf{G}$ be an isomorphism, then we have the 
following commutative diagram:
\[
 \begin{tikzcd}
 \mathbf{F}:\  \cdots \arrow{r} 
 & F_i \arrow[r,"d_i"] \arrow[d,"\varphi_i"]
 & F_{i-1} \arrow[r]\arrow[d,"\varphi_{i-1}"]
& \cdots \arrow[r] 
 & F_1 \arrow[r,"d_1"]\arrow[d,"\varphi_1"] 
 & F_0 \arrow[r]\arrow[d,"\varphi_0"] &0
 \\
 \mathbf{G}:\  \cdots \arrow{r} 
 & G_i \arrow[r,"\partial_i"]
 &G_{i-1} \arrow[r]
 & \cdots \arrow[r] 
 & G_1 \arrow[r,"\partial_1"]
 & G_0 \arrow[r] &0.
 \end{tikzcd}
\]
By using a permutation of the basis elements of $\mathbf{G}$, without the 
loss of generality we can assume that $F_i=G_i$ for all $i$. Let $U_i$ be the 
matrix of $\varphi_i^{-1}$ under the basis of $F_i$ and the basis of $G_i$. It is 
easy to see that $U_i \in \mbox{Aut}(F_i)$.
Let $A_i$ be the matrix of $d_i$ under the basis of $\mathbf{F}$ and $B_i$ 
the basis of $\partial_i$ under the basis of $\mathbf{G}$, then from the 
previous commutative diagram, we get the following commutative diagram: 
\[
 \begin{tikzcd}
 \mathbf{F}:\  \cdots \arrow{r} & F_i \arrow[r,"A_i"] & F_{i-1} \arrow[r]& \cdots \arrow[r] 
 & F_1 \arrow[r,"A_1"] &  F_0 \arrow[r] &0 \\
 \mathbf{G}:\  \cdots \arrow{r} & G_i \arrow[r,"B_i"] \arrow[u,"U_i"] 
 &G_{i-1} \arrow[r] \arrow[u,"U_{i-1}"]& \cdots \arrow[r] 
 & G_1 \arrow[r,"B_1"]\arrow[u,"U_1"] & G_0 \arrow[r] \arrow[u,"U_0"]&0,
 \end{tikzcd}
\]
which implies that $B_i=U_{i-1}^{-1}A_iU_i$. 
Let $\sigma$ be the change of basis map of $\mathbf{F}$ corresponding 
to all the $U_i\in \mbox{Aut}(F_i)$, then by Remark \ref{R:equal} (2) we 
see that $\sigma(\mathbf{F})=\mathbf{G}$.

Let $\mathbf{F}$ and $\mathbf{G}$ be two minimal free resolutions of 
$S/M$. It is well-known that $\mathbf{F}\cong \mathbf{G}$. Hence, 
there exists a change of basis map $\sigma$ of $\mathbf{F}$ such that 
$\sigma(\mathbf{F})=\mathbf{G}$. So given any minimal free resolution 
of $S/M$, we can use change of basis to get all the minimal free resolutions 
of $S/M$. 

\begin{definition}\label{D:subcomplex}
Let $(\mathbf{F},d)$ be a complex of multigraded free $S$-modules with a 
fixed basis $f_1,\ldots, f_p$. Let $(\mathbf{G},\partial)$ be a complex 
of multigraded free $S$-modules. If there exist $1\leq i_1 <\cdots <i_q\leq p$ 
such that the free submodule $\mathbf{H}$ of $\mathbf{F}$ generated 
by $f_{i_1}, \ldots, f_{i_q}$ satisfies the following two conditions:
\begin{itemize}
 \item[(1)] $d(\mathbf{H})\subseteq \mathbf{H}$, i.e., $(\mathbf{H},d)$ is 
                 a complex of multigraded free $S$-modules;
 \item[(2)] $(\mathbf{H},d)=(\mathbf{G},\partial)$;
\end{itemize}
then we call $(\mathbf{G},\partial)$ a \emph{subcomplex} of 
$(\mathbf{F},d)$, and we write 
$(\mathbf{G},\partial)=(\mathbf{F}|_{\{f_{i_1},\ldots,f_{i_q}\}},d)$.
\end{definition}

\begin{remark}\label{R:subcomplex}
This definition of a subcomplex is different from the usual definition of 
 a subcomplex (for example, Definition 3.5 in \cite{B:Pe}). The reason is 
 that in this paper any complex of multigraded free $S$-modules is 
 assumed to have a fixed basis. 
\end{remark}

\begin{remark}\label{R:kcomplex}
In this paper any complex of $k$-vector spaces is also assumed to be finite and have 
a fixed basis. Then for complexes of $k$-vector spaces, we have definitions 
similar to those in Definitions \ref{D:isomorphism}, \ref{D:equal} and \ref{D:subcomplex}.
The results about change of basis also hold for complexes of $k$-vector spaces. 
\end{remark}

\subsection{Taylor Bases}

In this paper we will frequently use the notations of the frame of an $M$-complex 
and the $M$-homogenization of an $r$-frame, which are defined in 
Section 3 of \cite{B:PV} or in Section 55 of \cite{B:Pe}. 
Also, let $\mathbf{F}$ be a minimal free resolution of $S/M$; sometimes, for 
convenience, we write the resolution as $\mathbf{F} \to 0$. 

\begin{definition}\label{D:multidegree}
Let $\Omega$ be the simplex with the vertex set $\{1,\ldots,r\}$, where the 
vertex $i$ is corresponding to the minimal monomial generator $m_i$ of $M$.  For any 
$A\subseteq \{1,\ldots,r\}$,  we define the \emph{multidegree} of $A$ as 
$\mbox{mdeg}(A)=\mbox{lcm}(m_i | i\in A)$.  
Note that $\mbox{mdeg}(\emptyset)=1$.
Let $c=\lambda_1c_1+\cdots+\lambda_tc_t$ be a chain in $\Omega$, 
where $\lambda_1,\ldots, \lambda_t$ are nonzero elements in $k$ and 
$c_1,\ldots, c_t$ are some different faces of the same dimension in $\Omega$, then we define  
the \emph{multidegree} of $c$ as 
$\mbox{mdeg}(c)=\mbox{lcm}(\mbox{mdeg}(c_i) |1\leq i \leq t)$, 
and we define the \emph{support} of $c$ as 
$\mbox{supp}(c)=c_1\cup \cdots \cup c_t$.  
Let $f=\lambda_1x^{\mathbf{a}_1}c_1+\cdots+\lambda_tx^{\mathbf{a}_t}c_t$, 
where $x^{\mathbf{a}_1},\ldots, x^{\mathbf{a}_t}$ are some monomials 
in $S$. If there exists a monomial $m\in S$ such that 
$x^{\mathbf{a}_1}\mbox{mdeg}(c_1)=\cdots=x^{\mathbf{a}_t}\mbox{mdeg}(c_t)=m$, 
then we say that $f$ is \emph{multigraded} with multidegree $m$, written as $\mbox{mdeg}(f)=m$. 
By setting $x_1=\cdots=x_n=1$ in $f$, we get the chain $c$ back and $c$ is 
called the \emph{dehomogenization} of $f$. 
Let 
\[
g=\lambda_1\frac{\mbox{mdeg}(c)}{\mbox{mdeg}(c_1)}c_1+
    \cdots+\lambda_t\frac{\mbox{mdeg}(c)}{\mbox{mdeg}(c_t)}c_t,
\]
then $g$ is multigraded with multidegree $\mbox{mdeg}(c)$. 
We call $g$ the \emph{homogenization} of $c$ and write 
$g=\hbar(c)$. Note that the dehomogenization 
of $g$ is $c$.
By setting $x_1=\cdots=x_n=0$ in $f$, we get a chain which is called 
the \emph{initial part} of $f$, denoted by $\mbox{in}(f)$. 
$\mbox{in}(g)$ is called the \emph{initial part} of $c$, denoted by $\mbox{in}(c)$.
If $\mbox{in}(c)\neq 0$ then we call $c$ a \emph{Taylor chain}. 
\end{definition}

Let $\mathbf{T}$ be the Taylor resolution of $S/M$, then the set of 
faces of $\Omega$, including $\emptyset$, can be viewed as a multigraded 
basis of $\mathbf{T}$. 

\begin{definition}\label{D:Taylorbasis}
Let $\mathbf{F}$ be a minimal free resolution of $S/M$. It is well-known that 
there exists a trivial complex $\mathcal{E}$ such that 
$\mathbf{T}\cong \mathbf{F}\oplus \mathcal{E}$.
Then there exists a change of basis map $\sigma$ such that 
$\sigma(\mathbf{T})= \mathbf{F}\oplus \mathcal{E}$. 
Let $f_1, \ldots, f_p$ be the multigraded basis of $\sigma(\mathbf{T})$. 
Since $\mathbf{F}$ is a subcomplex of $\sigma(\mathbf{T})$, it follows 
that there exist $1\leq i_1<\cdots <i_q \leq p$ such that 
$\mathbf{F}=\sigma(\mathbf{T})|_{\{f_{i_1},\ldots,f_{i_q}\}}$. 
Let $g_1,\ldots,g_p$ be the dehomogenizations of  $f_1,\ldots,f_p$, respectively.  
Then $g_{i_1},\ldots,g_{i_q}$ is called a \emph{Taylor basis} of $\mathbf{F}$.
Let $N$ be the multigraded free $S$-module generated by 
$f_{i_1},\ldots,f_{i_q}$, then $N$ is a submodule of $\mathbf{T}$.
We call $N$ a \emph{Taylor submodule} for $M$. 
\end{definition}

\begin{remark}\label{R:TBin}
Let $f$ be a multigraded basis element of $\sigma(\mathbf{T})$ and 
let $g$ be the dehomogenization of $f$. By the discussion about the 
automorphism group of a multigraded free $S$-module at the beginning 
of Subsection 2.1, we see that $\mbox{in}(f)\neq 0$, which implies 
that $f$ is the homogenization of $g$.  Hence, 
$\mbox{in}(f)=\mbox{in}(g)$ and $g$ is a Taylor chain.
\end{remark}

\begin{remark}\label{R:TBuse}
Let $(\mathbb{W},d)$ be the frame of $\sigma(\mathbf{T})$ where $d$ 
is the boundary map of $\Omega$, then $g_1, \ldots, g_p$ is a basis of 
$\mathbb{W}$.  Let $(\mathbb{V},d)$ be the frame of $\mathbf{F}$, 
then $g_{i_1},\ldots, g_{i_q}$ is a basis of $\mathbb{V}$.  From the 
Taylor basis $g_{i_1},\ldots, g_{i_q}$, one can use the boundary map 
$d$ to obtain $\mathbb{V}$, and then the $M$-homogenization of 
$\mathbb{V}$ is $\mathbf{F}$. So, to construct a minimal free resolution 
of $S/M$, it is equivalent to obtain a Taylor basis. Note that because of 
Definition \ref{D:equal}, we do not need to worry about how the elements 
in a Taylor basis are ordered, and a Taylor basis determines a unique 
minimal free resolution. 
\end{remark}

\begin{remark}\label{R:TBmany}
Given a Taylor submodule $N$ for $M$ with a basis $f_{i_1},\ldots,f_{i_q}$ 
as in Definition \ref{D:Taylorbasis}, if we think of $N$ as a complex of 
multigraded free $S$-modules, then $N$ is a minimal free resolution of $S/M$;
indeed, $N=\mathbf{F}$.
Hence, by using change of basis in each homological degree of $N$, we can 
get \emph{every} minimal free resolutions of $S/M$. 
On the other hand, as is shown by the next example, given a minimal 
free resolution $\mathbf{F}$ of $S/M$, $\mathbf{F}$ may have many 
different Taylor bases, and consequently, there may be many different 
Taylor submodules for $M$. Since every Taylor submodule for $M$ 
with a suitable multigraded basis can give rise to $\mathbf{F}$, it follows 
that from $\mathbf{F}$ we can get \emph{all} the Taylor submodules 
for $M$. The set of all Taylor submodules for $M$ is denoted by $\Sigma_M$.
\end{remark}

\begin{example}\label{E:TBmany}
Let $S=k[x,y,z]$ and $M$ the monomial ideal generated by $m_1=xy, m_2
=xz, m_3=yz$. Then 
\[
\begin{tikzcd}[ampersand replacement=\&]
\mathbf{F}: \  0\to S(-xyz)^2  \ar[r, "{\begin{pmatrix} -z & -z \\ y & 0\\ 0 & x  \end{pmatrix}}"]
\& [2.2em] S(-xy) \oplus S(-xz) \oplus S(-yz) \arrow[r, "{\begin{pmatrix} xy & xz & yz \end{pmatrix}}"]
\& [3.3em] S
\end{tikzcd}
\]
is a minimal free resolution of $S/M$. 
For any $\lambda, \mu \in k$, let $B_{\lambda,\mu}$ be the set of 
chains consisting of $\emptyset$, $\{1\}$, $\{2\}$, $\{3\}$, 
$\lambda\{1,2\}+(1-\lambda)\{1,3\}-(1-\lambda)\{2,3\}$, 
$\mu \{1,3\}+(1-\mu)\{1,2\}+(1-\mu)\{2,3\}$. 
Then $B_{\lambda,\mu}$  is a Taylor basis of $\mathbf{F}$, 
and every Taylor basis of $\mathbf{F}$ is some $B_{\lambda,\mu}$. 
Let $N_{\lambda,\mu}$ be the multigraded free submodule of the 
Taylor resolution $\mathbf{T}$ generated by the elements in $B_{\lambda,\mu}$, 
then $N_{\lambda,\mu}$ is a Taylor submodule for $M$ and 
$\Sigma_M=\{N_{\lambda,\mu}|\lambda,\mu \in k\}$.
Note that $T_2=S(-xyz)^3$ with basis $\{1,2\}$, $\{1,3\}$, $\{2,3\}$. 
Let $\mathbf{a}=(a_1,a_2,a_3), \mathbf{b}=(b_1,b_2,b_3)\in k^3$ 
such that $e_1=\{1,2\}-\{1,3\}+\{2,3\}, e_2=a_1\{1,2\}+a_2\{1,3\}+a_3\{2,3\}, 
e_3=b_1\{1,2\}+b_2\{1,3\}+b_3\{2,3\}$ is a basis of $T_2$. 
Let $G_{\mathbf{a},\mathbf{b}}\subset T_2$ be the multigraded free
$S$-module generated by $e_2, e_3$. 
Then by using consecutive cancellations in Subsection 2.3, 
it is easy to see that $T_0\oplus T_1\oplus G_{\mathbf{a}, \mathbf{b}}$ 
is a Taylor submodule for $M$. 
Let $\widetilde{\Sigma} \subseteq \Sigma_M$ be the set of all such 
Taylor submodules. Since
\[
\begin{vmatrix}
1     &-1     &1 \\
\lambda &1-\lambda & -1+\lambda \\
1-\mu & \mu & 1-\mu
\end{vmatrix}
=1 \neq 0,
\]
it follows that $N_{\lambda, \mu} \in \widetilde{\Sigma}$, which 
implies that $\widetilde{\Sigma}=\Sigma_M$. 
So every Taylor submodule for $M$ can be written as $T_0\oplus T_1\oplus G$, 
where $G=S(-xyz)^2 \subset T_2$ and $\{1,2\}-\{1,3\}+\{2,3\} \notin G$.
\end{example}

\begin{remark}\label{R:simplicial}
Let $\mathbf{F}$ be a minimal free resolution of $S/M$. It is easy to see that 
$\mathbf{F}$ is a simplicial resolution if and only if there exists a Taylor basis 
of $\mathbf{F}$ consisting of some faces of $\Omega$; and $S/M$ has a 
simplicial resolution if and only of there exists a Taylor submodule $N$ for $M$ 
such that $N$ has a basis consisting of some faces of $\Omega$. In the above 
example if we take $\lambda=\mu=1$ then $\mathbf{F}$ is simplicial. 
\end{remark}

\subsection{Consecutive Cancellations}
Consecutive cancellation is a technique introduced by Peeva (for example, 
see Section 7 in \cite{B:Pe}) for
removing short trivial complexes from a nonminimal free resolution. 

Given a graded module $N$ over $S$, building a free resolution of $N$ 
over $S$ consists of repeatedly solving systems of polynomial equations, 
which is not easy and the computation often involves the Gr\"{o}bner 
basis theory. However, if a free resolution of $N$ is given, then we can 
use consecutive cancellations to get a minimal free resolution of $N$. 

Since any monomial ideal $M$ has the Taylor resolution, in this subsection we 
will discuss how to use consecutive cancellations to get a minimal free 
resolution of $S/M$ with a Taylor basis.  

\begin{proposition}\label{P:CCbasis}
Let $F_1, F_2, F_3, F_4$ be multigraded free $S$-modules with bases 
$h_1,\ldots, h_t$; $g_1, \ldots, g_q$; $f_1, \ldots, f_p$; $e_1, \ldots, e_s$; 
respectively. Let
\[
\begin{tikzcd}
F_4 \arrow[r,"d_4"] & F_3 \arrow[r, "d_3"] & F_2 \arrow[r, "d_2"] & F_1
\end{tikzcd}
\]
be a complex of multigraded free $S$-modules. 
Let $C, A, B$ be the matrices of $d_2, d_3, d_4$, respectively, under 
the given bases.  Assume that
\[
A=
\begin{pmatrix}
A_1 & \beta \\
\alpha & a
\end{pmatrix}_{q \times p}
\ \ \ 
B=
\begin{pmatrix}
B_1 \\
\gamma
\end{pmatrix}_{p \times s}
\ \ \ 
C=
\begin{pmatrix}
C_1 & \eta
\end{pmatrix}_{t \times q}
\]
where $a\neq 0 \in k$; $\alpha=(a_1, \ldots, a_{p-1})$ with $a_i=0$ 
or $a_i$ being a scalar multiple of some monomial in $S$, and similarly, 
$\beta$ is a column vector with entries $b_1, \ldots, b_{q-1}$; 
$\gamma$ is a row vector and $\eta$ is a column vector. Let 
\begin{align*}
\widetilde{f_1} &=f_1-a^{-1}a_1f_p, \\
 & \vdots \\
\widetilde{f_{p-1}} &=f_{p-1} -a^{-1}a_{p-1}f_p, \\
\widetilde{g_q} &=b_1g_1+\cdots +b_{q-1}g_{q-1}+ag_q.
\end{align*}
Let $\widetilde{F_2}$ be the multigraded free submodule of $F_2$ 
generated by $g_1, \ldots, g_{q-1}$; let $\widetilde{F_3}$ be the 
multigraded free submodule of $F_3$ generated by $\widetilde{f_1}, 
\ldots, \widetilde{f_{p-1}}$. Let $\mbox{mdeg}(f_p)=m$, then we 
have that $F_2=\widetilde{F_2}\oplus S(-m)$ and 
$F_3=\widetilde{F_3}\oplus S(-m)$.
And under the new bases 
$h_1,\ldots, h_t$; $g_1, \ldots, g_{q-1}, \widetilde{g_q}$; 
$\widetilde{f_1}, \ldots, \widetilde{f_{p-1}},  f_p$; $e_1, \ldots, e_s$,
the complex can be written as
\[
\begin{tikzcd}[ampersand replacement=\&]
F_4  \ar[r, "{\begin{pmatrix} B_1 \\  0  \end{pmatrix}}"]
\& [1.2em] F_3 \arrow[r, "{\begin{pmatrix} A_1-a^{-1}\beta\alpha & 0 \\ 0 & 1 \end{pmatrix}}"]
\& [3.5em] F_2  \ar[r, "{\begin{pmatrix} C_1 &  0  \end{pmatrix}}"]
\& [2.2em] F_1,
\end{tikzcd}
\]
which is equal to 
\[
\begin{tikzcd}[ampersand replacement=\&]
(F_4  \ar[r, " B_1 "]
\& \widetilde{F_3} \arrow[r, "A_1-a^{-1}\beta\alpha"]
\& [2.5em] \widetilde{F_2}  \ar[r, " C_1"]
\&  F_1)\oplus (0 \arrow[r]
\& [-1em] S(-m) \arrow[r,"1"] 
\&  S(-m) \arrow[r]
\&  [-1em] 0).
\end{tikzcd}
\]
\end{proposition}

\begin{proof}
Let $\widetilde{C}, \widetilde{A}, \widetilde{B}$ be the matrices of 
$d_2, d_3, d_4$, respectively, under the new bases. Let 
\[
U_2=\begin{pmatrix} E_{q-1} & \beta \\ 0 & a \end {pmatrix}
\ \ \ \ \ 
U_3=\begin{pmatrix} E_{p-1} & 0 \\ -a^{-1}\alpha & 1 \end{pmatrix},
\]
then $U_2 \in \mbox{Aut}(F_2), U_3 \in \mbox{Aut}(F_3)$ correspond 
to the change of bases in $F_2$ and $F_3$, so that by Subsection 2.1 
we have the following commutative diagram:
\[
 \begin{tikzcd}
  F_4 \arrow[r,"B"] &  F_3 \arrow[r, "A"]&  F_2 \arrow[r,"C"] &  F_1  \\
  F_4 \arrow[r,"\widetilde{B}"] \arrow[u,"E_s"] 
 &F_3 \arrow[r,"\widetilde{A}"] \arrow[u,"U_3"] 
 & F_2 \arrow[r,"\widetilde{C}"]\arrow[u,"U_2"] 
 & F_1  \arrow[u,"E_t"].
 \end{tikzcd}
\]
Hence, we have that
\begin{align*}
\widetilde{B} &=U_3^{-1}B=\begin{pmatrix} E_{p-1} & 0 \\ a^{-1}\alpha & 1 \end{pmatrix}
                                   \begin{pmatrix}B_1 \\ \gamma \end{pmatrix}
                             =\begin{pmatrix}B_1 \\ a^{-1}\alpha B_1+\gamma \end{pmatrix} \\
\widetilde{A} &=U_2^{-1}AU_3=\begin{pmatrix} E_{q-1} & -a^{-1}\beta \\ 0 & a^{-1} \end {pmatrix}
                                  \begin{pmatrix}A_1 & \beta \\ \alpha & a \end{pmatrix}
                                 \begin{pmatrix} E_{p-1} & 0 \\ -a^{-1}\alpha & 1 \end{pmatrix}
                               =\begin{pmatrix} A_1-a^{-1}\beta\alpha & 0 \\ 0 & 1 \end{pmatrix}\\
\widetilde{C} &=CU_2=\begin{pmatrix} C_1 & \eta \end{pmatrix} 
                                      \begin{pmatrix} E_{q-1} & \beta \\ 0 & a \end {pmatrix}
                                 =\begin{pmatrix} C_1 & C_1\beta+a\eta \end{pmatrix}.
\end{align*}
Since $CA=0$ and $AB=0$, it follows that $C_1\beta +a\eta =0$ 
and $\alpha B_1+a\gamma=0$, which implies that 
$\widetilde{B}=\begin{pmatrix} B_1 \\ 0 \end{pmatrix}$ and 
$\widetilde{C}=\begin{pmatrix} C_1 & 0 \end{pmatrix}$. 
\end{proof}

\begin{remark}\label{R:CCgeneral}
First, the matrix $A_1-a^{-1}\beta\alpha$ 
can be obtained from $ \left(\begin{smallmatrix}A_1 & \beta \\ \alpha & a \end{smallmatrix}\right)$ 
by using elementary row operations. 
Second, if a is not at the $q^{\mbox{th}}$ row or $p^{\mbox{th}}$ column, 
similar result holds. 
Third, if $F_1, F_2, F_3, F_4$ are graded free $S$-modules, similar result holds. 
Finally, let $R=S/I$ where $I$ is a graded ideal in $S$, then similar result holds 
for graded free $R$-modules. In particular, we have similar results for complexes 
of $k$-vector spaces. 
\end{remark}

\begin{definition}\label{D:CC}
As in Proposition \ref{P:CCbasis},  the process of changing the bases of $F_2$ 
and $F_3$, and then removing the short trivial complex 
$0\to S(-m) \to S(-m) \to 0 $ to obtain 
$F_4 \to \widetilde{F_3} \to \widetilde{F_2} \to F_1$ 
is called a \emph{consecutive cancellation} at the entry $a$, 
or called a \emph{consecutive cancellation} with respect to $f_p$ and $g_q$. 
\end{definition}

It is easy to see that $F_4 \to F_3 \to F_2 \to F_1$ 
and $F_4 \to \widetilde{F_3} \to \widetilde{F_2} \to F_1$  have the same 
homology. In particular, the former is exact if and only if the latter is exact.

\begin{theorem}\label{T:CCminimal}
Let $M$ be a monomial ideal minimally generated by $r$ monomials. Let $\Omega$ 
be  the simplex with the vertex set $\{1,\ldots, r\}$. Let $\mathbf{T}$ be the 
Taylor resolution of $S/M$ with basis consisting of the faces of $\Omega$. 
By using a series of consecutive cancellations we can get a minimal free resolution 
$\mathbf{F}$ of $S/M$. By dehomogenizing the basis elements of $\mathbf{F}$ 
we get a Taylor basis of $\mathbf{F}$. And there exists a trivial complex 
$\mathcal{E}$ such that $\mathbf{T} \cong \mathbf{F} \oplus \mathcal{E}$. 
\end{theorem}

\begin{proof}
If no entry in the matrices of $\mathbf{T}$ is a nonzero scalar in $k$, then 
$\mathbf{T}$ is a minimal free resolution of $S/M$ with the faces of $\Omega$ 
as a Taylor basis. 

Otherwise, pick any entry in any matrix of $\mathbf{T}$ which is a nonzero 
scalar and do a consecutive cancellation at that entry. After deleting a short 
trivial complex, we obtain a new complex $\mathbf{U}$. Note that 
$H_0(\mathbf{U})=H_0(\mathbf{T})=S/M$ and 
$H_i(\mathbf{U})=H_i(\mathbf{T})=0$  for $i\geq 1$, so that $\mathbf{U}$ 
is a multigraded free resolution of $S/M$ with a multigraded basis and 
$\mbox{rank}(\mathbf{U})=\mbox{rank}(\mathbf{T})-2$. If no entry in the 
matrices of $\mathbf{U}$ is a nonzero scalar, then $\mathbf{U}$ is a 
minimal free resolution of $S/M$ and the dehomogenization of the basis 
elements of $\mathbf{U}$ gives a Taylor basis of $\mathbf{U}$. 

Otherwise, pick any entry in any matrix of $\mathbf{U}$ which is a nonzero 
scalar and do a consecutive cancellation at that entry, then we get a 
multigraded free resolution $\mathbf{V}$ of $S/M$ with a multigraded basis and 
$\mbox{rank}(\mathbf{V})=\mbox{rank}(\mathbf{U})-2$. 

Since $\mbox{rank}(\mathbf{T})=2^r < \infty$, it follows that after doing a 
finite number of consecutive cancellations, we will get a minimal free resolution 
$\mathbf{F}$ of $S/M$.  Let $\mathcal{E}$ be 
the direct sum of the short trivial complexes associated to the consecutive 
cancellations, then $\mathcal{E}$ is a trivial complex and it is easy to see 
that $\mathbf{T} \cong \mathbf{F} \oplus \mathcal{E}$.
Thus, dehomogenizing the multigraded basis elements of 
$\mathbf{F}$, we get a Taylor basis of $\mathbf{F}$.
\end{proof}

\begin{remark}\label{R:CCnotall}
Not all minimal free resolutions of $S/M$ can be obtained by using consecutive 
cancellations. For example, let $S$ and $M$ be as in Example \ref{E:TBmany}, 
then the following minimal free resolution of $S/M$ 
\[
\begin{tikzcd}[ampersand replacement=\&]
\mathbf{F}: \  0\to S(-xyz)^2  \ar[r, "{\begin{pmatrix} -z & -z \\ y & -y\\ 0 & 2x  \end{pmatrix}}"]
\& [2.2em] S(-xy) \oplus S(-xz) \oplus S(-yz) \arrow[r, "{\begin{pmatrix} xy & xz & yz \end{pmatrix}}"]
\& [3.3em] S
\end{tikzcd}
\]
with its Taylor basis  $\emptyset$, $\{1\}$, $\{2\}$, $\{3\}$, 
$\{1,2\}$, $\{1,3\}+\{2,3\}$, can not be obtained by using
consecutive cancellations. 
\end{remark}

Similar to Theorem \ref{T:CCminimal}, in general, if $N$ is a graded 
$R$-module and $\mathbf{G}$ is a graded free resolution of $N$ 
over $R$, then we can use consecutive cancellations to get a minimal 
free resolution of $N$. As an application we prove the following lemma, 
which is Lemma 9.3 in \cite{B:Pe}. 

\begin{lemma}\label{L:trivial}
Let $\mathbf{G}$ be a finite exact graded complex of finitely generated 
free $R$-modules, then $\mathbf{G}$ is isomorphic to a trivial complex. 
\end{lemma}

\begin{proof}
Let $N=0$, then $\mathbf{G}$ is a free resolution of $N$ and 
$\mathbf{F}=0$ is the only minimal free resolution of $N$, so 
that similar to Theorem \ref{T:CCminimal}, there is a trivial complex 
$\mathcal{E}$ such that $\mathbf{G}\cong \mathbf{F}\oplus \mathcal{E}=\mathcal{E}$.
\end{proof}

Next we use consecutive cancellations to prove a result about 
Taylor basis, which will be used in Section 3. 

\begin{proposition}\label{P:CCTB}
Let $(\mathbf{T},d)$ be the Taylor resolution of $S/M$. Let $\sigma$ 
be a change of basis map of $\mathbf{T}$ and $f_1, \ldots, f_p$ 
the multigraded basis of $(\sigma(\mathbf{T}),d)$. 
Let $\mathbf{F}=\sigma(\mathbf{T})|_{\{f_{i_1}, \ldots, f_{i_q}\}}$ 
be a subcomplex of $\sigma(\mathbf{T})$ such that $\mathbf{F}$ 
is a minimal free resolution of $S/M$. 
Then the dehomogenizations of $f_{i_1}, \ldots, f_{i_q}$ is a 
Taylor basis of $\mathbf{F}$. 
\end{proposition}

\begin{proof}
Without the loss of generality we assume that 
$\{f_{i_1}, \ldots, f_{i_q}\}=\{f_1, \ldots, f_q\}$. 
If $q=p$ then the result obviously holds. 

If $q<p$ then we consider the quotient complex $\sigma(\mathbf{T})/\mathbf{F}$. 
Since $H(\sigma(\mathbf{T}))=H(\mathbf{F})$, it follows that 
$\sigma(\mathbf{T})/\mathbf{F}$ is an exact complex of multigraded 
free $S$-modules with basis $\overline{f_{q+1}}, \ldots, \overline{f_p}$. 
By Lemma \ref{L:trivial} $\sigma(\mathbf{T})/\mathbf{F}$ is isomorphic 
to a trivial complex, which implies that we can use a consecutive 
cancellation to separate a short trivial complex from $\sigma(\mathbf{T})/\mathbf{F}$. 
Hence, there exist $q+1 \leq j \neq l \leq p$ such that 
\[
d(\overline{f_j})=\lambda_{q+1}\overline{f_{q+1}}+\cdots +\lambda_p\overline{f_p},
\]
where $\lambda_{q+1}, \ldots, \lambda_{p}$ are multigraded elements in $S$ and 
$\lambda_l\neq 0 \in k$. 
Thus, there exist multigraded elements $\lambda_1, \ldots, \lambda_q \in S$ 
such that 
\[
d(f_j)=\lambda_1f_1+\cdots +\lambda_qf_q+\lambda_{q+1}f_{q+1}+\cdots +\lambda_pf_p,
\]
with $\lambda_l\neq 0 \in k$. 
After doing a consecutive cancellation at the entry $\lambda_l$ in some 
matrix of $\sigma(\mathbf{T})$, we will get a complex of multigraded free 
$S$-modules $\mathbf{G}$ and a short trivial complex.  Note that $\mathbf{G}$ 
is a free resolution of $S/M$, and by Proposition \ref{P:CCbasis} it is easy to 
see that $\mathbf{G}$ has a multigraded basis $f_1, \ldots, f_q, g_1, \ldots, g_{p-2}$ 
and $\mathbf{G}|_{\{f_1, \ldots, f_q\}}=\mathbf{F}$. 

If $q=p-2$ then $\sigma(\mathbf{T})$ is isomorphic to the direct sum of $\mathbf{F}$ 
and a short trivial complex and the result holds. Otherwise, we can use the above 
method to separate another short trivial complex from $\mathbf{G}$.  
In general, after using $(q-p)/2$ consecutive cancellations, we will get 
$\mathbf{F}$ with the multigraded basis $f_1, \ldots, f_q$. 
So $\sigma(\mathbf{T})$ is isomorphic to the direct sum of $\mathbf{F}$ and 
a trivial complex, which implies that the dehomogenizations of $f_1, \ldots, f_q$ 
is a Taylor basis of $\mathbf{F}$. 
\end{proof}

\subsection{Atomic Lattices and the Lcm-lattice of a Monomial Ideal}

In \cite{B:GPW} Gasharov, Peeva and Welker introduce the concept of the 
lcm-lattice of a monomial ideal. The \emph{lcm-lattice} $L_M$ of $M$ is defined as 
the lattice with elements labelled by the least common multiples of subsets of 
$\{m_1,\ldots, m_r\}$ ordered by divisibility. They prove that minimal free 
resolutions of $S/M$ are determined by $L_M$. 

$L_M$ is an atomic lattice. The bottom element in $L_M$ is $1$ regarded 
as the lcm of the empty set; the atoms are $m_1, \ldots, m_r$; 
and the top  element is $\mbox{lcm}(m_1, \ldots, m_r)$. 
Conversely, in \cite{B:Ph} Phan proves that every atomic lattice is 
the lcm-lattice of some monomial ideal. So the study of monomial resolutions 
is closely related to the study of atomic lattices. 

Let $L$ be an atomic lattice with atoms $\alpha_1, \ldots, \alpha_r$. 
In this paper we will use the following method 
to label the elements in $L$: for any element $\alpha \in L$, $\alpha$ is 
labeled by $A_{\alpha}=\{1\leq i \leq r| \alpha_i \leq \alpha \}$. 
Hence, the bottom element $\hat{0}$ is 
labeled by the empty set $\emptyset$; 
the atoms $\alpha_1, \ldots, \alpha_r$ 
are labeled by $\{1\}, \ldots, \{r\}$, respectively; 
and the top element $\hat{1}$ is labeled by $\{1, \ldots, r\}$. 

Since for any $\alpha \neq \hat{0} \in L$ we have that $\alpha$ is the 
join of all the atoms below $\alpha$, that is, 
$\alpha=\bigvee\limits_{i\in A_{\alpha}} \alpha_i$. 
Thus, it is easy to see that $\alpha\leq \beta$ if and only if 
$A_{\alpha} \subseteq A_{\beta}$, and $\alpha=\beta$ if and 
only if $A_{\alpha}=A_{\beta}$. So the set of all $A_{\alpha}$ ordered 
by inclusion is a poset which is equal to $L$. 

For any $\alpha\in L$, following \cite{B:Cl} we define the \emph{rank} 
of $\alpha$ to be the maximal length of the chains from $\hat{0}$ to 
$\alpha$ and denote it by $\mbox{rk}(\alpha)$. Then we have that 
$\mbox{rk}(\hat{0})=0$, $\mbox{rk}(\alpha_i)=1$ for $1\leq i \leq r$, 
and if $\alpha$ is not the bottom element or an atom, then 
$\mbox{rk}(\alpha)\geq 2$. $\mbox{rk}(\hat{1})$ is called the  \emph{rank} 
of $L$, denoted by $\mbox{rk}(L)$. Also, it is easy to see that 
$\mbox{rk}(\alpha) \leq |A_\alpha|$ for any $\alpha\in L$.

\begin{definition}\label{D:closure}
Let $L$ be an atomic lattice with atoms $\alpha_1, \ldots, \alpha_r$.  
For any $A\subseteq \{1, \ldots, r\}$, let 
$\overline{A}=\bigcap\limits_{A_{\alpha} \supseteq A}  A_{\alpha}$,
then $\overline{A}$ is called the \emph{closure of $A$}. 
Let $\Omega$ be the simplex with the vertex set $\{1,\ldots, r\}$. 
Let $c=\lambda_1c_1+\cdots+\lambda_lc_l$ be a chain in $\Omega$, 
where $\lambda_1, \ldots, \lambda_l$ are nonzero scalars in $k$ and 
$c_1, \ldots, c_l$ are some different faces in $\Omega$. 
If there exists $c_i$ such that 
$\overline{c_i}=\overline{\mbox{supp}(c)}=A_\alpha$ for some 
$\alpha \in L$ then we say that $c$ is a \emph{Taylor chain at $\alpha$}. 
\end{definition}

Next we prove some basic results about atomic lattices. 

\begin{proposition}\label{P:closure}
$\overline{A}=A_{\beta}$ for some $\beta \in L$. In particular, if 
$L=L_M$ then $\overline{A}=A_m$ with $m=\mbox{lcm}(m_i|i\in A)=\mbox{mdeg}(A)$. 
\end{proposition}

\begin{proof}
Let $\beta=\bigvee\limits_{i\in A}\alpha_i$. Then $A_{\beta}\supseteq A$, 
which implies that $A_\beta \supseteq \overline{A}$; on the other hand, 
for any $A_\alpha \supseteq A$ we have that 
$\alpha=\bigvee\limits_{i\in A_\alpha} \alpha_i \geq \bigvee\limits_{i\in A} \alpha_i=\beta$, 
which implies that $\overline{A}=\bigcap\limits_{A_\alpha \supseteq A}A_\alpha \supseteq A_\beta$. 
So, $\overline{A}=A_\beta$. 
Note that if $L=L_M$ then we have that $\bigvee\limits_{i\in A} m_i=\mbox{lcm}(m_i|i\in A)$, 
so that $\overline{A}=A_m$ with $m=\mbox{lcm}(m_i|i\in A)$. 
\end{proof}

Note that for the lcm-lattice $L_M$, every element $m\in L_M$ is labeled 
by $A_m=\{1\leq i \leq r|m_i \ \mbox{divides}\  m\}$ with 
$\mbox{mdeg}(A_m)=m$, and for any $A\subseteq \{1,\ldots, r\}$ we 
have that $\mbox{mdeg}(A)=m$ if and only if $\overline{A}=A_m$. 
If $c$ is a Taylor chain as in Definition \ref{D:multidegree} with 
$\mbox{mdeg}(c)=m\in L_M$, then we have that 
$\overline{\mbox{supp}(c)}=A_m$ and $c$ is a Taylor chain at $m$. 

\begin{proposition}\label{P:meetjoin}
Let $L$ be an atomic lattice with atoms $\alpha_1, \ldots, \alpha_r$. 
Let $\beta_1, \ldots, \beta_s$ be some elements in  $L$,  then we have that
\[
 A_{\beta_1}\cap \cdots \cap A_{\beta_s}=A_{\beta_1\wedge \cdots \wedge \beta_s} \ \ \ \ \ \ \ 
 \overline{A_{\beta_1}\cup \cdots \cup A_{\beta_s}}=A_{\beta_1\vee \cdots \vee \beta_s}.
\]
\end{proposition}

\begin{proof}
For any $1\leq j \leq s$ we have that $\beta_1 \wedge \cdots \wedge \beta_s \leq \beta_j$, 
which implies that $A_{\beta_1\wedge \cdots \wedge \beta_s}\subseteq A_{\beta_j}$, 
and then $A_{\beta_1\wedge \cdots \wedge \beta_s} \subseteq  A_{\beta_1}\cap \cdots \cap A_{\beta_s}$;
On the other hand, for any $i\in  A_{\beta_1}\cap \cdots \cap A_{\beta_s}$ we have 
that $i\in A_{\beta_1}$, which means $\alpha_i \leq \beta_1$, and similarly, 
$\alpha_i\leq \beta_2, \ldots, \alpha_i\leq \beta_s$, so that 
$\alpha_i\leq \beta_1\wedge \cdots \wedge \beta_s$, which implies that 
$i\in  A_{\beta_1\wedge \cdots \wedge \beta_s}$, and then 
$A_{\beta_1}\cap \cdots \cap A_{\beta_s} \subseteq A_{\beta_1\wedge \cdots \wedge \beta_s}$. 
So, $A_{\beta_1}\cap \cdots \cap A_{\beta_s}=A_{\beta_1\wedge \cdots \wedge \beta_s}$.

For any $1\leq j \leq s$ we have that $\beta_j \leq \beta_1\vee \cdots \vee \beta_s$, 
which implies that $A_{\beta_j} \subseteq A_{\beta_1\vee \cdots \vee \beta_s}$, 
and then $A_{\beta_1}\cup \cdots \cup A_{\beta_s}\subseteq A_{\beta_1\vee \cdots \vee \beta_s}$, 
so that $\overline{A_{\beta_1}\cup \cdots \cup A_{\beta_s}}\subseteq A_{\beta_1\vee \cdots \vee \beta_s}$;
On the other hand, let $\overline{A_{\beta_1}\cup \cdots \cup A_{\beta_s}}=A_\beta$ 
for some $\beta \in L$, then $A_{\beta_1}\subseteq A_\beta$ which implies that 
$\beta_1 \leq \beta$, and similarly, $\beta_2\leq \beta, \ldots, \beta_s\leq \beta$, 
so that $\beta_1\vee \cdots \vee \beta_s \leq \beta$, which implies that 
$A_{\beta_1\vee \cdots \vee \beta_s} \subseteq A_\beta=\overline{A_{\beta_1}\cup \cdots \cup A_{\beta_s}}$. 
So, $\overline{A_{\beta_1}\cup \cdots \cup A_{\beta_s}}=A_{\beta_1\vee \cdots \vee \beta_s}$.
\end{proof}

\begin{corollary}\label{C:unionclosure}
Let $L$ be an atomic lattice with atoms $\alpha_1, \ldots, \alpha_r$. 
Let $A_1, \ldots, A_s$ be some subsets of $\{1, \ldots, r\}$. Then we have that 
\[
\overline{A_1 \cup \cdots \cup A_s}=\overline{\overline{A_1} \cup \cdots \cup \overline{A_s}}. 
\]
\end{corollary}

\begin{proof}
By Proposition \ref{P:closure} and its proof we see that 
$\overline{A_1 \cup \cdots \cup A_s}=A_\beta$, where 
$\beta=\bigvee\limits_{i\in A_1 \cup \cdots \cup A_s}\alpha_i$, 
and for any $1\leq j\leq s$, $\overline{A_j}=A_{\beta_j}$,  
where $\beta_j=\bigvee\limits_{i\in A_j}\alpha_i$. 
Hence, we have that $\beta=\beta_1\vee \cdots \vee \beta_s$. 
So by Proposition \ref{P:meetjoin} we have that 
$\overline{A_1 \cup \cdots \cup A_s}=\overline{\overline{A_1} \cup \cdots \cup \overline{A_s}}$. 
\end{proof}

\begin{proposition}\label{P:cover}
Let $L$ be an atomic lattice with atoms $\alpha_1, \ldots, \alpha_r$. 
Assume $\beta\neq \hat{0} \in L$ is not an atom. Let $\beta_1, \ldots, \beta_t$ 
be all the elements covered by $\beta$ in $L$. 
Then $t\geq 2$; $\forall 1\leq i<j \leq t, \beta=\beta_i \vee \beta_j$; and 
$A_\beta=A_{\beta_1}\cup \cdots \cup A_{\beta_t}$.
\end{proposition}

\begin{proof}
Assume $t=1$, then $A_\beta=A_{\beta_1}$ so that $\beta=\beta_1$, 
contradicting to the assumption that $\beta_1$ is covered by $\beta$. 
Hence, $t\geq 2$. 

Let $(\hat{0},\beta)=\{\alpha \in L |\hat{0} <\alpha <\beta\}$. Since 
$\beta_1, \ldots, \beta_t$ are all the elements covered by $\beta$, it 
follows that $\beta_1, \ldots, \beta_t$ are the maximal elements in 
$(\hat{0},\beta)$. For any $1\leq i<j\leq t$ we have that 
$\beta_i \leq \beta_i \vee \beta_j \leq \beta$. 
Assume $\beta_i = \beta_i \vee \beta_j$, then $\beta_j \leq \beta_i$, 
which contradicts to the maximality of $\beta_i$ and $\beta_j$ in 
$(\hat{0}, \beta)$. Hence, $\beta_i \neq \beta_i \vee \beta_j$. 
Since $\beta_i$ is covered by $\beta$, it follows that 
$\beta_i \vee \beta_j=\beta$. 

For any  $1\leq j \leq t$ we have that $A_{\beta_j}\subseteq A_\beta$, 
so that $A_{\beta_1}\cup \cdots \cup A_{\beta_t}\subseteq A_\beta$. 
On the other hand, for any $i\in A_\beta$ we have that $\alpha_i <\beta$. 
Hence, $\alpha_i \in (\hat{0}, \beta)$, so that there exists $1\leq j\leq t$ 
such that $\alpha_i \leq \beta_j$, which implies that $i\in A_{\beta_j}$. 
Thus, $A_\beta \subseteq A_{\beta_1}\cup \cdots \cup A_{\beta_t}$. 
So, $A_\beta=A_{\beta_1}\cup \cdots \cup A_{\beta_t}$.
\end{proof}

By Proposition \ref{P:cover} we have the following definition. 

\begin{definition}\label{D:simplicialcomplex}
Let $L$ be an atomic lattice with atoms $\alpha_1, \ldots, \alpha_r$. 
Let $\beta\neq \hat{0} \in L$. If $\beta$ is an atom of $L$,
then we set $\Delta_\beta=\{\emptyset \}$; 
if $\beta$ is not an atom,  let $\beta_1, \ldots, \beta_t$ be all the 
elements in $L$ covered by $\beta$, then 
we set $\Delta_\beta$ to be the simplicial complex with the vertex 
set $A_\beta$ and facets $A_{\beta_1}, \ldots, A_{\beta_t}$, 
that is, 
\[
\Delta_\beta=\langle A_{\beta_1}, \ldots, A_{\beta_t} \rangle.
\]
We call $\Delta_\beta$ the \emph{simplicial complex at $\beta$}. 
For any $\beta\neq \hat{0} \in L$ let $\Omega_\beta$ be the simplex 
with the vertex set $A_\beta$. By Proposition \ref{P:cover} it is easy 
to see that $\Delta_\beta \neq \Omega_\beta$. 
We call $\Omega_\beta$ the \emph{simplex at $\beta$}. 
Let $\mathbb{Q}_\beta$ be the quotient complex 
$\widetilde{C}(\Omega_\beta;k)/\widetilde{C}(\Delta_\beta;k)=\widetilde{C}(\Omega_\beta, \Delta_\beta;k)$. 
We call $\mathbb{Q}_\beta$ the \emph{quotient complex at $\beta$}. 
Let $Q_\beta$ be the set of faces $\{ A\subseteq A_\beta| \overline{A}=A_\beta \}$, 
then $Q_\beta$ is a basis of $\mathbb{Q}_\beta$. 
\end{definition}

Let $L=L_M$, then for any $m\neq 1 \in L_M$ we have a simplicial 
complex $\Delta_m$ at $m$. By Theorem 57.9 and Theorem 58.8 
in \cite{B:Pe}, we have the following formula for the multigraded 
Betti numbers of $S/M$:
\[
\mbox{for \ any}\ i\geq 1 \  \mbox{and}\ m\neq 1\in L_M, \ 
b_{i,m}(S/M)=\dim_k\widetilde{H}_{i-2}(\Delta_m;k).
\]
For example, if $m_j$ is an atom in $L_M$ then we have that 
$\Delta_{m_j}=\{\emptyset\}$ and 
$b_{1,m_j}=\dim_k \widetilde{H}_{-1}(\{\emptyset\};k)=1$. 
Next we will prove a stronger result which relates a Taylor basis 
of a minimal free resolution of $S/M$ 
and a basis of $\widetilde{H}_{i-2}(\Delta_m;k)$ for all $m\neq 1 \in L_M$.

\begin{theorem}\label{T:basis}
Let $M$ be a monomial ideal in $S$ minimally generated by $r$ monomials. 
Let $(\Omega, d)$ be the simplex with the vertex set $\{1, \ldots, r\}$ .
Let $i\geq 1$ and $m\neq 1 \in L_M$ such that $b_{i,m}(S/M)=p\geq 1$. 
Let $g_1, \ldots, g_p$ be all the elements in a Taylor basis $B$ of a 
minimal free resolution $\mathbf{F}$ of $S/M$ such that $g_1, \ldots, g_p$ 
are of multidegree $m$ and in homological degree $i$. 
Then $[d(g_1)], \ldots, [d(g_p)]$ is a basis of $\widetilde{H}_{i-2}(\Delta_m;k)$,
and in particular, 
\[
b_{i,m}(S/M)=\dim_k\widetilde{H}_{i-2}(\Delta_m;k).
\]
\end{theorem}

\begin{proof}
Let $(\mathbf{T}, \partial)$ be the Taylor resolution of $S/M$ with a basis 
containing all the faces of $\Omega$. Let $\sigma$ be a change of basis 
map of $\mathbf{T}$ such that $\sigma(\mathbf{T})=\mathbf{F}\oplus \mathcal{E}$, 
where $\mathcal{E}$ is a trivial complex and the dehomogenization of the 
basis elements of $\mathbf{F}$ is the Taylor basis $B$. Then we have 
the following commutative diagram:
\[
 \begin{tikzcd}
 \mathbf{T}:\  \ \ 0 \arrow{r} & T_r \arrow[r,"A_r"] &[2em] T_{r-1} \arrow[r]& \cdots \arrow[r] 
 & T_1 \arrow[r,"A_1"] &[2em]  T_0 \arrow[r] &0 \\
 \sigma(\mathbf{T}):\  0 \arrow{r} & T_r \arrow[r,"U_{r-1}^{-1}A_rU_r"] \arrow[u,"U_r"] 
 &T_{r-1} \arrow[r] \arrow[u,"U_{r-1}"]& \cdots \arrow[r] 
 & T_1 \arrow[r,"U_0^{-1}A_1U_1"]\arrow[u,"U_1"] & T_0 \arrow[r] \arrow[u,"U_0"]&0,
 \end{tikzcd}
\]
where $A_i$ is the matrix of $\partial_i$ under the basis of $\mathbf{T}$, 
$U_j \in \mbox{Aut}(T_j)$ is corresponding to the change of basis map 
$\sigma_j$ of $T_j$, and the matrix of $\partial_i$ under the new basis 
of $\mathbf{T}$ is $U_{i-1}^{-1}A_iU_i$. 
Let $\widetilde{A}_i$ and $\widetilde{U}_j$ be obtained from $A_i$ 
and $U_j$, respectively, by setting $x_1=\cdots=x_n=0$ in these matrices. 
Since $U_j$ is invertible, it follows that $\det(U_j)=\det(\widetilde{U}_j)$, 
and then $\widetilde{U}_j$ is also invertible. Then we have the following 
commutative diagram:
\[
 \begin{tikzcd}
 \mathbf{T}\otimes k: 0 \arrow{r} & [-1.5em]T_r\otimes k \arrow[r,"\widetilde{A}_r"] 
 &[1.2em] T_{r-1}\otimes k \arrow[r]& [-1.6em]\cdots \arrow[r] 
 &[-1.6em] T_1\otimes k \arrow[r,"\widetilde{A}_1"] &[1.2em]  T_0\otimes k \arrow[r] & [-1.6em]0 \\
 \sigma(\mathbf{T})\otimes k: 0 \arrow{r} 
 & T_r\otimes k \arrow[r,"\widetilde{U}_{r-1}^{-1}\widetilde{A}_r\widetilde{U}_r"] \arrow[u,"\widetilde{U}_r"] 
 &T_{r-1}\otimes k \arrow[r] \arrow[u,"\widetilde{U}_{r-1}"]& \cdots \arrow[r] 
 & T_1\otimes k \arrow[r,"\widetilde{U}_0^{-1}\widetilde{A}_1\widetilde{U}_1"]\arrow[u,"\widetilde{U}_1"] 
 & T_0\otimes k \arrow[r] \arrow[u,"\widetilde{U}_0"]&0.
 \end{tikzcd}
\]
Since $\widetilde{U}_0, \ldots, \widetilde{U}_r$ are invertible, it follows 
that $\mathbf{T}\otimes k \cong \sigma(\mathbf{T})\otimes k$. 
Let $\widetilde{\sigma}$ be the change of basis map of $\mathbf{T}\otimes k$ 
corresponding to $\widetilde{U}_0, \ldots, \widetilde{U}_r$, then 
we have that $\widetilde{\sigma}(\mathbf{T}\otimes k)=  \sigma(\mathbf{T})\otimes k$.

Note that $\mathbf{T}\otimes k$ has the same basis as that of $\mathbf{T}$ and 
the initial parts of the basis elements of $\sigma(\mathbf{T})$ are the basis 
elements of $\sigma(\mathbf{T})\otimes k$. 
Since $\sigma(\mathbf{T})\otimes k=(\mathbf{F}\otimes k) \oplus (\mathcal{E}\otimes k)$, 
it follows that $\mathbf{F}\otimes k$ is a direct sum of complexes of the form 
$0\to k \to 0$, placed in different homological degrees, each corresponding to 
an element in the Taylor basis of $\mathbf{F}$; and $\mathcal{E}\otimes k$ is 
a direct sum of complexes of the form $0\to k \xrightarrow{1} k \to 0$, placed 
in different homological degrees.  Note that 
\[
\mathbf{T}\otimes k=\left(\bigoplus\limits_{m\neq 1 \in L_M}\mathbb{Q}_m\right)\oplus(T_0\otimes k).
\]
Since the change of basis map $\widetilde{\sigma}$ does not mix basis 
elements of different multidegrees, it follows that $\widetilde{\sigma}$ 
can be restricted to each $\mathbb{Q}_m$. Hence, $\widetilde{\sigma}(\mathbb{Q}_m)$ 
is the direct sum of all $0\to k\to 0$ in $\mathbf{F}\otimes k$ with a basis 
element of multidegree $m$ and all $0 \to k \xrightarrow{1} k \to 0$ in 
$\mathcal{E}\otimes k$ with basis elements of multidegree $m$, placed 
in appropriate homological degrees. 

Let $h_1, \ldots, h_t$ be all the basis elements of $\mathbf{F}\otimes k$ 
with multidegree $m$. Then $[h_1], \ldots, [h_t]$ is a basis of 
$H(\widetilde{\sigma}(\mathbb{Q}_m);k)=H(\mathbb{Q}_m;k)$. 
Since $\mbox{in}(g_1), \ldots, \mbox{in}(g_p)$ are the basis elements 
of $\mathbf{F}\otimes k$ with multidegree $m$ in homological degree $i$, 
it follows that $\mbox{in}(g_1), \ldots, \mbox{in}(g_p)$ are $k$-linear 
combinations of some $(i-1)$-dimensional faces in $Q_m$, and then 
$[\mbox{in}(g_1)], \ldots, [\mbox{in}(g_p)]$ is a basis of $H_{i-1}(\mathbb{Q}_m;k)$.

Note that for any $1\leq j\leq p$, $\mbox{in}(g_j)$ is mapped to zero in 
$\mathbb{Q}_m$, so that $d(\mbox{in}(g_j))$ 
is a cycle in $\Omega_m$ containing no faces in $Q_m$, which implies that 
$d(\mbox{in}(g_j))$ is a cycle in $\Delta_m$. 
Since $\Omega_m$ is a simplex, by the exact sequence of complexes:
\[
0\to \widetilde{C}(\Delta_m;k) \to \widetilde{C}(\Omega_m;k) \to \mathbb{Q}_m \to 0,
\]
we have that $H_{i-1}(\mathbb{Q}_m;k)\cong \widetilde{H}_{i-2}(\Delta_m;k)$ 
and moreover, $[d(\mbox{in}(g_1))], \ldots, [d(\mbox{in}(g_p))]$ 
is a basis of $\widetilde{H}_{i-2}(\Delta_m;k)$.

Note that for any $1\leq j\leq p$, $g_j-\mbox{in}(g_j)$ is a $k$-linear 
combination of some $(i-1)$-dimensional faces of $\Delta_m$, so that 
$[d(g_j-\mbox{in}(g_j))]=0$ in $\widetilde{H}_{i-2}(\Delta_m;k)$, which 
implies that $[d(g_j)]=[d(\mbox{in}(g_j))]$ in $\widetilde{H}_{i-2}(\Delta_m;k)$. 
So $[d(g_1)], \ldots, [d(g_p)]$ is a basis of $\widetilde{H}_{i-2}(\Delta_m;k)$. 
\end{proof}

By Theorem \ref{T:basis} we can easily get a basis for 
$\widetilde{H}_{i-2}(\Delta_m;k)$ from a Taylor basis.  
Conversely, one may wonder if it is possible to obtain a Taylor 
basis form the bases of $\widetilde{H}(\Delta_m;k)$ for 
all $m\neq 1 \in L_M$. We will discuss this problem in 
Theorem \ref{T:cbasis}.

\section{Atomic Lattice Resolutions}

In this section, given any atomic lattice $L$ with $r$ atoms, we  
construct an $r$-frame $\mathbb{V}(L)$. Let $L_M$ be the lcm-lattice 
of $M$. We show that the $M$-homogenization of $\mathbb{V}(L_M)$  
is a minimal free resolution of $S/M$, which is called an atomic lattice 
resolution of $S/M$. And conversely, we show that every minimal 
resolution of $S/M$ can be obtained as an atomic lattice resolution.  
These results are called  the atomic lattice resolution theory, which 
can  be used to prove some results about monomial resolutions.  
But first we need to do some preparations and introduce a new 
concept called the exact closure of a complex of $k$-vector spaces.

\subsection{Exact Closures}
As mentioned in Remark \ref{R:kcomplex}, every complex of $k$-vector 
spaces $(\mathbb{U},d)$ is assumed to be finite and have a fixed basis. 
More specifically, we always assume that for any $i<0$ or $i\gg 0$, $U_i=0$;
and we set $\dim_k\mathbb{U}=\sum\limits_i\dim_kU_i$. 

\begin{definition}\label{D:exactclosure}
Let $(\mathbb{U},d)$ be a complex of $k$-vector spaces. 
Let $(\mathbb{V},\partial)$ be an exact complex of $k$-vector spaces 
such that $(\mathbb{U}, d)$ is a subcomplex of $(\mathbb{V},\partial)$, 
that is, there exists a subset $\{f_{i_1}, \ldots, f_{i_t}\}$ of the basis 
$\{f_1, \ldots, f_s\}$ of $\mathbb{V}$ such that 
$(\mathbb{U},d)=(\mathbb{V}|_{\{f_{i_1}, \ldots, f_{i_t}\}}, \partial)$.
If there does not exist an exact complex of $k$-vector spaces 
$(\mathbb{W}, \delta)$ such that $(\mathbb{U},d)$ is a subcomplex 
of $(\mathbb{W}, \delta)$ and $\dim_k\mathbb{W}<\dim_k\mathbb{V}$, 
then we say that $(\mathbb{V},\partial)$ is an \emph{exact closure} of 
$(\mathbb{U},d)$. 
\end{definition}

The next construction and proposition imply the existence of an exact closure.

\begin{construction}\label{C:exactclosure}
Let $(\mathbb{U},d): 0\to U_p \xrightarrow{d_p} U_{p-1} \to \cdots \to U_1 
\xrightarrow{d_1} U_0 \xrightarrow{d_0} 0$ be a complex of $k$-vector 
spaces with a fixed basis $f_1, \ldots, f_s$. For any $0\leq i\leq p$, let 
\[
\mu_i=\dim_k\frac{\mbox{Ker}d_i}{\mbox{Im}d_{i+1}},
\]
where we assume that $U_{p+1}=0$ and $d_{p+1}=0$. 
Pick $c_1^i, \ldots, c_{\mu_i}^i \in \mbox{Ker}d_i \subseteq U_i$ 
such that $[c_1^i], \ldots, [c_{\mu_i}^i]$ is a basis of 
$\mbox{Ker}d_i/\mbox{Im}d_{i+1}$. 
Let $G_{i+1}$ be a $\mu_i$-dimensional $k$-vector space with basis 
$g_1^i, \ldots, g_{\mu_i}^i$. We define a linear map 
$\varphi_{i+1}: G_{i+1}\to U_i$ such that 
$\varphi_{i+1}(g_j^i)=c_j^i$ for all $1\leq j\leq \mu_i$. 

Let $V_0=U_0$ and for any $1\leq i\leq p+1$  let $V_i=U_i\oplus G_i$. 
Let $\partial_0=d_0$ and for any $1\leq i\leq p+1$  let 
\[ 
\partial_i=\begin{pmatrix} d_i & \varphi_i \\ 0 & 0 \end{pmatrix} : 
V_i=U_i\oplus G_i \to V_{i-1}=U_{i-1}\oplus G_{i-1}. 
\]
Then by the next proposition, the complex 
\[
(\mathbb{V}, \partial): 0\to V_{p+1} \xrightarrow{\partial_{p+1}} 
V_p \to \cdots \to V_1 \xrightarrow{\partial_1} V_0 \xrightarrow{\partial_0} 0
\]
with basis $\{f_1, \ldots, f_s\}\cup \{g_j^i| 0\leq i\leq p, 1\leq j \leq \mu_i \}$ 
is an exact closure of $(\mathbb{U}, d)$. 
\end{construction}

In the proof of the next proposition we will use the following fact from 
elementary linear algebra: Let $\varphi : V\to W$ be a linear map between 
finite $k$-vector spaces $V$ and $W$; let $U$ be a subspace of $V$; then 
we have that 
\[
\dim_kV-\dim_kU\geq \dim_k\varphi(V)-\dim_k\varphi(U). 
\]

\begin{proposition}\label{P:exactclosure}
$(\mathbb{V},\partial)$ obtained in Construction \ref{C:exactclosure} is an 
exact closure of $(\mathbb{U}, d)$. 
\end{proposition}

\begin{proof} 
Since for any $0\leq i \leq p$ we have that $d_id_{i+1}=0$ and 
$\mbox{Im}\varphi_{i+1}  \subseteq \mbox{Ker}d_i$, it follows that 
\[
\partial_i\partial_{i+1}=
\begin{pmatrix} d_i &\varphi_i \\ 0 & 0 \end{pmatrix}
\begin{pmatrix} d_{i+1} & \varphi_{i+1} \\ 0 & 0 \end{pmatrix} 
=\begin{pmatrix} d_id_{i+1} & d_i\varphi_{i+1} \\ 0& 0 \end{pmatrix} =0,
\]
which implies that $(\mathbb{V}, \partial)$ is a complex. 

For any $0\leq i \leq p$, let $a\in U_i$ and $b\in G_i$ such that 
$\partial_i \begin{pmatrix} a\\ b \end{pmatrix} =0$, then 
\[
\begin{pmatrix} d_i & \varphi_i \\ 0 & 0 \end{pmatrix} 
\begin{pmatrix} a\\ b \end{pmatrix} =\begin{pmatrix} d_i(a)+\varphi_i(b) \\ 0 \end{pmatrix} =0, 
\ \mbox{i.e.}, \  d_i(a)=-\varphi_i(b).
\]
Since $d_i(a)\in \mbox{Im}d_i$, $-\varphi_i(b) \in \mbox{Im}\varphi_i$ 
and $\mbox{Im}d_i\cap \mbox{Im}\varphi_i=\{0\}$, it follows that 
$d_i(a)=\varphi_i(b)=0$, which implies that $a\in\mbox{Ker}d_i$, and 
by the injectivity of $\varphi_i$ we get that $b=0$. Hence, we have 
that $\mbox{Ker}d_i=\mbox{Ker}\partial_i$. 
Note that $\mbox{Im}\partial_{i+1}=\mbox{Im}d_{i+1}\oplus 
\mbox{Im}\varphi_{i+1}=\mbox{Ker}d_i$. 
Thus, we have that $\mbox{Im}\partial_{i+1}=\mbox{Ker}\partial_i$ for all 
$0\leq i \leq p$. Note that $\mbox{Ker}\partial_{i+1}=0$. So, 
$(\mathbb{V}, \partial)$ is an exact complex. 

Let $(\mathbb{W},\delta)$ be an exact complex of $k$-vector spaces 
such that $(\mathbb{U}, d)$ is a subcomplex of $(\mathbb{W}, \delta)$. 
Then for any $0\leq i \leq p$ we have that 
\begin{align*}
\dim_kW_{i+1}-\dim_kU_{i+1} 
&\geq \dim_k\mbox{Im}\delta_{i+1}-\dim_k\mbox{Im}d_{i+1}\\
&= \dim_k\mbox{Ker}\delta_i-\dim_k\mbox{Im}d_{i+1}\\
&\geq \dim_k\mbox{Ker}d_i-\dim_k\mbox{Im}d_{i+1}\\
&=\mu_i,
\end{align*}
which implies that
\begin{align*}
\dim_kW_0 & \geq \dim_kU_0=\dim_kV_0\\
\dim_kW_1 & \geq \dim_kU_1+\mu_0=\dim_kV_1 \\
         &\vdots \\
\dim_kW_p & \geq \dim_kU_p+\mu_{p-1}=\dim_kV_p\\
\dim_kW_{p+1} & \geq \dim_kU_{p+1}+\mu_p=0+\mu_p=\dim_kV_{p+1}.
\end{align*}
Thus, we have that $\dim_k\mathbb{W}=\sum\limits_{i}\dim_kW_i\geq 
\sum\limits_{i=0}^{p+1}\dim_kV_i=\dim_k\mathbb{V}$. 
So, $(\mathbb{V},\partial)$ is an exact closure of $(\mathbb{U},d)$. 
\end{proof}

\begin{remark}\label{R:exactclosure1}
 By the definition of an exact closure and the proof of Proposition 
\ref{P:exactclosure}, it is easy to see that if $(\mathbb{V},\partial)$ 
is an exact closure of $(\mathbb{U},d)$ then we have that 
\begin{align*}
\dim_kV_0 & =\dim_kU_0, \ V_i=0 \ \mbox{for any} i<0 \ \mbox{or} \ i>p+1, \\
\dim_kV_{i+1} & =\dim_kU_{i+1}+\mu_i, \ \mbox{for any} \  0\leq i \leq p;
\end{align*}
and conversely, if $(\mathbb{V}, \partial)$ is an exact complex of 
$k$-vector spaces such that $(\mathbb{U},d)$ is a subcomplex of 
$(\mathbb{V}, \partial)$ and 
$\dim_k\mathbb{V}=\dim_k\mathbb{U}+\sum\limits_{i=0}^p\mu_i$, 
then $(\mathbb{V},\partial)$ is an exact closure of $(\mathbb{U}, d)$. 
Hence, if $(\mathbb{U},d)$ is exact, then $(\mathbb{U},d)$ has a unique 
exact closure which is itself; if $(\mathbb{U},d)$ is not exact, then there 
are  many choices for $c_j^i$, so that $(\mathbb{U},d)$ has 
 many different exact closures. 
\end{remark}

\begin{remark}\label{R:exactclosure2}
Let $(\mathbb{W},\delta)$ be an exact complex of $k$-vector spaces 
such that $(\mathbb{U},d)$ is a subcomplex of $(\mathbb{W}, \delta)$. 
Then for any $0\leq i\leq p$ we have that 
$\mbox{Im}\delta_{i+1}=\mbox{Ker}\delta_i\supseteq \mbox{Ker}d_i$, 
so that in Construction \ref{C:exactclosure} we can choose 
$g_1^i, \ldots, g_{\mu_i}^i\in W_{i+1}$ such that 
$\varphi_{i+1}(g_j^i)=\delta_{i+1}(g_j^i)=c_j^i\in \mbox{Ker}d_i$, 
and then $G_{i+1}$ spanned by $g_1^i,\ldots, g_{\mu_i}^i$  is a 
subspace of $W_{i+1}$. We call the resulting exact closure 
$(\mathbb{V}, \partial)$ an \emph{exact closure} of $(\mathbb{U}, d)$ 
constructed in $(\mathbb{W}, \delta)$. 
\end{remark}

In the next proposition we give a criterion for exact closure and we use it 
in Remark \ref{R:exactclosure3} to show that every exact closure of 
$(\mathbb{U},d)$ can be obtained by Construction \ref{C:exactclosure}. 

\begin{proposition}\label{P:exactclosureker}
Let $(\mathbb{W}, \delta)$ be an exact complex of $k$-vector spaces such 
that $(\mathbb{U}, d)$ is a subcomplex of $(\mathbb{W},\delta)$. Then 
$(\mathbb{W}, \delta)$ is an exact closure of $(\mathbb{U}, d)$ if and only 
if $\mbox{Ker}\delta_i=\mbox{Ker}d_i$ for all $i$. 
\end{proposition}

\begin{proof}
For any $0\leq i \leq p$ let $\mu_i$ be as in Construction \ref{C:exactclosure}; 
For any $i<0$ or $i>p$ let $\mu_i=0$. Hence, we have that 
$\dim_k\mbox{Ker}d_i=\dim_k\mbox{Im}d_{i+1}+\mu_i$ for all $i$. 

\emph{if}: For any $i$, since $(\mathbb{W},\delta)$ is exact and 
$\mbox{Ker}\delta_i=\mbox{Ker}d_i$, it follows that 
\begin{align*}
\dim_kW_i&=\dim_k\mbox{Ker}\delta_i+\dim_k\mbox{Im}\delta_i \\
                &=\dim_k\mbox{Ker}d_i +\dim_k\mbox{Ker}\delta_{i-1} \\
                &=\dim_k\mbox{Ker}d_i+\dim_k\mbox{Ker}d_{i-1} \\
                &=\dim_k\mbox{Ker}d_i+\dim_k\mbox{Im}d_i+\mu_{i-1}\\
                &=\dim_kU_i+\mu_{i-1},
\end{align*}
so that $\dim_k\mathbb{W}=\sum\limits_i\dim_kW_i
=\sum\limits_i(\dim_kU_i+\mu_{i-1}) 
=\dim_k\mathbb{U}+\sum\limits_{i=0}^p\mu_i$. 
Thus, by Remark \ref{R:exactclosure1}, $(\mathbb{W}, \delta)$ is an 
exact closure of $(\mathbb{U},d)$. 

\emph{only if}: Assume that $(\mathbb{W}, \delta)$ is an exact 
closure of $(\mathbb{U},d)$, then by Remark \ref{R:exactclosure1} 
we have that for any $i<0$ or $i>p+1$, $W_i=0$; 
$\dim_kW_0=\dim_kU_0$; and for any $0\leq i\leq p$, 
$\dim_kW_{i+1}=\dim_kU_{i+1}+\mu_i$. 
Hence, for any $0\leq i\leq p$ we have that 
\begin{align*}
\mu_i&=\dim_k\mbox{Ker}d_i-\dim_k\mbox{Im}d_{i+1}\\
               &\leq \dim_k\mbox{Ker}\delta_i-\dim_k\mbox{Im}d_{i+1}  \\
               &=\dim_k\mbox{Im}\delta_{i+1}-\dim_k\mbox{Im}d_{i+1}  \\
               &\leq  \dim_kW_{i+1}-\dim_kU_{i+1}=\mu_i,
\end{align*}
which implies that $\dim_k\mbox{Ker}\delta_i=\dim_k\mbox{Ker}d_i$, 
and then $\mbox{Ker}\delta_i=\mbox{Ker}d_i$. 
Note that for any $i<0$ or $i>p$, $\mbox{Ker}\delta_i=\mbox{Ker}d_i=0$. 
So $\mbox{Ker}\delta_i=\mbox{Ker}d_i$ for all $i$.
\end{proof}

\begin{remark}\label{R:exactclosure3}
Let $(\mathbb{W},\delta)$ be an exact closure of $(\mathbb{U},d)$, then 
by Remark \ref{R:exactclosure1} it is easy to see that for any $0\leq i\leq p$ 
there are $\mu_i$ basis elements $g_1^i, \ldots, g_{\mu_i}^i \in W_{i+1}$ 
which are not contained in $U_{i+1}$. For any $1\leq j\leq \mu_i$,  
let $c_j^i=\delta_{i+1}(g_j^i)\in \mbox{Im}\delta_{i+1}=\mbox{Ker}\delta_i$.
Then by Proposition \ref{P:exactclosureker} we have that $c_j^i\in \mbox{Ker}d_i$. 
Note that $\mbox{Im}\delta_{i+1}/\mbox{Im}d_{i+1}=
\mbox{Ker}\delta_i/\mbox{Im}d_{i+1}=\mbox{Ker}d_i/\mbox{Im}d_{i+1}$ is a 
$k$-vector space of dimension $\mu_i$. 
Since every vector in $\mbox{Im}\delta_{i+1}/\mbox{Im}d_{i+1}$ can be 
expressed as a $k$-linear combination of $[c_1^i], \ldots, [c_{\mu_i}^i]$, 
it follows that $[c_1^i], \ldots, [c_{\mu_i}^i]$ is a basis of 
$\mbox{Ker}d_i/\mbox{Im}d_{i+1}$. Thus,  by letting $G_{i+1}$ be the $k$-vector 
space generated by $g_1^i, \ldots, g_{\mu_i}^i$, we can use 
Construction \ref{C:exactclosure} to get $(\mathbb{W},\delta)$. 
So every exact closure of $(\mathbb{U},d)$ can be obtained by 
Construction \ref{C:exactclosure}. 
\end{remark}

\subsection{The Atomic Lattice Resolution Theory}
Let $L$ be an atomic lattice. In the next construction, we will associate $L$ 
with an exact complex of $k$-vector spaces $\mathbb{V}(L)$. 

\begin{construction}\label{C:VL}
Let $L$ be an atomic lattice with atoms $\alpha_1, \ldots, \alpha_r$. 
Let $\Omega$ be the simplex with the vertex set $\{1,\ldots, r\}$ 
and let $d$ be the boundary map in $\Omega$. 
In this construction we will inductively associate each element $\alpha\in L$ 
with a set of chains in $\Omega$ such that the support of each chain is contained 
in $A_{\alpha}$, and these chains are called \emph{Taylor basis elements at $\alpha$}. 
We use $\Gamma_\alpha$ to denote the set of Taylor basis elements at $\alpha$. 
For any $\alpha \neq \hat{0} \in L$ we can apply the boundary map $d$ to the 
set $\bigcup\limits_{\beta\leq \alpha}\Gamma_\beta$ to obtain an exact complex of 
$k$-vector spaces. This complex with basis 
$\bigcup\limits_{\beta\leq \alpha}\Gamma_\beta$  is denoted by 
$(\mathbb{V}_\alpha,d)$ and is called the \emph{frame at $\alpha$}. 
The frame $(\mathbb{V}_{\hat{1}}, d)$ at the top element $\hat{1}$ of $L$ 
will be denoted by $\mathbb{V}(L)$. 

Base case: Let $\Gamma_{\hat{0}}=\{ \lambda \emptyset \}$ for some $\lambda \neq 0 \in k$. 
For any atom $\alpha_i$ in $L$, let $\Gamma_{\alpha_i}=\{\lambda_i\{i\}\}$ 
for some $\lambda_i \neq 0 \in k$. Then $(\mathbb{V}_{\alpha_i}, d)$ is the exact 
complex $0\to k\xrightarrow{\lambda_i/\lambda} k \to 0$. 
Without the loss of generality and for simplicity, in the rest of the paper we will 
assume that $\Gamma_{\hat{0}}=\{ \emptyset \}$ and for any $1\leq i\leq r$, 
$\Gamma_{\alpha_i}=\{\{i\}\}$. 

Inductive step: Let $\alpha \in L$ with $\mbox{rk}(\alpha)\geq 2$, then $\alpha$ 
is not the bottom element or an atom of $L$ and $|A_\alpha|\geq 2$. 
Assume that for any $\beta <\alpha$ in $L$, $\Gamma_\beta$ has been obtained 
and we have the frame $(\mathbb{V}_\beta, d)$ at $\beta$. 
Next we want to associate $\alpha$ with some $\Gamma_\alpha$. 
Let $\{e_1, \ldots, e_s\}=\bigcup\limits_{\beta<\alpha}\Gamma_\beta$. 
Then if $e_i \in \Gamma_\beta$ we have that 
$\mbox{supp}(e_i)\subseteq A_\beta \subset A_\alpha$. 
By the inductive assumption we can apply the boundary map $d$ to the 
chains in $\{e_1,\ldots, e_s\}$ to obtain a complex of $k$-vector spaces, 
which is denoted by $(\mathbb{U}_\alpha, d)$. 
Let $(\mathbb{V}_\alpha,\partial)$ be an exact closure of $(\mathbb{U}_\alpha, d)$.  
Let $f \notin\{e_1, \ldots, e_s\}$ be a basis element of $(\mathbb{V}_\alpha, \partial)$ 
in homological degree $i$, then there exist nonzero scalars 
$\lambda_{j_1}, \ldots, \lambda_{j_l}$ and some Taylor basis elements 
$e_{j_1}, \ldots, e_{j_l}$ of dimension $i-2$ such that 
\[
\partial(f)=\lambda_{j_1}e_{j_1}+\cdots+\lambda_{j_l}e_{j_l}\in \mbox{Ker}d_{i-1}.
\]
Since $\mbox{supp}(e_{i_1})\cup \cdots \cup \mbox{supp}(e_{i_l})\subseteq A_\alpha$ 
and $d(\lambda_{j_1}e_{j_1}+\cdots+\lambda_{j_l}e_{j_l})=\partial^2(f)=0$, 
it follows that $\lambda_{j_1}e_{j_1}+\cdots+\lambda_{j_l}e_{j_l}$ is a cycle 
in the simplex $\Omega_\alpha$ at $\alpha$. 
Hence, there exists a chain $g$ in $\Omega_\alpha$ of dimension $i-1$ such that 
\[
d(g)=\lambda_{j_1}e_{j_1}+\cdots+\lambda_{j_l}e_{j_l}.
\]
Replacing the basis element $f$ by $g$, we have that 
$\mbox{supp}(g)\subseteq A_\alpha$ and $g$ is called a Taylor basis element 
at $\alpha$. We do this replacement for all the basis elements in 
$(\mathbb{V}_\alpha, \partial)$ which are not contained in $\{e_1, \ldots, e_s\}$. 
Collecting the Taylor basis elements at $\alpha$, we get $\Gamma_\alpha$, and 
the exact complex $(\mathbb{V}_\alpha,\partial)$ with basis 
$\bigcup\limits_{\beta \leq \alpha}\Gamma_\beta$ is denoted by 
$(\mathbb{V}_\alpha, d)$. 
\end{construction}

In Theorem \ref{T:main1} we prove that $\mathbb{V}(L_M)$ is an $r$-frame and 
the $M$-homogenization of $\mathbb{V}(L_M)$ is a minimal free resolution of 
$S/M$ with a Taylor basis $\bigcup\limits_{m\in L_M} \Gamma_m$.

\begin{theorem}\label{T:main1}
Let $M$ be a monomial ideal in $S$ minimally generated by monomials 
$m_1, \ldots, m_r$. Let $\Gamma_m$, $(\mathbb{V}_m, d)$ and $(\mathbb{U}_m, d)$ 
be as in Construction \ref{C:VL} where the atomic lattice is the lcm-lattice $L_M$.  
Then we have the following results: 

for any $m\neq 1 \in L_M$,  any Taylor basis element in $\Gamma_m$ is 
a Taylor chain at $m$; and there exists a set $P_m$ of Taylor chains at 
$m$ such that if we apply the boundary map $d$ to to the chains in $P_m$ we 
obtain a direct sum of short trivial complexes, which is denoted by $(\mathcal{E}_m,d)$. 
Let $(\mathbb{W}_m,d)$ be the complex of $k$-vector spaces obtained by applying $d$ to the 
chains in $( \bigcup\limits_{\widetilde{m}\leq m}\Gamma_{\widetilde{m}} ) 
\bigcup ( \bigcup \limits_{1\neq \widetilde{m} \leq m}P_{\widetilde{m}} )$, 
then we have that $(\mathbb{W}_m, d)=(\mathbb{V}_m,d)\bigoplus 
(\bigoplus \limits_{1\neq \widetilde{m}\leq m}(\mathcal{E}_{\widetilde{m}},d))
\cong \widetilde{C}(\Omega_m; k)[-1]$. 
Let $(M_{\leq m})$ be the monomial ideal generated by all the $m_i$ such that 
$m_i$ divides $m$, then the $(M_{\leq m})$-homogenization of $(\mathbb{W}_m, d)$ 
is isomorphic to the Taylor resolution of $S/(M_{\leq m})$, and the 
$(M_{\leq m})$-homogenization of $(\mathbb{V}_m, d)$ is a minimal free 
resolution of $S/(M_{\leq m})$ with a Taylor basis 
$\bigcup\limits_{\widetilde{m}\leq m}\Gamma_{\widetilde{m}}$. 
In particular, the $M$-homogenization of $(\mathbb{V}(L_M), d)$ 
is a minimal free resolution of $S/M$ with a Taylor basis 
$\bigcup \limits_{m\in L_M} \Gamma_m$. 
Such a minimal free resolution of $S/M$ is called an \emph{atomic lattice resolution}
of $S/M$. 
\end{theorem}

\begin{proof}
We will prove this theorem by using the method of  strong induction in a poset.

Base case: By Construction \ref{C:VL} we have that 
$\Gamma_1=\{\emptyset\}$ and $\Gamma_{m_i}=\{\{i\}\}$ for any $1\leq i \leq r$. 
Hence, $(\mathbb{W}_{m_i},d)=(\mathbb{V}_{m_i},d)$ is the exact complex  
$0\to k \xrightarrow{1} k \to 0$ with a basis $\emptyset, \{i\}$; and then $P_{m_i}=\emptyset$. 
Note that $(M_{\leq m_i})=(m_i)$. The $(m_i)$-homogenization of 
$(\mathbb{W}_{m_i},d)$ is 
$0\to S(-m_i) \xrightarrow{m_i} S\to 0$, which is the Taylor resolution of 
$S/(m_i)$ and is a minimal free resolution of $S/(m_i)$ with a Taylor basis 
$\emptyset, \{i\}$. 

Inductive step: Let $m\in L_M$ with $\mbox{rk}(m)\geq 2$, then 
$|A_m| \geq 2$. Assume that the theorem holds for all 
$\widetilde{m} \neq 1 \in L_M$ with $\widetilde{m} <m$. 
By the inductive assumption, it is easy to see that if we apply the boundary 
map $d$ to the chains in 
$Q_m\bigcup (\bigcup\limits_{\widetilde{m}<m}\Gamma_{\widetilde{m}})
\bigcup (\bigcup \limits_{1\neq \widetilde{m}<m} P_{\widetilde{m}})$, 
we get a complex  isomorphic to $\widetilde{C}(\Omega_m;k)[-1]$. 
We denote this complex by $(\mathbb{H}_m, d)$.  
Then $(\mathbb{U}_m, d)\bigoplus 
(\bigoplus\limits_{1\neq \widetilde{m}<m}(\mathcal{E}_{\widetilde{m}},d))$ 
is a subcomplex of $(\mathbb{H}_m, d)$ with basis 
$ (\bigcup\limits_{\widetilde{m}<m}\Gamma_{\widetilde{m}})
\bigcup (\bigcup \limits_{1\neq \widetilde{m}<m} P_{\widetilde{m}})$,  
and is isomorphic to $\widetilde{C}(\Delta_m;k)[-1]$. 

Let $(\mathbb{V}_m,d)$ be the exact closure of $(\mathbb{U}_m, d)$ as 
in Construction \ref{C:VL}, then $(\mathbb{V}_m, d)$ can be viewed as an 
exact closure of $(\mathbb{U}_m, d)$ constructed in $(\mathbb{H}_m, d)$. 
Let $f\in \Gamma_m$, then $f$ is a basis element of $\mathbb{V}_m$ such that $f\notin \mathbb{U}_m$, 
so that there exist $\lambda_1, \ldots, \lambda_l \neq 0 \in k$ and 
$e_1, \ldots, e_l \in \bigcup \limits_{\widetilde{m}<m}\Gamma_{\widetilde{m}}$ 
such that 
\[
d(f)=\lambda_1e_1+\cdots +\lambda_le_l \in \mathbb{U}_m, 
\]
and there does not exist $g\in \mathbb{U}_m$ such that $d(g)=d(f)$. 
Let $f=\nu_1c_1+\cdots +\nu_sc_s$ where $\nu_1,\ldots, \nu_s$ are 
nonzero scalars in $k$ and $c_1, \ldots, c_s$ are some different faces 
of $\Omega_m$ of the same dimension. 
If there does not exist any $c_i$ such that $c_i\in Q_m$, then there exist 
$a_1, \ldots, a_p, b_1, \ldots, b_q \neq 0 \in k$, 
$f_1, \ldots, f_p \in \bigcup\limits_{\widetilde{m}<m}\Gamma_{\widetilde{m}}$, 
and $g_1, \ldots, g_q \in \bigcup\limits_{1\neq \widetilde{m}<m} P_{\widetilde{m}}$ 
such that 
\[
f=a_1f_1+\cdots+a_pf_p+b_1g_1+\cdots+b_qg_q.
\]
Since $f\notin \mathbb{U}_m$, it follows that $q\geq 1$. 
Note that by the inductive assumption the elements in  $P_{\widetilde{m}}$ 
appear in pairs. Namely, for any $h\in P_{\widetilde{m}}$ we have that 
either $d(h)=0$ or $d(h)\in P_{\widetilde{m}}$. 
If $d(g_1)=\cdots=d(g_q)=0$ then we let 
$g=a_1f_1+\cdots+a_pf_p \in \mathbb{U}_m$ which satisfies that 
$d(g)=d(f)$. This is a contradiction, so that $d(g_1), \ldots, d(g_q)$ are 
not all zero, and then 
\[
d(f)=a_1d(f_1)+\cdots +a_pd(f_p)+b_1d(g_1)+\cdots +b_qd(g_q) \notin \mathbb{U}_m,
\]
which is also a contradiction. 
Thus, there exists a $c_i$ such that $c_i \in Q_m$. So, $f$ is a Taylor chain 
at $m$. 

Let $\widehat{m}=\mbox{lcm}(\mbox{mdeg}(e_1), \ldots, \mbox{mdeg}(e_l))$, 
then we have that $\widehat{m} \leq m \in L_M$. Assume that 
$\widehat{m}<m$, then $\lambda_1e_1+\cdots+\lambda_le_l$ is a cycle 
in the exact complex $(\mathbb{V}_{\widehat{m}},d)$. Hence, there exists 
$g\in \mathbb{V}_{\widehat{m}} \subseteq \mathbb{U}_m$ such that 
$d(g)=\lambda_1e_1+\cdots+\lambda_le_l=d(f)$, which is a contradiction. 
Thus, we have that $\widehat{m}=m$. 
So the $(M_{\leq m})$-homogenization of $(\mathbb{V}_m,d)$ is a multigraded 
complex of free $S$-modules, which is denoted by $(\mathbf{F}_m, \partial)$. 
Note that $(\mathbf{F}_m,\partial)$ has a multigraded basis 
$\{\hbar(f)|f\in \bigcup\limits_{\widetilde{m}\leq m}\Gamma_m \}$, where 
$\hbar(f)$ is the homogenization of $f$. 
Since for any $1\neq \widetilde{m} \leq m$ the frame of the complex 
$\mathbf{F}_m(\leq \widetilde{m})$ is $(\mathbb{V}_{\widetilde{m}},d)$, 
which is exact, and obviously, $H_0(\mathbf{F}_m)=S/(M_{\leq m})$, 
it follows from Theorem 3.8 (2) in \cite{B:PV} that $\mathbf{F}_m$ is a 
free resolution of $S/(M_{\leq m})$. 
For any $f \in \Gamma_m$ with $d(f)=\lambda_1e_1+\cdots+\lambda_le_l$, 
we have that 
\[
\partial(\hbar(f))=\lambda_1\frac{\mbox{mdeg}(f)}{\mbox{mdeg}(e_1)}\hbar(e_1)
                             +\cdots +\lambda_l\frac{\mbox{mdeg}(f)}{\mbox{mdeg}(e_l)}\hbar(e_l).
\]
Since $\mbox{mdeg}(e_1),\ldots, \mbox{mdeg}(e_l)$ are 
strictly smaller than $\mbox{mdeg}(f)=m$ in $L_M$, 
it follows that the entries in the differential maps of $\mathbf{F}$ are in 
the maximal ideal $(x_1,\ldots, x_n)$ of $S$. Therefore, $(\mathbf{F}_m, \partial)$ 
is a minimal free resolution of $S/(M_{\leq m})$. 

Note that any face of $\Delta_m$ can be 
obtained as a linear combination of some chains in 
$ (\bigcup\limits_{\widetilde{m}<m}\Gamma_{\widetilde{m}})
\bigcup (\bigcup \limits_{1\neq \widetilde{m}<m} P_{\widetilde{m}})$. 
Since $(\mathbb{V}_m, d)\bigoplus 
(\bigoplus\limits_{1\neq \widetilde{m}<m}(\mathcal{E}_{\widetilde{m}},d))$ 
is a complex with basis 
$ (\bigcup\limits_{\widetilde{m} \leq m}\Gamma_{\widetilde{m}})
\bigcup (\bigcup \limits_{1\neq \widetilde{m}<m} P_{\widetilde{m}})$ 
and the Taylor basis elements in $\Gamma_m$ are Taylor chains at $m$, 
it follows that the set of chains $\{\mbox{in}(f)|f\in \Gamma_m\}$ 
is linearly independent over $k$. Hence, there exists a subset 
$\widetilde{Q}_m$ of $Q_m$ such that 
$ (\bigcup\limits_{\widetilde{m} \leq m}\Gamma_{\widetilde{m}})
\bigcup (\bigcup \limits_{1\neq \widetilde{m}<m} P_{\widetilde{m}})\bigcup \widetilde{Q}_m$ 
is a new basis of $(\mathbb{H}_m, d)$. 
We denote $(\mathbb{H}_m,d)$ with this new basis by $(\widetilde{\mathbb{H}}_m, d)$. 
It is easy to see that the $(M_{\leq m})$-homogenization of 
$(\widetilde{\mathbb{H}}_m, d)$ is isomorphic to the Taylor resolution 
of $S/(M_{\leq m})$. 
So by Proposition \ref{P:CCTB} we have that  
$\bigcup\limits_{\widetilde{m} \leq m}\Gamma_{\widetilde{m}}$ 
is a Taylor basis of $(\mathbf{F}_m, \partial)$. 

Similar to the proof of Proposition \ref{P:CCTB} we can use a series 
of consecutive cancellations in $(\widetilde{\mathbb{H}}_m, d)$ to 
the pairs in $\bigcup\limits_{1\neq \widetilde{m} \leq m} P_{\widetilde{m}}$, and 
by Proposition \ref{P:CCbasis} we see that the basis elements in 
$\bigcup\limits_{\widetilde{m} \leq m}\Gamma_{\widetilde{m}}$ remain unchanged 
and $\widetilde{Q}_m$ will be changed to $\widehat{Q}_m$ whose 
elements are some Taylor chains at $m$. 
Let $(\mathbb{G}_m, d)$ be the resulting exact complex with basis 
$(\bigcup\limits_{\widetilde{m} \leq m}\Gamma_{\widetilde{m}})\bigcup \widehat{Q}_m$. 
Then we can similarly use a series of consecutive cancellations in 
$(\mathbb{G}_m, d)$ to obtain $(\mathbb{V}_m, d)$ with basis 
$\bigcup\limits_{\widetilde{m} \leq m}\Gamma_{\widetilde{m}}$. 
In the process of these consecutive cancellations we will obtain 
some short trivial complexes with pairs of basis elements. 
Let $P_m$ be the set of these pair of basis elements, and then 
$ (\bigcup\limits_{\widetilde{m} \leq m}\Gamma_{\widetilde{m}})
\bigcup (\bigcup \limits_{1\neq \widetilde{m}<m} P_{\widetilde{m}})\bigcup P_m$ 
is a new basis of $(\widetilde{\mathbb{H}}_m, d)$ and consequently, 
a new basis of $\widetilde{C}(\Omega_m; k)[-1]$. 
Hence, we have that $|\Gamma_m|+|P_m|=|Q_m|$. 
Let $\widetilde{P}_m$ be the set of elements in $P_m$ which are 
Taylor chains at $m$. Then it is easy to see that 
$|Q_m|\leq |\Gamma_m|+|\widetilde{P}_m|$, which implies that 
$|\widetilde{P_m}|\geq |P_m|$, and then $\widetilde{P}_m=P_m$. 
Thus, every element in $P_m$ is a Taylor chain at $m$. 
Finally, let $(\mathcal{E}_m, d)$ be the complex obtained by applying  
$d$ to the chains in $P_m$, then $(\mathcal{E}_m, d)$ is a 
direct sum of some short trivial complexes; 
let $(\mathbb{W}_m,d)$ be the complex obatined by applying $d$ to the 
chains in $( \bigcup\limits_{\widetilde{m}\leq m}\Gamma_{\widetilde{m}} ) 
\bigcup ( \bigcup \limits_{1\neq \widetilde{m} \leq m}P_{\widetilde{m}} )$, 
then it is easy to see that $(\mathbb{W}_m, d)=(\mathbb{V}_m,d)\bigoplus 
(\bigoplus \limits_{1\neq \widetilde{m}\leq m}(\mathcal{E}_{\widetilde{m}},d))
\cong \widetilde{C}(\Omega_m; k)[-1]$. 
\end{proof}

\begin{remark}\label{R:main1remark1}
The notations of $\Gamma_m$, $P_m$, $(\mathbb{U}_m; d)$, 
$(\mathbb{V}_m; d)$ and $(\mathcal{E}_m; d)$ as in Theorem \ref{T:main1} 
will be used in the rest of the paper. 
Although there may be many different choices for $\Gamma_m$, 
by Theorem \ref{T:main1} we see that 
$|\Gamma_m|=\sum\limits_ib_{i,m}(S/M)$, which is unique. 
By \cite{B:Ph} we know that every atomic lattice can be realized as the 
lcm-lattice of some monomial ideal. Hence, it is easy to see that in 
Construction \ref{C:VL} every Taylor basis element in $\Gamma_\alpha$ 
is a Taylor chain at $\alpha$ and $|\Gamma_\alpha|$ is unique. 

In the proof of Theorem \ref{T:main1} we have that 
\[
\mbox{mdeg}(f)=\mbox{lcm}(\mbox{mdeg}(e_1), \ldots, \mbox{mdeg}(e_l))=m. 
\]
Thus, in Construction \ref{C:VL} we have that 
\[
\overline{\mbox{supp}(g)}=\overline{\overline{\mbox{supp}(e_{j_1})}\cup 
                                             \cdots \cup \overline{\mbox{supp}(e_{j_l})}}=A_\alpha.
\]
So by Corollary \ref{C:unionclosure} we have that 
\[
\overline{\mbox{supp}(g)}=\overline{\mbox{supp}(e_{j_1})\cup 
                                             \cdots \cup \mbox{supp}(e_{j_l})}=A_\alpha.
\]
Let $e_{j_1}, \ldots, e_{j_l}$ be the Taylor basis elements at 
$\beta_1, \ldots, \beta_l \in L$, respectively, then by Proposition \ref{P:meetjoin} 
we have that 
\[
\overline{A_{\beta_1} \cup \cdots \cup A_{\beta_l}}=A_{\beta_1 \vee \cdots \vee \beta_l}=A_\alpha,
\]
so that $\alpha=\beta_1 \vee \cdots \vee \beta_l$. 

Note that in general 
$\mbox{supp}(g)\neq \mbox{supp}(e_{j_1}) \cup \cdots \cup \mbox{supp}(e_{j_l})$. 
Let $M$ be as in Example \ref{E:TBmany}, then we can use Construction \ref{C:VL} 
to get $\Gamma_1={\emptyset}$, $\Gamma_{xy}=\{\{1\}\}$, $\Gamma_{xz}=\{\{2\}\}$, 
$\Gamma_{yz}=\{\{3\}\}$ and $\Gamma_{xyz}=\{\{1, 2\}, \{1,2\}+\{2,3\}\}$. 
Here we have that $d(\{1,2\}+\{2,3\})=-\{1\}+\{3\}$ and then
\[
\mbox{supp}(\{1,2\}+\{2,3\})=\{1,2, 3\} \supsetneq 
\mbox{supp}(\{1\})\cup \mbox{supp}(\{3\})=\{1, 3\}. 
\]
An example of $\mbox{supp}(g)\subsetneq \mbox{supp}(e_{j_1}) \cup \cdots \cup \mbox{supp}(e_{j_l})$ 
is shown in Example \ref{E:big}. 

Note that for any $\alpha, \beta \neq \hat{0} \in L$, $(\mathbb{V}_\alpha, d)$ is 
a subcomplex of $(\mathbb{V}_\beta, d)$ if and only if $\alpha \leq \beta$, which 
implies that the set of all the exact frames $(\mathbb{V}_\alpha, d)$ ordered by 
subcomplex is a poset isomorphic to $L-\{\hat{0}\}$. 
\end{remark}

\begin{remark}\label{R:main1remark2}
In general, finding a minimal free resolution of a graded ideal in $S$ is equivalent to 
repeatedly solving systems of linear equations over $S$, which is often  
difficult and involves the Gr\"ober basis theory. 
However, by Theorem \ref{T:main1} we see that finding a minimal free 
resolution of a monomial ideal is equivalent to obtaining exact closures of 
some complexes of $k$-vector spaces, which only requires solving some 
systems of linear equations over $k$. 
\end{remark}

By the proof of Theorem \ref{T:main1} we have the following corollary, 
which is useful for us to find the exact closure of $(\mathbb{U}_m, d)$ 
in the examples. 

\begin{corollary}\label{C:main1c1}
Let $m\neq 1 \in L_M$ such that $m$ is not an atom of $L_M$, then 
for any $i\geq 1$, 
\[
H_i(\mathbb{U}_m, d) \cong \widetilde{H}_{i-1}(\Delta_m; k),
\]
and then $|\Gamma_m|=\dim_k\widetilde{H}(\Delta_m; k)$. 
In particular, if $\widetilde{H}(\Delta_m; k)=0$ then $(\mathbb{U}_m, d)$ 
is exact and $(\mathbb{V}_m, d)=(\mathbb{U}_m, d)$. 
\end{corollary}

\begin{proof}
By the proof of Theorem \ref{T:main1} we have that 
\[
(\mathbb{U}_m, d) \bigoplus (\bigoplus \limits_{1\neq \widetilde{m} <m} 
(\mathcal{E}_{\widetilde{m}}, d))\cong \widetilde{C}(\Delta_m; k)[-1].
\]
Since $(\mathcal{E}_{\widetilde{m}}, d)$ is exact, it follows that 
for any $i\geq 1$, $H_i(\mathbb{U}_m, d) \cong \widetilde{H}_{i-1}(\Delta_m; k)$. 
\end{proof}

The next corollary will be used in Section 4. 

\begin{corollary}\label{C:main1c2}
Let $\beta_1, \ldots, \beta_t\in L_M$ such that  there do not 
exist $1\leq i \neq j \leq t$ with $\beta_i < \beta_j$. 
Let $\Delta=\langle A_{\beta_1}, \ldots, A_{\beta_t} \rangle$ be the simplicial complex 
with facets $A_{\beta_1}, \ldots, A_{\beta_t}$. 
Let $\Pi=\{m \in L_M| m \leq \beta_i \ \mbox{ for some} \ \beta_i\}$. 
Let $T=\bigcup\limits_{m \in \Pi} \Gamma_m$ and 
$P=\bigcup \limits_{1\neq m \in \Pi} P_m$. 
Let $(\mathbb{U}, d)$ be the complex obtained by applying the boundary 
map $d$ to the Taylor basis elements in $T$, and $(\mathcal{E}, d)$ be 
the exact complex obtained by applying $d$ to the elements in $P$. 
Then we have that $\widetilde{C}(\Delta;k)[-1]\cong (\mathbb{U}, d) \bigoplus (\mathcal{E}, d)$. 
\end{corollary}

\begin{proof}
By the proof of Theorem  \ref{T:main1}, we see that $(\mathbb{U}, d)$ 
is a complex, $(\mathcal{E}, d)$ is an exact complex, the chains in $T\cup P$
are  linearly independent over $k$, and every face of $\Delta$ can be 
written as a $k$-linear combination of some chains in $T\cup P$. 
Thus, $T\cup P$ is a new basis of $\widetilde{C}(\Delta; k)[-1]$, which implies 
that $\widetilde{C}(\Delta;k)[-1]\cong (\mathbb{U}, d) \bigoplus (\mathcal{E}, d)$.
\end{proof}

\begin{example}\label{E:big}
Let $M=(x_1x_2, x_2x_3, x_3x_4, x_4x_5, x_5x_6, x_1x_6)$ be the edge 
ideal of the hexagon as in Example 2.14 of \cite{B:CM2}. 
We will use Construction \ref{C:VL} to obtain a Taylor basis for some minimal 
free resolution of $S/M$. 
As mentioned in Subsection 2.4, in the examples of this paper we will use 
$A_m$ to label the element $m$ in $L_M$, and for simplicity we will write a 
face or a set $\{i_1, \ldots, i_t\}$  in abbreviation as $i_1\cdots i_t$. 
For example, we write $123+124$ to represent the chain $\{1,2,3\}+\{2,3,4\}$ 
and write $k(123+124)$ to represent the $k$-vector space with a basis element 
$\{1,2,3\}+\{2,3,4\}$. Similarly, we will use $\Delta_{A_m}, \mathbb{U}_{A_m}, 
\mathbb{V}_{A_m}$ to denote $\Delta_m, \mathbb{U}_m, \mathbb{V}_m$, 
respectively. 

\begin{figure}
\begin{tikzpicture}[node distance=1.4cm]
\title{The lcm-lattice}
\node(123456)                                                                  {$123456$};
\node(1256)    [below left =1cm and -0.3cm of 123456]   {$1256$};
\node(1456)    [below right=1cm and -0.3cm of 123456]  {$1456$};
\node(1236)    [left of =1256]                                          {$1236$};
\node(1234)    [left of =1236]                                          {$1234$};
\node(2345)    [right of =1456]                                        {$2345$};
\node(3456)    [right of =2345]                                        {$3456$};
\node(123)      [below of =1234]	                                     {$123$};
\node(126)      [below of =1236]	                                     {$126$};
\node(156)      [below of =1256]	                                     {$156$};
\node(234)      [below of =1456]	                                     {$234$};
\node(345)      [below of =2345]	                                     {$345$};
\node(456)      [below of =3456]	                                     {$456$};
\node(25)        [below=4cm of 123456]                            {$25$};
\node(23)        [left of =25]                                              {$23$};
\node(16)        [left of =23]                                              {$16$};
\node(14)        [left of =16]                                              {$14$};
\node(12)        [left of =14]                                              {$12$};
\node(34)        [right of =25]                                            {$34$};
\node(36)        [right of =34]                                            {$36$};
\node(45)        [right of =36]                                            {$45$};
\node(56)        [right of =45]                                            {$56$};
\node(1)          [below=2.5cm of 123]                                  {$1$};
\node(2)          [below=2.5cm of 126]                                  {$2$};
\node(3)          [below=2.5cm of 156]                                  {$3$};
\node(4)          [below=2.5cm of 234]                                  {$4$};
\node(5)          [below=2.5cm of 345]                                  {$5$};
\node(6)          [below=2.5cm of 456]                                  {$6$};
\node(0)          [below=6.5cm of 123456]                            {$\emptyset$};

\draw(123456) --(1234);
\draw(123456) --(1236);
\draw(123456) --(1256);
\draw(123456) --(1456);
\draw(123456) --(2345);
\draw(123456) --(3456);
\draw(1234) --(123);
\draw(1234) --(234);
\draw(1234) --(14);
\draw(1236) --(123);
\draw(1236) --(126);
\draw(1236) --(36);
\draw(1256) --(126);
\draw(1256) --(156);
\draw(1256) --(25);
\draw(1456) --(156);
\draw(1456) --(456);
\draw(1456) --(14);
\draw(2345) --(234);
\draw(2345) --(345);
\draw(2345) --(25);
\draw(3456) --(345);
\draw(3456) --(456);
\draw(3456) --(36);
\draw(123) --(12);
\draw(123) --(23);
\draw(126) --(12);
\draw(126) --(16);
\draw(156) --(16);
\draw(156) --(56);
\draw(234) --(23);
\draw(234) --(34);
\draw(345) --(34);
\draw(345) --(45);
\draw(456) --(45);
\draw(456) --(56);
\draw(12) --(1);
\draw(12) --(2);
\draw(14) --(1);
\draw(14) --(4);
\draw(16) --(1);
\draw(16) --(6);
\draw(23) --(2);
\draw(23) --(3);
\draw(25) --(2);
\draw(25) --(5);
\draw(34) --(3);
\draw(34) --(4);
\draw(36) --(3);
\draw(36) --(6);
\draw(45) --(4);
\draw(45) --(5);
\draw(56) --(5);
\draw(56) --(6);
\draw(1) --(0);
\draw(2) --(0);
\draw(3) --(0);
\draw(4) --(0);
\draw(5) --(0);
\draw(6) --(0);

\end{tikzpicture}
\caption{The lcm-lattice $L_M$ of Example \ref{E:big}}
\end{figure}
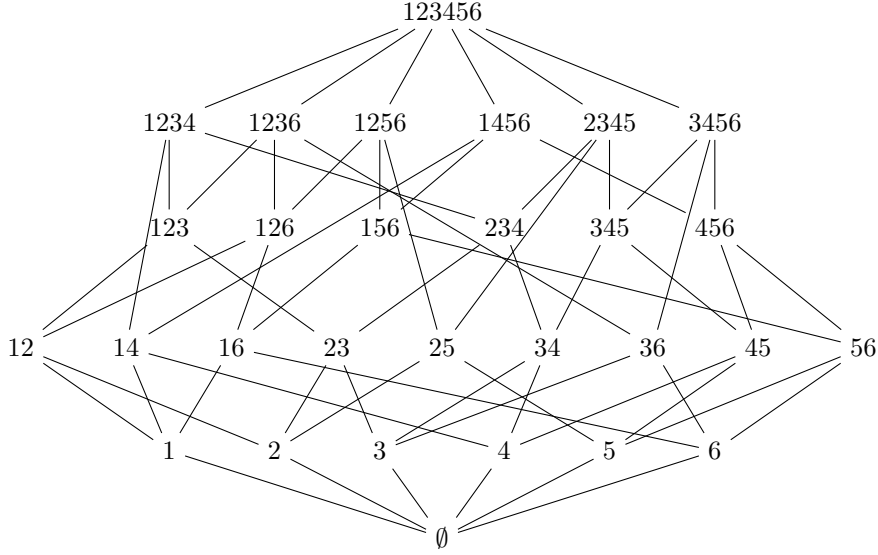

The lcm-lattice $L_M$ of $M$ is shown in Figure 1.
It is easy to see that we can take 
$\Gamma_{\emptyset}=\{\emptyset\}$, 
$\Gamma_1=\{1\}, \ldots, \Gamma_6=\{6\}$. 
Since $\Delta_{12}=\langle 1,2 \rangle$ and $\widetilde{H}_0(\Delta_{12};k)\cong k$ 
with a basis element $[-1+2]$, by Corollary \ref{C:main1c1} we see that 
$\mathbb{U}_{12}$ has a cycle  $-1+2$, and then by Construction \ref{C:VL} 
we can take $\Gamma_{12}=\{12\}$. 
Similarly, we have that $\Gamma_{14}=\{14\},\ldots, \Gamma_{56}=\{56\}$. 
Since $\Delta_{123}=\langle 12, 23 \rangle$ and $\widetilde{H}(\Delta_{123};k)=0$, 
by Corollary \ref{C:main1c1} we have that $\Gamma_{123}=\emptyset$. 
Similarly, we have that $\Gamma_{126}=\cdots=\Gamma_{456}=\emptyset$. 
Since $\Delta_{1234}=\langle 123, 14, 234 \rangle$, it follows that 
$\widetilde{H}_1(\Delta_{1234}; k)\cong k$ with a basis element 
$[12-14+24]$. Since $(\mathbb{U}_{1234}, d)$ has basis elements 
$12, 14, 23, 34$ in homological degree $2$, by Corollary \ref{C:main1c1} 
we have that $H_2(\mathbb{U}_{1234}, d)\cong k$ and a corresponding 
cycle in $\mathbb{U}_{1234}$ is $12-14+23+34$.  
Note that $d(123+134)=12-14+23+34$. Hence, by Construction \ref{C:VL} 
we let $\Gamma_{1234}=\{123+134\}$.  
Observe that here we actually have many choices for $\Gamma_{1234}$. 
For example, we can also let $\Gamma_{1234}=\{124+234\}$ or let 
$\Gamma_{1234}=\{\frac{1}{2}(123+124+134+234)\}$. 
Similarly, we get $\Gamma_{1236}=\{123+136\}$, 
$\Gamma_{1256}=\{125+156\}$, $\Gamma_{1456}=\{146+456\}$, 
$\Gamma_{2345}=\{234+245\}$, $\Gamma_{3456}=\{346+456\}$. 
Finally, from $\Delta_{123456}=\langle 1234, 1236, 1256, 1456, 2345, 3456 \rangle$ 
we have that $\widetilde{H}_2(\Delta_{123456}; k) \cong k^2$ with a 
basis consisting of $[346-146+136-134]$ and $[245-145+125-124]$. 
Corresponding to these two basis elements, in $\mathbb{U}_{123456}$ 
we have two cycles $(346+456)-(146+456)+(123+136)-(123+134)$ 
and $(234+245)-(146+456)+(125+156)-(123+134)$, from which we 
get $\Gamma_{123456}=\{1346, 1245-1456+1234\}$. 
So, $\emptyset, 1,2,3, 4, 5,6, 12, 14, 16, 23, 25, 34, 36, 45, 56, 
123+134, 123+136, 125+156, 146+456, 234+245, 346+456, 
1346, 1245-1456+1234$ is a Taylor basis of some minimal free resolution 
of $S/M$. 

Note that in this example we have that 
\[
d(1346)=(346+456)-(146+456)+(123+136)-(123+134),
\]
but as mentioned in Remark \ref{R:main1remark1} we have that 
\[
\mbox{supp}(1346)\subsetneq \mbox{supp}(346+456)\cup \mbox{supp}(146+456) 
 \cup \mbox{supp}(123+136) \cup \mbox{supp}(123+134).
\]
\end{example}

In the next theorem we prove that every minimal free resolution of $S/M$ 
is an atomic lattice resolution. 

\begin{theorem}\label{T:main2}
Let $M$ be a monomial ideal in $S$ minimally generated by monomials 
$m_1, \ldots, m_r$. Let $\mathbf{F}$ be a minimal free resolution of 
$S/M$ with a Taylor basis $B$. For any $m\in L_M$, 
let $\widetilde{\Gamma}_m=\{f\in B| \mbox{mdeg}(f)=m\}$. 
For any $m\neq 1 \in L_M$, let $(\widetilde{\mathbb{V}}_m, d)$ be 
the frame of $\mathbf{F}(\leq m)$ with basis 
$\bigcup\limits_{\widetilde{m}\leq m}\widetilde{\Gamma}_{\widetilde{m}}$ 
and $(\widetilde{\mathbb{U}}_m, d)$ the frame of $\mathbf{F}(<m)$ 
with basis $\bigcup\limits_{\widetilde{m}< m}\widetilde{\Gamma}_{\widetilde{m}}$. 
Then in Construction \ref{C:VL} we can take $\Gamma_m=\widetilde{\Gamma}_m$  
for any $m\in L_M$ and then we have that $(\mathbb{V}_m, d)=(\widetilde{\mathbb{V}}_m, d)$ 
for any $m \neq 1 \in L_M$. In particular, $\mathbf{F}$ is an atomic 
lattice resolution of $S/M$. 
\end{theorem}

\begin{proof}
We will prove this theorem by using strong induction in $L_M$. 

Base case: Without the loss of generality we can assume that 
$\widetilde{\Gamma}_1=\{\emptyset\}$ and 
$\widetilde{\Gamma}_{m_i}=\{\{i\}\}$ for any $1\leq i \leq r$. 
Then in Construction \ref{C:VL} we can take 
$\Gamma_1=\widetilde{\Gamma}_1$ and 
$\Gamma_{m_i}=\widetilde{\Gamma}_{m_i}=\{\{i\}\}$ for any $1\leq i \leq r$, 
which implies that $(\mathbb{V}_{m_i}, d)=(\widetilde{\mathbb{V}}_{m_i},d)$. 

Inductive step: Let $m\in L_M$ with $\mbox{rk}(m)\geq 2$. Assume that 
in Construction \ref{C:VL} we have obtained $\Gamma_{\widetilde{m}}$ and 
$(\mathbb{V}_{\widetilde{m}}, d)$ such that 
$\Gamma_{\widetilde{m}}=\widetilde{\Gamma}_{\widetilde{m}}$ for any 
$\widetilde{m}<m \in L_M$ and 
$(\mathbb{V}_{\widetilde{m}}, d)=(\widetilde{\mathbb{V}}_{\widetilde{m}}, d)$ 
for any $1\neq \widetilde{m} < m \in L_M$. 
Hence, in Construction \ref{C:VL} we have that 
$(\mathbb{U}_m, d)=(\widetilde{\mathbb{U}}_m, d)$. 
Next we want to show that $(\widetilde{\mathbb{V}}_m, d)$ is an 
exact closure of $(\mathbb{U}_m, d)$. 

Indeed, $(\mathbb{U}_m, d)$ is a subcomplex of the exact complex 
$(\widetilde{\mathbb{V}}_m, d)$. Hence, by Remark \ref{R:exactclosure2} 
there exists an exact closure of $(\mathbb{U}_m, d)$ constructed in 
$(\widetilde{\mathbb{V}}_m, d)$, which we denote by $(\mathbb{G}_m, d)$. 
Let $f\notin \mathbb{U}_m$ be a basis element of $\mathbb{G}_m$ such that 
$d(f)=\lambda_1e_1+\cdots+\lambda_le_l$, where $\lambda_1, \ldots, \lambda_l$ 
are some nonzero scalars in $k$ and $e_l, \ldots, e_l$ are some different 
Taylor basis elements in $\bigcup\limits_{\widetilde{m}<m}\Gamma_{\widetilde{m}}$. 
Let $\widehat{m}=\mbox{lcm}(\mbox{mdeg}(e_1), \ldots, \mbox{mdeg}(e_l))$, 
then it is easy to see that $\widehat{m}|m$. 
Assume that $\widehat{m}<m$ in $L_M$. Then $\lambda_1e_1+\cdots+\lambda_le_l$ 
is a cycle in the exact complex $(\mathbb{V}_{\widehat{m}}, d)$, so that 
there exists $g\in \mathbb{V}_{\widehat{m}} \subseteq \mathbb{U}_m$ such 
that $d(g)=\lambda_1e_1+\cdots+\lambda_le_l$.
Hence, $\lambda_1e_1+\cdots+\lambda_le_l$ is a trivial cycle in $\mathbb{U}_m$, 
which is a contradiction. Thus, we have that $\widehat{m}=m$. 
Note that $(\mathbb{U}_m)_1=(\mathbb{G}_m)_1=(\widetilde{\mathbb{V}}_m)_1$. 
Let $\mathbf{G}_m$ be the $(M_{\leq m})$-homogenization of $\mathbb{G}_m$. 
Then for any $1\neq \widetilde{m} <m \in L_M$ the frame of 
$\mathbf{G}_m(\leq \widetilde{m})$ is the exact complex $\mathbb{V}_{\widetilde{m}}$, 
and the frame of $\mathbf{G}_m(\leq m)$ is the exact complex $\mathbb{G}_m$. 
Thus, by Theorem 3.8 (2) in \cite{B:PV} we have that $\mathbf{G}_m$ is a 
free resolution of $S/(M_{\leq m})$. 
Since $\mathbf{F}(\leq m)$ is a minimal free resolution of $S/(M_{\leq m})$, 
it follows that 
\[\dim_k(\widetilde{\mathbb{V}}_m)_i=\mbox{rank}(\mathbf{F}(\leq m)_i
\leq \mbox{rank}(\mathbf{G}_m)_i=\dim_k(\mathbb{G}_m)_i, \ \mbox{for any} \  i\geq 0, 
\]
so that $\dim_k(\widetilde{\mathbb{V}}_m)_i=\dim_k(\mathbb{G}_m)_i$ 
for all $i$ and then by Remark \ref{R:exactclosure1} we see that 
$(\widetilde{\mathbb{V}}_m, d)$ is an exact closure of $(\mathbb{U}_m, d)$. 

Note that for any $f\in \widetilde{\Gamma}_m$, $f$ is a Taylor chain at $m$, 
which implies that $\mbox{supp}(f) \subseteq A_m$, so that in Construction 
\ref{C:VL} we can take $\Gamma_m=\widetilde{\Gamma}_m$ and then 
we have that $(\mathbb{V}_m, d)=(\widetilde{\mathbb{V}}_m, d)$.
\end{proof}

For convenience, we call the results in this subsection, including 
Construction \ref{C:VL}, Theorem \ref{T:main1}, Theorem \ref{T:main2} 
and their proofs, \emph{the atomic lattice resolution theory}.

\subsection{Some Applications of the Atomic Lattice Resolution Theory}
Note that Theorem \ref{T:main2} and Corollary \ref{C:main1c1} imply 
Theorem \ref{T:basis} in section 2.4. To some extend, the next theorem 
can be viewed as a converse of Theorem \ref{T:basis}. 

\begin{theorem}\label{T:cbasis}
Let $M$ be a monomial ideal minimally generated by monomials $m_1,\ldots, m_r$. 
Let $e_1^1=\emptyset$ and $e_1^{m_i}=\{i\}$ for any $1\leq i \leq r$. 
And for any $m\in L_M$ with $\mbox{rk}(m)\geq 2$, 
let $e_1^m, \ldots, e_{t_m}^m$ be a set of chains in $\Omega_m$ 
such that $[d(e_1^m)], \ldots, [d(e_{t_m}^m)]$ is a basis of 
$\widetilde{H}(\Delta_m;k)$. 
Let $B=\{e_i^m|m\in L_M, i\geq 1\}$. 
Assume that for any $m\neq 1 \in L_M$, $d(e_i^m)$ can be written as 
a $k$-linear combination of some $e_j^{\widetilde{m}}$ with $\widetilde{m}<m$, 
and let $(\mathbb{V}, d)$ be the complex of $k$-vector spaces obtained by 
applying $d$ to the elements in $B$. 
Then the $M$-homogenization of $(\mathbb{V}, d)$ is a minimal free 
resolution of $S/M$ with a Taylor basis $B$. 
\end{theorem}

\begin{proof}
For any $m\in L_M$ let $\widetilde{\Gamma}_m=\{e_i^m \in B| i\geq 1\}$. 
For any $m\neq 1 \in L_M$ let $(\widetilde{\mathbb{U}}_m, d)$ be the 
complex obtained by applying $d$ to the elements in 
$\bigcup\limits_{\widetilde{m}<m}\widetilde{\Gamma}_{\widetilde{m}}$, 
and let $(\widetilde{\mathbb{V}}_m, d)$ be the complex obtained by applying 
$d$ to the elements in 
$\bigcup \limits_{\widetilde{m} \leq m}\widetilde{\Gamma}_{\widetilde{m}}$. 
It suffices to show that for any $m\neq 1 \in L_M$ the 
$(M_{\leq m})$-homogenization of $\widetilde{\mathbb{V}}_m$ is a minimal 
free resolution of $S/(M_{\leq m})$ with a Taylor basis 
$\bigcup \limits_{\widetilde{m} \leq m}\widetilde{\Gamma}_{\widetilde{m}}$.
We will prove this result by using strong induction in $L_M$. 

Base case: if $m=m_i$ is an atom then $(\widetilde{\mathbb{V}}_{m_i}, d)$ 
is $0\to k\xrightarrow{1} k \to 0$, and the $(M_{\leq m_i})$-homogenization 
of $\widetilde{\mathbb{V}}_{m_i}$ is $0 \to S(-m_i) \xrightarrow{m_i} S \to 0$, 
which is a minimal free resolution of $S/(M_{\leq m_i})$ with a Taylor basis 
$\emptyset, \{i\}$. 

Inductive step: let $m\in L_M$ with $\mbox{rk}(m)\geq 2$. Assume that 
the result holds for all $\widetilde{m}\neq 1 \in L_M$ with $\widetilde{m}<m$.  
Then by Theorem \ref{T:main2} it is easy to see that in Construction \ref{C:VL} 
we can take $\Gamma_{\widetilde{m}}=\widetilde{\Gamma}_{\widetilde{m}}$ 
for all $\widetilde{m}<m$, and then $(\mathbb{U}_m, d)=(\widetilde{\mathbb{U}}_m, d)$. 
Similar to the proof of Corollary \ref{C:main1c1}, we have that 
\[
\widetilde{C}(\Delta_m; k)[-1]=(\mathbb{U}_m, d) \bigoplus (\bigoplus \limits_{1\leq \widetilde{m} <m} 
(\mathcal{E}_{\widetilde{m}}, d)),
\]
where $(\mathcal{E}_{\widetilde{m}}, d)$ is a trivial complex. 
Since  $[d(e_1^m)], \ldots, [d(e_{t_m}^m)]$ is a basis of $\widetilde{H}(\Delta_m;k)$ 
and  $d(e_i^m)\in \mathbb{U}_m$ for any $1\leq i \leq t_m$, 
it follows that $[d(e_1^m)], \ldots, [d(e_{t_m}^m)]$ is a basis of 
$H(\mathbb{U}_m,d)$. Note that $\mbox{supp}(e_i^m) \subseteq A_m$. 
Hence, in Construction \ref{C:VL} we can take $\Gamma_m=\widetilde{\Gamma}_m$ 
and then $(\mathbb{V}_m, d)=(\widetilde{\mathbb{V}}_m, d)$. So by 
Theorem \ref{T:main1} the $(M_{\leq m})$-homogenization of 
$\widetilde{\mathbb{V}}_m$ is a minimal free resolution of 
$S/(M_{\leq m})$ with a Taylor basis 
$\bigcup \limits_{\widetilde{m} \leq m}\widetilde{\Gamma}_{\widetilde{m}}$.
\end{proof}

\begin{remark}\label{R:cbasis}
By Theorem \ref{T:cbasis} and its proof, it is easy to see that in Construction 
\ref{C:VL} we can take $\Gamma_m=\{A_m\}$ for all Scarf multidegrees $m\in L_M$. 
In Example \ref{E:big} we have seen how Corollary \ref{C:main1c1} can help 
us to find $\Gamma_m$. Theorem \ref{T:cbasis} goes further and gives a 
geometric method to find $\Gamma_m$ step by step. As is shown by the 
next example, this method is very handy for us to calculate examples. 
\end{remark}

\begin{example}\label{E:cbasis}
Let $M$ be a monomial ideal in $S=k[a,b, c,d]$ generated by monomials 
$m_1=a^2b, m_2=ac, m_3=ad, m_4=bcd$. 

First we use this example to illustrate how one can obtain the lcm-lattice of 
a given monomial ideal. Set 
\[
\mathcal{L}=\{12, 13, 14, 23, 24, 34, 123, 124, 134,234, 1234\}, 
\]
where $12$ means the set $\{1, 2\}$ as in Example \ref{E:big}. 
Since $\mbox{lcm}(m_1, m_2)=a^2bc$ and neither of $m_3$ and $m_4$  
divides $a^2bc$, it follows that $A_{a^2bc}=12$. 
Similarly, we have that $A_{a^2bd}=13$. 
Since $\mbox{lcm}(m_1, m_4)=a^2bcd$ and both of $m_2$ and $m_3$ 
divide $a^2bcd$, it follows that $A_{a^2bcd}=1234$, which 
means that whenever $1$ and $4$ appear in some $A_m$, $2$ and $3$ must also 
be in $A_m$. Hence, we delete $14, 124, 134$ from $\mathcal{L}$ and set 
\[
\mathcal{L}=\{12, 13, 1234, 23, 24, 34,  123, 234 \}. 
\]
Then we have that $A_{acd}=23$. 
Since $\mbox{lcm}(m_2,m_4)=abcd$ and $m_3$ divides $abcd$, 
we have that $A_{abcd}=234$ and $24$ is deleted. Then we set
\[ 
\mathcal{L}=\{12, 13, 1234, 23, 234, 34, 123\}. 
\]
Similarly, $34$ and $123$ will be 
deleted and in the end we will get 
\[
\mathcal{L}=\{12, 13, 23, 234, 1234\}, 
\]
from 
which we can draw the lcm-lattice $L_M$ as shown in Figure 2.

\begin{figure}
\begin{tikzpicture}[node distance=1.4cm]
\title{The lcm-lattice2}
\node(1234)                                                                    {$1234$};
\node(234)    [below right=0.8cm and 0.6cm of 1234]      {$234$};
\node(23)      [below left =0.8cm and -0.1cm of 234]         {$23$};
\node(13)      [left of =23]                                               {$13$};
\node(12)      [left of =13]                                               {$12$};
\node(1)        [below of =12]	                                    {$1$};
\node(2)        [below of =13]	                                    {$2$};
\node(3)        [below of =23]	                                    {$3$};
\node(4)        [right of =3]	                                    {$4$};
\node(0)        [below=5cm of 1234]                                 {$\emptyset$};

\draw(1234) --(234);
\draw(1234) --(12);
\draw(1234) --(13);
\draw(234) --(23);
\draw(234) --(4);
\draw(12) --(1);
\draw(12) --(2);
\draw(13) --(1);
\draw(13) --(3);
\draw(23) --(2);
\draw(23) --(3);
\draw(1) --(0);
\draw(2) --(0);
\draw(3) --(0);
\draw(4) --(0);

\end{tikzpicture}
\caption{The lcm-lattice $L_M$ of Example \ref{E:cbasis}}
\end{figure}
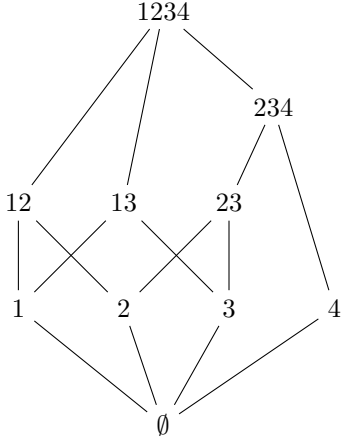

Now we have that $\Delta_{234}=\langle 23, 4 \rangle$ 
and $\widetilde{H}_0(\Delta_{234}; k)\cong k$ with a basis $[-2+4]$; 
$\Delta_{1234}=\langle 12, 13, 234 \rangle $ and 
$\widetilde{H}_1(\Delta_{1234};k)\cong k$ with a basis 
$[12-13+23]$. Hence, by Theorem \ref{T:cbasis} we can take 
$\Gamma_{234}=\{24\}$ and $\Gamma_{1234}=\{123\}$. 
Thus, we have that $B=\{\emptyset, 1,2,3,4, 12, 13, 23, 24, 123\}$. 
By applying $d$ to the faces in $B$ we get the complex $(\mathbb{V},d)$
with basis $B$:
\[
\begin{tikzcd}[ampersand replacement=\&]
0 \ar[r] \&k  \ar[r, "{\begin{pmatrix} 1 \\ -1\\ 1\\ 0  \end{pmatrix}}"]
\&[1.5em]k^4 
\arrow[r, "{\begin{pmatrix} -1& -1&0 & 0 \\ 1& 0& -1& -1 \\ 0 & 1& 1& 0\\ 0& 0& 0& 1 \end{pmatrix}}"]
\& [6.5em]k^4 \arrow[r, "{\begin{pmatrix} 1& 1& 1& 1\end{pmatrix}}"]
\& [4em]k \ar[r]
\& 0.
\end{tikzcd}
\]
The $M$-homogenization of $\mathbb{V}$ is a minimal free resolution 
of $S/M$ with a Taylor basis $B$:
\[
\begin{tikzcd}[ampersand replacement=\&]
0\to  S  \ar[r, "{\begin{pmatrix} d \\ -c\\ ab \\0  \end{pmatrix}}"] 
\& [1.1em] S^4 \arrow[r, "{\begin{pmatrix} -c& -d& 0& 0 \\ ab&  0& -d& -bd\\
0& ab& c & 0 \\ 0& 0& 0&a \end{pmatrix} }"]
\& [8em] S^4 \arrow[r, "{\begin{pmatrix} a^2b & ac& ad &bcd \end{pmatrix} }"]
\& [7em] S.
\end{tikzcd}
\]
This example will be used in Section 4.2.
\end{example}

We can always use the lcm-lattice $L_M$ and Construction \ref{C:VL} to obtain a Taylor 
basis, which induces a minimal free resolution of $S/M$. 
However, if a minimal free resolution $\mathbf{F}$ of $S/M$ is already given, we can use 
the next construction to get a Taylor basis of $\mathbf{F}$, which does not 
involve the lcm-lattice $L_M$. 

\begin{construction}\label{C:basis3}
Let $M$ be a monomial ideal minimally generated by monomials $m_1, \ldots, m_r$. 
Let $\mathbf{F}$ be a minimal free resolution of $S/M$. Let $(\mathbb{V}, \partial)$ 
be the frame of $\mathbf{F}$ with $A_i$ the matrix of  $\partial_i$. 
Let $\Omega$ be the simplex with the vertex set $\{1, \ldots, r\}$ and $d$ the 
boundary map of $\Omega$. 
In this construction we will use induction on the homological degree $i$ to obtain 
a set $B_i$  such that every element in $B_i$ is an $(i-1)$-dimensional chain in 
$\Omega$; 
let $B$ be the union of all $B_i$, let $(\mathbb{W}, d)$ be the complex of 
$k$-vector spaces obtained by applying $d$ to the chains in $B$, 
let $C_i$ be the matrix of $d_i$, then we have that $C_i=A_i$ for all $i$, 
which implies that $(\mathbb{W}, d)=(\mathbb{V}, \partial)$.  
In the next proposition we will prove that $B$ is a Taylor basis of $\mathbf{F}$. 

Base case: without the loss of generality we can assume that the map $F_1 \to F_0$ 
is given by $S^r \xrightarrow{(\begin{smallmatrix} m_1& \cdots & m_r  \end{smallmatrix})} S$. 
Then we set $B_0=\{\emptyset\}$  and $B_1=\{\{1\}, \ldots, \{r\}\}$ and then 
$C_1=A_1=\begin{pmatrix} 1& \cdots & 1 \end{pmatrix}$. 

Inductive step: Let $i\geq 2$. Assume that for all $0\leq j <i$ we have obtained $B_j$ 
and we have that $C_j=A_j$. Then we want to construct $B_i$ and show that 
$C_i=A_i$. 

Let $A_i=\begin{pmatrix} \alpha_1 & \cdots &\alpha_q \end{pmatrix}$ be a 
$p\times q$ matrix. Since $C_{i-1}=A_{i-1}$, we have that $|B_{i-1}|=\dim_kV_{i-1}=p$. 
Let $B_{i-1}=\{g_1, \ldots, g_p\}$ and $B_{i-2}=\{h_1, \ldots, h_t\}$. Let the column vector 
$\alpha_1=\begin{pmatrix} a_1 & \cdots & a_p \end{pmatrix}^T$. 
Let $a_{i_1}, \ldots, a_{i_l}$ be the nonzero scalars among 
$a_1, \ldots, a_p$. Since $A_{i-1}A_i=0$, it follows that 
$A_{i-1}\alpha_1=0$ and then we have that
\begin{align*}
d_{i-1}(a_{i_1}g_{i_1}+\cdots+a_{i_l}g_{i_l})
               &=d_{i-1}(a_1g_1+\cdots+a_pg_p) \\
               &=d_{i-1}\left( \begin{pmatrix} g_1& \cdots &g_p \end{pmatrix} 
                    \begin{pmatrix} a_1 \\ \vdots \\ a_p \end{pmatrix} \right)  \\
               &=\begin{pmatrix} h_1& \cdots &h_t \end{pmatrix} C_{i-1}
                     \begin{pmatrix} a_1 \\ \vdots \\ a_p \end{pmatrix}   \\
               &=\begin{pmatrix}h_1& \cdots &h_t \end{pmatrix} A_{i-1}\alpha_1 \\
              &=0. 
\end{align*}
Hence, $a_{i_1}g_{i_1}+\cdots+a_{i_l}g_{i_l}$ is a cycle in $\Omega$. 
Let $\widetilde{\Omega}$ be the simplex with the vertex set 
$\overline{\mbox{supp}(g_{i_1})\cup \cdots \cup \mbox{supp}(g_{i_l})}$. 
Then there is a chain $f_1$ of dimension $i-1$ in $\widetilde{\Omega}$ such that 
$d(f_1)=a_{i_1}g_{i_1}+\cdots+a_{i_l}g_{i_l}=a_1g_1+\cdots+a_pg_p$. 
Similarly, we can obtain $(i-1)$-dimensional chains $f_2, \ldots, f_q$. 
Let $B_i=\{f_1, \ldots, f_q\}$, then it is easy to see that $C_i=A_i$. 
\end{construction}

\begin{theorem}\label{T:basis3}
In Construction \ref{C:basis3}, $B$ is a Taylor basis of $\mathbf{F}$. 
And conversely, every Taylor basis of $\mathbf{F}$ can be obtained by 
Construction \ref{C:basis3}. 
\end{theorem}

\begin{proof}
Since we have that $(\mathbb{W}, d)=(\mathbb{V}, \partial)$, by Theorem 4.14 in \cite{B:PV} 
it follows that the $M$-homogenization of $\mathbb{W}$ is $\mathbf{F}$. 
For any $m\neq 1 \in L_M$, let $(\mathbb{W}_{\leq m}, d)$ be the frame of 
$\mathbf{F}(\leq m)$ and $(\mathbb{W}_{<m}, d)$ the frame of $\mathbf{F}(<m)$. 
By Theorem \ref{T:main2} we see that $(\mathbb{W}_{\leq m}, d)$ is an 
exact closure of $(\mathbb{W}_{<m}, d)$. 
Let $\widehat{\Gamma}_1=\{\emptyset\}$ and for any $m\neq 1 \in L_M$, 
let $\widehat{\Gamma}_m=\{f\in B| f\in \mathbb{W}_{\leq m}, f\notin \mathbb{W}_{<m} \}$. 
Next we prove by using strong induction in $L_M$ that in Construction \ref{C:VL} 
we can take $\Gamma_m=\widehat{\Gamma}_m$ for all $m\in L_M$. 

Base case: obviously, we have that $\Gamma_1=\{\emptyset\}=\widehat{\Gamma}_1$ 
and $\Gamma_{m_i}=\{\{i\}\}=\widehat{\Gamma}_{m_i}$ for any $1\leq i \leq r$.   

Inductive step: Let $m\in L_M$ with $\mbox{rk}(m)\geq 2$. Assume that for any 
$\widetilde{m}<m \in L_M$ we have taken $\Gamma_{\widetilde{m}}=\widehat{\Gamma}_{\widetilde{m}}$ 
in Construction \ref{C:VL}. Then $(\mathbb{W}_{< m}, d)$ is equal to 
$(\mathbb{U}_m, d)$ in Construction \ref{C:VL} and we can set 
$(\mathbb{V}_m, \partial)=(\mathbb{W}_{\leq m}, d)$ in Construction \ref{C:VL}. 
For any $f\in \widehat{\Gamma}_m$, let $d(f)=\lambda_1e_1+\cdots+\lambda_le_l$, 
where $\lambda_1, \ldots, \lambda_l$ are some nonzero scalars in $k$ and 
$e_1, \ldots, e_l$ are some basis elements in $\mathbb{W}_{<m}$. 
By the induction hypothesis we have that $\mbox{supp}(e_i)\subset A_m$ 
for any $1\leq i \leq l$. Note that in Construction \ref{C:basis3} we have that 
$\mbox{supp}(f)\subseteq \overline{\mbox{supp}(e_1)\cup \cdots \cup \mbox{supp}(e_l)}$. 
Hence, we have that $\mbox{supp}(f)\subseteq A_m$. 
Thus, in Construction \ref{C:VL} we can take $\Gamma_m=\widehat{\Gamma}_m$. 

So, by Theorem \ref{T:main1}, 
$B=\bigcup\limits_{m\in L_M}\widehat{\Gamma}_m=\bigcup\limits_{m\in L_M} \Gamma_m$ 
is a Taylor basis of $\mathbf{F}$. 

Conversely, let $\widetilde{B}$ be a Taylor basis of $\mathbf{F}$. 
By Theorem \ref{T:main2} we know that $\widetilde{B}$ can be obtained 
by Construction \ref{C:VL}. By Theorem \ref{T:main1} and Remark \ref{R:main1remark1} 
it is easy to see that every Taylor basis obtained by Construction \ref{C:VL} 
satisfies the conditions in Construction \ref{C:basis3} and then can be obtained 
by Construction \ref{C:basis3}. So, $\widetilde{B}$ can be obtained by 
Construction \ref{C:basis3}. 
\end{proof}

Note that in Construction \ref{C:basis3}, if we use 
$\mbox{supp}(g_{i_1})\cup \cdots \cup \mbox{supp}(g_{i_l})$ 
instead of $\overline{\mbox{supp}(g_{i_1})\cup \cdots \cup \mbox{supp}(g_{i_l})}$, 
then by Remark \ref{R:main1remark1} it is easy to see that 
the second part of Theorem \ref{T:basis3} may not hold. 

Next we use the atomic lattice resolution theory to prove some results about 
Scarf multidegrees. 

\begin{proposition}\label{P:pm}
For any $m\neq 1 \in L_M$ let $P_m$ be as defined in Theorem \ref{T:main1}. 
Then $m$ is a Scarf multidegree if and only if $P_m=\emptyset$. 
\end{proposition}

\begin{proof}
\emph{if}: Assume that $m$ is not a Scarf multidegree then there exists 
$A\subset A_m$ such that $|A|=|A_m|-1$ and $\overline{A}=A_m$. 
Hence, from the Taylor resolution of $S/M$, we can first do a consecutive 
cancellation with respect to basis elements $A_m$ and $A$, and then after 
a series of consecutive cancellations we obtain a minimal free resolution 
of $S/M$ with a Taylor basis. By Remark \ref{R:main1remark1} and the 
proof of Theorem \ref{T:main1} we see that 
$|\Gamma_m| \leq |Q_m|-2$, which implies that $|P_m|=|Q_m|-|\Gamma_m|\geq 2$, 
and then $P_m\neq \emptyset$. 

\emph{only if}: Assume that $m$ is a Scarf multidegree then by remark 
\ref{R:cbasis} we can take $\Gamma_m=\{A_m\}$ in Construction \ref{C:VL}. 
Note that $Q_m=\{A_m\}$, and then $|P_m|=|Q_m|-|\Gamma_m|=1-1=0$, 
so that $P_m=\emptyset$. 
\end{proof}

In Theorem 5.6 and Theorem 7.1 of \cite{B:Me},  Mermin proves that 
the intersection of all the minimal free resolutions of $S/M$ embedded 
in the Taylor resolution is the Scarf complex of $M$. His proof uses 
the Lyubeznik resolutions of $M$. Here we will give this result a 
different proof. Our proof is slightly longer, but the idea is simple. 

\begin{theorem}[\cite{B:Me}]\label{T:intersection}
Let $M$ be a monomial ideal minimally generated by $r$ monomials. 
Let $\Omega$ be the simplex with the vertex set $\{1, \ldots, r\}$. 
Let $\mathbf{T}$ be the Taylor resolution of $S/M$ with basis the 
faces of $\Omega$. 
Let 
\[
\Omega_M=\{A_m|m\in L_M \ \mbox{is \ a\  Scarf\ multidegree} \}
\]
be the  Scarf complex of $M$. 
Then $\mathbf{T}|_{\Omega_M}$ is a subcomplex of $\mathbf{T}$. 
As in Remark \ref{R:TBmany} let $\Sigma_M$ be the set of Taylor 
submodules for $M$. Then we have that 
\[
\bigcap\limits_{N\in \Sigma_M}N=\mathbf{T}|_{\Omega_M}. 
\]
\end{theorem}

\begin{proof}
Pick any $A_m\in \mathbf{T}|_{\Omega_M}$. Since $m$ is a Scarf 
multidegree, by the proof of Proposition \ref{P:pm} we see that 
$Q_m=\{A_m\}$ and $|\Gamma_m|=1$. Let $|A_m|=i$ and 
$\Gamma_m=\{f\}$. Since $f$ is a Taylor chain at $m$, it follows 
that there exists $\lambda\neq 0 \in k$ such that $\mbox{in}(f)=\lambda A_m$. 
Assume that $f\neq \mbox{in}(f)$ then there exist nonzero scalars 
$\lambda_1, \ldots, \lambda_p \in k$ and $(i-1)$-dimensional faces 
$c_1, \ldots, c_p$ of $\Omega$ such that $c_j\neq A_m$ for any 
$1\leq j \leq p$, $c_j\neq c_l$ for any $1\leq j< l \leq p$, and 
\[
f=\lambda A_m+\lambda_1c_1+\cdots+\lambda_pc_p. 
\]
Let $\widetilde{m}=\mbox{mdeg}(c_1)$ then we have that 
$\widetilde{m}<m \in L_M$, so that $\widetilde{m}$ is also a Scarf 
multidegree and $c_1=A_{\widetilde{m}} \subsetneq A_m$, which 
contradicts to the assumption that $|c_1|=|A_m|=i$. 
Thus, we have that $f=\mbox{in}(f)=\lambda A_m$ and then 
$\Gamma_m=\{\lambda A_m\}$. 
By Theorem \ref{T:main2} it is easy to see that for any Taylor 
submodule $N$ for $M$ there exists a nonzero scalar $\lambda  \in k$ 
such that $\lambda A_m \in N$ and then $A_m \in N$. 
So we have that $\mathbf{T}|_{\Omega_M} \subseteq \bigcap\limits_{N\in \Sigma_M}N$. 

On the other hand, let $A\subseteq \{1, \ldots, r\}$ such that  
$\overline{A}=A_m$ where $m$ is not a Scarf multidegree. 
We want to show that there exists $N\in \Sigma_M$ such that 
every element in $N$ does not contain a term involving $A$, or 
equivalently, there exists a minimal free resolution of $S/M$ with 
a Taylor basis such that $A$ does not appear in any Taylor basis element. 

Case 1: If $\overline{A}=A_m$ and $A\neq A_m$ then there exists 
$B\in Q_m$ such that $|B|=|A|+1$. From the Taylor resolution, we can 
first do a consecutive cancellation with respect to $B$ and $A$, and then 
use a series of consecutive cancellations to obtain a minimal free resolution 
of $S/M$ with a Taylor basis. 
By Proposition \ref{P:CCbasis} it is easy to see that no element in the Taylor 
basis contains a term involving $A$.  

Case 2: If $A=\overline{A}=A_m$ then there exists $B\in Q_m$ such that 
$|B|=|A|-1$. From the Taylor resolution, we can first do a consecutive 
cancellation with respect to $A$ and $B$, and then use a series of 
consecutive cancellations to obtain a minimal free resolution $\mathbf{F}$ of $S/M$ with 
a Taylor basis $\Gamma$, in which every chain does not have a term involving $B$. 
Let $|A|=i$. Since $m$ is not a Scarf multidegree, we have that $i\geq 2$. 
Suppose that $A$ appears in the Taylor basis elements $f_1, \ldots, f_t \in \Gamma$ 
in homological degree $i$. Let 
\[
f_1=\lambda_1A+\lambda_2c_2+\cdots+\lambda_lc_l,
\]
where $\lambda_1, \ldots, \lambda_l \neq 0 \in k$ and $c_2, \ldots, c_l$ are 
some $(i-1)$-dimensional faces in $\Omega$ different from $A$.  
Since $d(A)$ has a term involving $B$ and $d(f_1)$, which is a linear combination 
of some elements in $\Gamma$, does not have a term involving $B$, it follows that 
there exists $c_i$ such that $A\cap c_i=B$. Without the loss of generality we 
assume that $A\cap c_2=B$ and then $|A\cup c_2|=i+1\geq 3$. 
Suppose that $d(A\cup c_2)=\nu(A-\widetilde{c})$, where $\nu$ is $\pm 1$ and 
$\widetilde{c}$ is an $(i-1)$-dimensional chain not containing $A$. Then 
we have that $d(A)=d(\widetilde{c})$. Let 
\[
\widetilde{f}_1=\lambda_1\widetilde{c}+\lambda_2c_2+\cdots+\lambda_lc_l,
\]
then $\widetilde{f}_1$ does not have a term involving $A$, 
$d(\widetilde{f}_1)=d(f_1)$ and $\mbox{supp}(\widetilde{f}_1) \subseteq 
\mbox{supp}(\widetilde{c})\cup c_2 \cup \cdots \cup c_l
=A\cup c_2 \cup \cdots \cup c_l =\mbox{supp}(f_1)$. 
Similarly, we can get $\widetilde{f}_2, \ldots, \widetilde{f}_t$. 
By Theorem \ref{T:basis3} the Taylor basis $\Gamma$ can be obtained by 
Construction \ref{C:basis3}. Thus, by using Construction \ref{C:basis3} again,  we can 
keep the Taylor basis elements in $\Gamma$ whose homological degrees are 
less than $i$; and we replace the Taylor basis 
elements $f_1, \ldots, f_t$ by $\widetilde{f}_1, \ldots, \widetilde{f}_t$, respectively, 
and keep the other Taylor basis elements in $\Gamma$ of homological degree $i$; 
and then in higher homological degrees  we use Construction \ref{C:basis3} to get 
Taylor basis elements step by step. 
Thus, we obtain a Taylor basis of $\mathbf{F}$ in which every Taylor basis 
element does not have a term involving $A$. 
So we have that $\bigcap\limits_{N\in \Sigma_M}N \subseteq \mathbf{T}|_{\Omega_M}$, 
which implies that $\bigcap\limits_{N\in \Sigma_M}N=\mathbf{T}|_{\Omega_M}$. 
\end{proof}

The concept of exact closure in Section 3.1 and the construction of 
atomic lattice resolutions in Section 3.2 originate from the concept 
of nearly Scarf monomial ideals introduced by Peeva and Velasco in \cite{B:PV}. 
A monomial ideal $M$ is called a \emph{nearly Scarf monomial ideal} if for any 
$m\in L_M$, either $m$ is the top element of $L_M$, or $m$ is a Scarf 
multidegree. We show that Theorem 6.1 in \cite{B:PV} about nearly 
Scarf monomial ideals is a special case of Theorem \ref{T:main1}. 

\begin{theorem}[\cite{B:PV}]\label{T:nearly}
Let $M$ be a nearly Scarf monomial ideal in $S$. Let $m$ be the top element in $L_M$. 
Let $\mathbb{V}$ be an exact closure of $\widetilde{C}(\Delta_m; k)[-1]$. 
Then the $M$-homogenization of $\mathbb{V}$ is a minimal free resolution 
of $S/M$. 
\end{theorem}

\begin{proof}
Since $M$ is a nearly Scarf monomial ideal, it follows that for any $\widetilde{m}<m \in L_M$ 
we can take $\Gamma_{\widetilde{m}}=\{A_{\widetilde{m}}\}$ in Construction 
\ref{C:VL}, and then we have that $(\mathbb{U}_m, d)=\widetilde{C}(\Delta_m; k)[-1]$. 
Since $\mathbb{V}$ is an exact closure of $\widetilde{C}(\Delta_m; k)[-1]$, 
in Construction \ref{C:VL} we can take $\mathbb{V}(L_M)=(\mathbb{V}_m, d)=\mathbb{V}$. 
So by Theorem \ref{T:main1}, the $M$-homogenization of $\mathbb{V}$ is a 
minimal free resolution of $S/M$. 
\end{proof}

Theorem \ref{T:main1} and Theorem \ref{T:main2} imply that the lcm-lattice 
$L_M$ determines the minimal free resolutions of $S/M$, which is also proved 
in Theorem 3.3 in \cite{B:GPW}.  Similarly, Theorem 2.1 in \cite{B:CM2} and 
Theorem 5.3 in \cite{B:TV} show that the Betti poset of $M$ determines the 
minimal free resolutions of $S/M$. Next we will give this result a new proof. 

\begin{definition}\label{D:BettiPoset}
Let $M$ be a monomial ideal. Let $B_M$ be the poset obtained from the 
lcm-lattice $L_M$ by deleting all $m\neq 1 \in L_M$ with $\widetilde{H}(\Delta_m; k)=0$. 
$B_M$ is called the \emph{Betti poset} of $M$. 
Note that the bottom element 1 and the atoms $m_1, \ldots, m_r$ of 
$L_M$ are in $B_M$. 
\end{definition}

\begin{theorem}[\cite{B:CM2}, \cite{B:TV}]\label{T:BettiPoset}
Let $M$ be a monomial ideal in $S=k[x_1, \ldots, x_n]$ and $N$ 
a monomial ideal in $R=k[y_1, \ldots, y_l]$ 
such that there is an isomorphism $f$ between 
$B_M$ and $B_N$. Let $\mathbf{F}$ be a minimal free resolution of 
$S/M$. Then the $N$-homogenization of the frame of $\mathbf{F}$ is 
a minimal free resolution of $R/N$. 
\end{theorem}

\begin{proof}
Let $M$ and $N$ be minimally generated by $r$ monomials. 
By Theorem \ref{T:main2} we can use Construction \ref{C:VL} to obtain 
an $r$-frame $\mathbb{V}(L_M)$ associated with $L_M$ 
such that the $M$-homogenization of $\mathbb{V}(L_M)$ equals $\mathbf{F}$. 
By Corollary \ref{C:main1c1} we have that $\Gamma_m=\emptyset$ for 
any $m\in L_M-B_M$, so that in Construction \ref{C:VL}, for any $m\in B_M$, 
$\Gamma_m$ is obtained by using 
$\bigcup \limits_{\widetilde{m}<m \in B_M}\Gamma_{\widetilde{m}}$. 
Since $f$ is an isomorphism between $B_M$ and $B_N$, it follows that 
$A_{f(m)}=A_m$ for any $m\in B_M$. 
Hence, it is easy to see that by using Construction \ref{C:VL} for $L_N$, 
we have that $\Gamma_{\widehat{m}}=\emptyset$ for any 
$\widehat{m} \in L_N-B_N$, and we can choose $\Gamma_{f(m)}=\Gamma_m$ 
for  any $m \in B_M$. Thus, $\mathbb{V}(L_M)$ is also an $r$-frame associated 
with $L_N$. So by Theorem \ref{T:main1} the $N$-homogenization of 
$\mathbb{V}(L_M)$ is a minimal free resolution of $R/N$. 
\end{proof}

Next we obtain a bound for the projective dimension of $S/M$. 

\begin{proposition}\label{P:projdim}
Let $M$ be a monomial ideal. Then for any $m\in L_M$ we have that 
$\mbox{projdim}(S/(M_{\leq m}))\leq \mbox{rk}(m)$. In particular, we have 
that $\mbox{projdim}(S/M) \leq \mbox{rk}(L_M)$. 
\end{proposition}

\begin{proof}
We will prove by using strong induction in $L_M$. 

Base case: If $m=1$ then $\mbox{projdim}(S)=0=\mbox{rk}(1)$; 
if $m$ is an atom in $L_M$ then $\mbox{projdim}(S/(m))=1=\mbox{rk}(m)$. 

Inductive step: Let $m\in L_M$ with $\mbox{rk}(m)\geq 2$, and we assume 
that for any $\widetilde{m}<m \in L_M$ we have that 
$\mbox{projdim}(S/(M_{\leq \widetilde{m}})) \leq \mbox{rk}(\widetilde{m})$. 
Then we have that 
\[
\max\limits_{\widetilde{m}<m}(\mbox{projdim}(S/(M_{\leq \widetilde{m}}))
\leq \max\limits_{\widetilde{m}<m}(\mbox{rk}(\widetilde{m})) \leq \mbox{rk}(m)-1.
\]
Hence, the length of the complex $(\mathbb{U}_m, d)$ in Construction \ref{C:VL} 
is less than or equal to $\mbox{rk}(m)-1$, which implies that the length of the 
complex $(\mathbb{V}_m, d)$ in Construction \ref{C:VL} is less than or equal to 
$\mbox{rk}(m)$. So by Theorem \ref{T:main1} we have that 
$\mbox{projdim}(S/(M_{\leq m}))\leq \mbox{rk}(m)$. 
\end{proof}

Note that in Proposition \ref{P:projdim} we can replace $L_M$ by $B_M$ to 
get a sharper bound for $\mbox{projdim}(S/M)$. The next proposition about 
the reduced homology of $\Delta_m$ will be use in Section 4. It is interesting 
to see that this geometric result is proved by an algebraic method. 

\begin{proposition}\label{P:homdelta}
Let $m\neq 1 \in L_M$;
\begin{itemize}
 \item [(1)] for any $i\geq \mbox{rk}(m)$ we have that $\widetilde{H}_{i-1}(\Delta_m;k)=0$; 
 \item [(2)] if for any $1\neq \widetilde{m}<m \in L_M$ we have that 
                  $\widetilde{H}_{i_0}(\Delta_{\widetilde{m}}; k)=0$, then we have that 
                  $\widetilde{H}_j(\Delta_m; k)=0$ for any $j\geq i_0+1$. 
\end{itemize}
\end{proposition}

\begin{proof}
(1) For any $i\geq \mbox{rk}(m)$ we have that $i> \mbox{rk}(m)-1$, so that by 
the proof of Proposition \ref{P:projdim} we have that 
$H_i(\mathbb{U}_m, d)=0$. So by Corollary \ref{C:main1c1} we have 
that $\widetilde{H}_{i-1}(\Delta_m; k)=0$. 

(2) Since $\widetilde{H}_{i_0}(\Delta_{\widetilde{m}}; k)=0$ for all 
$1\neq \widetilde{m}<m \in L_M$, by Theorem \ref{T:basis} and Theorem \ref{T:main1} 
we see that in Construction \ref{C:VL} $\Gamma_{\widetilde{m}}$ has no elements 
of dimension $i_0+1$. Hence, there are no basis elements in $(\mathbb{U}_m, d)$ 
of dimension $i_0+1$. Then by Construction \ref{C:VL} it is easy to see that there 
are no basis elements in $\mathbb{U}_m$ of dimension greater than   
$i_0+1$, so that the length of $\mathbb{U}_m$ is less than or equal to $i_0+1$. 
Thus, we have that $H_j(\mathbb{U}_m, d)=0$ for any $j\geq i_0+2$. 
So by Corollary \ref{C:main1c1} we have that $\widetilde{H}_j(\Delta_m; k)=0$ 
for any $j\geq i_0+1$. 
\end{proof}

\section{Differential Maps Induced by Mayer-Vietoris Sequences}

In Construction \ref{C:VL} we use step-by-step calculations to obtain a 
minimal free resolution of $S/M$. Although a general formula for  
minimal free resolutions of \emph{all} monomial ideals seems impossible, we are 
still interested in finding explicit descriptions for minimal free resolutions 
of some classes of monomial ideals. 

We know that a Taylor basis determines 
a minimal free resolution of $S/M$, and by Theorem \ref{T:basis} and 
Theorem \ref{T:cbasis} we see that the Taylor basis and the reduced 
homology groups $\widetilde{H}(\Delta_m; k)$ are closely related. 
Hence, one may ask if it is possible to construct a minimal free resolution 
of $S/M$ directly from the reduced homology groups 
$\widetilde{H}(\Delta_m;k)$. 

This idea was first studied by Clark and Tchernev in \cite{B:Cl}. They 
introduce a new concept called the poset resolution, the frame of 
which is constructed from the lcm-lattice $L_M$, and has 
differential maps induced by the connecting homomorphisms 
in some Mayer-Vietoris sequences. Later, Clark and Mapes in \cite{B:CM1} 
and \cite{B:CM2}, and Wood in \cite{B:Wo}, generalize the 
construction to Betti posets, and they prove that rigid monomial 
ideals and Betti-linear monomial ideals have poset resolutions. 

This section can be viewed as an application of the atomic lattice 
resolution theory developed in Section 3, and it has two subsections. 
In Subsection 4.1 we rewrite the theory of poset resolutions and 
give a new proof of Betti-linear monomial ideals having poset resolutions. 
In Subsection 4.2 we develop a theory similar to poset resolutions 
and obtain an approximation formula for minimal free resolutions of 
\emph{all} monomial ideals. 
The theory in Subsection 4.2 turns out to be a generalization of the theory 
in Subsection 4.1.

\subsection{Poset Resolutions and Betti-linear Monomial Ideals}
Let $M$ be a monomial ideal minimally generated by monomials 
$m_1, \ldots, m_r$. Let $L_M$ be the lcm-lattice of $M$ and 
$B_M$ the Betti poset of $M$. For any $m\in L_M$ with 
$\mbox{rk}(m)\geq 2$, let 
\[
\mathcal{B}(m)=\{\widetilde{m} \in B_M|\widetilde{m}<m\}, 
\] 
then $\mathcal{B}(m)\neq \emptyset$. 
Let $\beta_1, \ldots, \beta_l$ be the maximal elements in 
$\mathcal{B}(m)$; and let 
\[
\widetilde{\Delta}_m=\langle A_{\beta_1}, \ldots, A_{\beta_l} \rangle,  
\]
then $\widetilde{\Delta}_m$ is a subcomplex of $\Delta_m$.  
These notations will be used throughout this subsection. 

\begin{lemma}\label{L:multidegree}
Let $m\in L_M$ with $\mbox{rk}(m)\geq 2$. Let 
$[f]\neq 0 \in \widetilde{H}(\Delta_m; k)$. Then we have that 
$\mbox{mdeg}(f)=m$. 
\end{lemma}

\begin{proof}
Since $f$ is a chain in $\Delta_m$, it follows that $\mbox{mdeg}(f)|m$. 
Assume that $\mbox{mdeg}(f)=\widehat{m}<m\in L_M$. Then $f$ is a 
cycle in the simplex $\Omega_{\widehat{m}} \subseteq \Delta_m$, so that $f$ is a trivial 
cycle in $\Delta_m$, which contradicts to the assumption that 
$[f]\neq 0$ in $\widetilde{H}(\Delta_m; k)$. So we have that 
$\mbox{mdeg}(f)=m$. 
\end{proof}

\begin{proposition}\label{P:iso}
Let $m\in L_M$ with $\mbox{rk}(m)\geq 2$.
Let $\sigma_m: \widetilde{H}(\widetilde{\Delta}_m; k) \to 
\widetilde{H}(\Delta_m; k)$  be the $k$-linear map induced by
the inclusion $\tau: \widetilde{\Delta}_m \to \Delta_m$. 
Then $\sigma_m$ is an isomorphism. 
\end{proposition}

\begin{proof}
Let $\mathcal{J}(m)$ be the set of all $\widetilde{m}<m \in L_M$ 
such that  there does not exist $\beta_j$ with $\widetilde{m}<\beta_j$. 
Then we have that $\Gamma_{\widetilde{m}}=\emptyset$ for any 
$\widetilde{m} \in \mathcal{J}(m)$. 
Hence, by Theorem \ref{T:main1} and Corollary \ref{C:main1c2}, 
it is easy to see that 
\[
\widetilde{C}(\Delta_m; k)[-1]\cong \widetilde{C}(\widetilde{\Delta}_m; k)[-1] 
\bigoplus ( \bigoplus \limits_{\widetilde{m}\in \mathcal{J}(m)}
(\mathcal{E}_{\widetilde{m}}; d) ).
\]
Since $(\mathcal{E}_{\widetilde{m}}; d)$ is a trivial complex for any 
$\widetilde{m} \in \mathcal{J}(m)$ and 
$\widetilde{C}(\widetilde{\Delta}_m; k)[-1]$ is a subcomplex of 
$\widetilde{C}(\Delta_m; k)[-1]$ induced by the inclusion map 
$\tau: \widetilde{\Delta}_m \to \Delta_m$, it follows that 
$\sigma_m: \widetilde{H}(\widetilde{\Delta}_m; k) \to \widetilde{H}(\Delta_m; k)$ 
induced by $\tau$ is an isomorphism. 
\end{proof}

Proposition \ref{P:iso} can also be proved by using homology theory 
as in Lemma 4.2 and Proposition 4.3 in \cite{B:Wo}. But the proof by 
using the atomic lattice resolution theory is much easier. 

Next we define a sequence of $k$-vector spaces and $k$-linear maps 
from the lcm-lattice $L_M$: 
\[
\mathcal{D}(L_M): \cdots \to \mathcal{D}_i 
\xrightarrow{\varphi_i} \mathcal{D}_{i-1} 
\to \cdots \to \mathcal{D}_1 \xrightarrow{\varphi_1} \mathcal{D}_0. 
\]

\begin{definition}[\cite{B:Cl}]\label{D:posetres}
Let $\mathcal{D}_0=\widetilde{H}_{-1}(\{\emptyset\}; k)\cong k$. For any $i\geq 1$,  let 
\[
\mathcal{D}_i=\bigoplus\limits_{1\neq m \in L_M}\widetilde{H}_{i-2}(\Delta_m; k). 
\]
Since $\widetilde{H}_{-1}(\Delta_m; k)\neq 0$ if and only if 
$\Delta_m=\{\emptyset\}$, it follows that 
\[
\mathcal{D}_1=\bigoplus\limits_{j=1}^r\widetilde{H}_{-1}(\Delta_{m_j}; k)\cong k^r. 
\]
We define $\varphi_1: \mathcal{D}_1 \to \mathcal{D}_0$ componentwise by 
$\mbox{id}_{\widetilde{H}_{-1}(\{\emptyset\}; k)}$, and then 
$\mathcal{D}_1 \xrightarrow{\varphi_1} \mathcal{D}_0$ can be written as 
$k^r \xrightarrow{(\begin{smallmatrix} 1 & \cdots & 1 \end{smallmatrix})} k$. 

For any $i\geq 2$ we define 
\[
\varphi_i: \mathcal{D}_i=\bigoplus\limits_{1\neq m \in L_M}\widetilde{H}_{i-2}(\Delta_m; k) 
\to \mathcal{D}_{i-1}=\bigoplus\limits_{1\neq m \in L_M}\widetilde{H}_{i-3}(\Delta_m; k)
\]
componentwise as follows. 
Suppose that $\widetilde{H}_{i-2}(\Delta_m; k) \neq 0$ with $i\geq 2$, then we have that 
$\mbox{rk}(m)\geq 2$. Next we componentwise define
\[
\varphi_{i,m}: \widetilde{H}_{i-2}(\Delta_m; k) 
\to \mathcal{D}_{i-1}=\bigoplus\limits_{1\neq m \in L_M}\widetilde{H}_{i-3}(\Delta_m; k).
\] 
Let $\widetilde{\Delta}_m=\langle A_{\beta_1}, \ldots, A_{\beta_l} \rangle$.
By Proposition \ref{P:iso} we have that 
$\sigma_m: \widetilde{H}_{i-2}(\widetilde{\Delta}_m; k) \to \widetilde{H}_{i-2}(\Delta_m; k)$ 
is an isomorphism, so that $\widetilde{H}(\widetilde{\Delta}_m; k) \neq 0$, 
which implies that $l\geq 2$. 

First we define $\varphi_{i,m}^{\beta_1}: \widetilde{H}_{i-2}(\Delta_m; k) 
\to \widetilde{H}_{i-3}(\Delta_{\beta_1}; k)$. 
Let $\Delta_1=\langle A_{\beta_1} \rangle$ and 
$\Delta_2=\langle A_{\beta_2}, \ldots, A_{\beta_l} \rangle$, then we have that 
$\Delta_1 \cup \Delta_2=\widetilde{\Delta}_m$ and by Proposition \ref{P:meetjoin} 
\[
\Delta_1 \cap \Delta_2=\langle A_{\beta_1}\cap A_{\beta_2}, \ldots, A_{\beta_1}\cap A_{\beta_l} \rangle 
=\langle A_{\beta_1\wedge \beta_2}, \ldots, A_{\beta_1\wedge \beta_l} \rangle.
\]
Since $\beta_1, \ldots, \beta_l$ are maximal elements in $\mathcal{B}(m)$, 
it follows that $\beta_1\wedge \beta_2<\beta_1, \ldots, \beta_1\wedge \beta_l<\beta_1$  
in $L_M$, which implies that $\Delta_1 \cap \Delta_2$ is a subcomplex of 
$\Delta_{\beta_1}$. 
Let $[f]\in \widetilde{H}_{i-2}(\widetilde{\Delta}_m; k)$ be such that 
$f=c_1-c_2$,  where $c_1$ is a chain in $\Delta_1$ and $c_2$ is a chain 
in $\Delta_2$, then from the Mayer-Vietoris sequence we have the connecting 
homomorphism 
\[
\delta: \widetilde{H}_{i-2}(\widetilde{\Delta}_m; k) \to  
\widetilde{H}_{i-3}(\Delta_1 \cap \Delta_2; k) 
\]
such that $\delta([f])=[d(c_1)]=[d(c_2)]$, where $d$ is the boundary map 
in $\widetilde{\Delta}_m$. Let 
\[
\iota: \widetilde{H}_{i-3}(\Delta_1 \cap \Delta_2; k) \to 
\widetilde{H}_{i-3}(\Delta_{\beta_1}; k)
\]
be the $k$-linear map induced by the inclusion map
$\Delta_1 \cap \Delta_2 \to \Delta_{\beta_1}$. 
Then we define the $k$-linear map 
$\varphi_{i,m}^{\beta_1}: \widetilde{H}_{i-2}(\Delta_m; k) 
\to \widetilde{H}_{i-3}(\Delta_{\beta_1}; k)$ by 
\[
\varphi_{i,m}^{\beta_1}=\iota \circ \delta \circ \sigma_m^{-1}. 
\]

Similarly, for any $2\leq j\leq l$, by letting 
\[
\Delta_1=\langle A_{\beta_j} \rangle, \ \ 
\Delta_2=\langle A_{\beta_1}, \ldots, A_{\beta_{j-1}}, A_{\beta_{j+1}}, \ldots, A_{\beta_l}\rangle, 
\]
and using the corresponding Mayer-Vietoris sequence, we can define 
the $k$-linear map $\varphi_{i,m}^{\beta_j}: \widetilde{H}_{i-2}(\Delta_m; k) 
\to \widetilde{H}_{i-3}(\Delta_{\beta_j}; k)$. 
And for any $1\neq \widetilde{m} \in L_M$ with $\widetilde{m} \notin 
\{\beta_1, \ldots, \beta_l\}$, we define 
$\varphi_{i,m}^{\widetilde{m}}: \widetilde{H}_{i-2}(\Delta_m; k) 
\to \widetilde{H}_{i-3}(\Delta_{\widetilde{m}}; k)$ to be the zero map. 

Thus, we can define
\[
\varphi_{i,m}=\sum_{1\neq \widetilde{m} \in L_M} \varphi_{i, m}^{\widetilde{m}} 
=\varphi_{i,m}^{\beta_1} +\cdots +\varphi_{i,m}^{\beta_l}, 
\]
and then $\varphi_i$ is defined componentwise. 
We call 
\[
\mathcal{D}(L_M): \cdots \to \mathcal{D}_i 
\xrightarrow{\varphi_i} \mathcal{D}_{i-1} 
\to \cdots \to \mathcal{D}_1 \xrightarrow{\varphi_1} \mathcal{D}_0. 
\]
a \emph{poset construction}. 
If we fix a basis for every $\widetilde{H}(\Delta_m;k)$ then we have
a basis for $\mathcal{D}(L_M)$. 

For any $i\geq 0$ let $\mathcal{F}_i=\mathcal{D}_i\otimes S$. 
Let $f_0$ be the basis element of $\mathcal{D}_0$. For any $1\leq j \leq r$, 
let $f_j$ be the basis element of $\widetilde{H}_{-1}(\Delta_{m_j}; k)$. 
We still use $f_0$ as the basis of $\mathcal{F}_0$ and use $f_1, \ldots, f_r$ 
as the basis of $\mathcal{F}_1$. To make $\mathcal{F}_0$ and 
$\mathcal{F}_1$ multigraded free $S$-modules, we define 
$\mbox{mdeg}(f_0)=1$ and $\mbox{mdeg}(f_j)=m_j$ for $1\leq j\leq r$. 
For any $i\geq 2$, let $[f]$ be a basis element of $\widetilde{H}_{i-2}(\Delta_m; k)$. 
We still use $[f]$ as a basis element of $\widetilde{H}_{i-2}(\Delta_m; k)\otimes S$, 
and we define $\mbox{mdeg}([f])=\mbox{mdeg}(f)=m$.  
Thus, $\mathcal{F}_i$ is a multidgraded free $S$-module with a multigraded basis. 

We define $\partial_1: \mathcal{F}_1 \to \mathcal{F}_0$ by 
$S^r \xrightarrow{(\begin{smallmatrix} m_1 & \cdots & m_r \end{smallmatrix})} S$. 
For any $i\geq 2$ and $1\leq j \leq l$, we can homogenize the map $\varphi_{i,m}^{\beta_j}$ 
to get 
\[
\partial_{i,m}^{\beta_j}: \widetilde{H}_{i-2}(\Delta_m; k)\otimes S \to 
\widetilde{H}_{i-3}(\Delta_{\beta_j}; k) \otimes S 
\]
with $\partial_{i, m}^{\beta_j}=\frac{m}{\beta_j}\otimes \varphi_{i,m}^{\beta_j}$. 
And then  we define $\partial_i: \mathcal{F}_i \to \mathcal{F}_{i-1}$ 
componentwise by $\partial_{i, m}$, where 
$\partial_{i, m}=\partial_{i, m}^{\beta_1} +\cdots +\partial_{i,m}^{\beta_l}$. 
Thus, for any $i\geq 0$, $\partial_i$ is multigraded. 

So we obtain a sequence of multigraded free $S$-modules and multigraded homomorphisms:
\[
\mathcal{F}(L_M):\  \cdots \to \mathcal{F}_i \xrightarrow{\partial_i} \mathcal{F}_{i-1} 
\to \cdots \to \mathcal{F}_1 \xrightarrow{\partial_1} \mathcal{F}_0.
\]
If $\mathcal{F}(L_M)$ is a minimal free resolution of $S/M$ then we say that 
$\mathcal{F}(L_M)$ is a \emph{poset resolution} of $S/M$. 
\end{definition}

\begin{remark}\label{R:posetres}
In \cite{B:Cl}, \cite{B:CM2} and \cite{B:Wo} the poset construction is defined differently by 
using $\widetilde{H}(\Delta(1, m); k)$, where $\Delta(1, m)$ is the order complex 
of the open interval $(1, m)$ in $L_M$.  By the crosscut theorem it is easy to see 
that $\Delta(1, m)$ and $\Delta_m$ are homotopy equivalent, so that 
$\widetilde{H}(\Delta(1, m); k)\cong \widetilde{H}(\Delta_m; k)$. In Section 5 of 
\cite{B:Cl} and Section 4 of \cite{B:Wo} it is proved that the poset constructions 
defined by using $\widetilde{H}(\Delta(1, m); k)$ and $\widetilde{H}(\Delta_m; k)$, 
respectively, are equivalent. Since $\Delta_m$ is much simpler than $\Delta(1, m)$ 
and closely related to the atomic lattice resolution theory, here we choose to define 
the poset construction by using $\widetilde{H}(\Delta_m; k)$. 

Note that the poset construction $\mathcal{D}(L_M)$ may not be a complex of 
$k$-vector spaces,  so that $\mathcal{D}(L_M)$ may not be an $r$-frame, and 
then the $M$-homogenization of $\mathcal{D}(L_M)$ may not make sense. 
This is why in Definition \ref{D:posetres} we define $\mathcal{F}(L_M)$ 
componentwise in each homological degree. 
This can be illustrated by Example \ref{E:notBetti}. 
\end{remark}

Note that if we fix a basis for every $\widetilde{H}_{i-2}(\Delta_m; k)\neq 0$ then 
$\mathcal{F}(L_M)$ is uniquely determined. A natural question is to ask when 
$\mathcal{F}(L_M)$ is a minimal free resolution of $S/M$. This question is closely 
related to the following classes of monomial ideals. 

\begin{definition}\label{D:ideals}
Let $M$ be a monomial ideal. 
\begin{itemize}
\item[(1)](\cite{B:FMS}) $M$ is called a \emph{homologically monotonic monomial 
ideal} if the following condition is satisfied: for any $1<\widetilde{m}<m \in L_M$, 
if $\widetilde{H}_i(\Delta_{\widetilde{m}}; k)\neq 0$ and 
$\widetilde{H}_j(\Delta_m; k) \neq 0$ then $i<j$. 
\item[(2)](\cite{B:CM1}) Let $M$ be a homologically monotonic monomial ideal and 
for any $m\in B_M$ we have that $\widetilde{H}(\Delta_m; k)\cong k$, then $M$ 
is called a \emph{rigid monomial ideal}. 
\item[(3)](\cite{B:Cl}) $M$ is called a \emph{lattice-linear monomial ideal} if $S/M$ 
has a minimal free resolution $(\mathbf{F}, \partial)$ 
with a multigraded basis $T$ such that the following condition is satisfied:
\[
\forall e\in T, \ \mbox{let} \ \partial(e)=\sum_{\widetilde{e}\in T} \lambda_{e, \widetilde{e}}\widetilde{e},
\]
then whenever $\lambda_{e, \widetilde{e}}\neq 0 \in S$ we have that 
$\mbox{mdeg}(\widetilde{e}) \lessdot \mbox{mdeg}(e)$ in $L_M$.
\item[(4)] (\cite{B:Wo}) $M$ is called a \emph{Betti-linear monomial ideal} if $S/M$ 
has a minimal free resolution $(\mathbf{F}, \partial)$ 
with a multigraded basis $T$ such that the following condition is satisfied:
\[
\forall e\in T, \ \mbox{let} \ \partial(e)=\sum_{\widetilde{e}\in T} \lambda_{e, \widetilde{e}}\widetilde{e},
\]
then whenever $\lambda_{e, \widetilde{e}}\neq 0 \in S$ we have that 
$\mbox{mdeg}(\widetilde{e}) \lessdot \mbox{mdeg}(e)$ in $B_M$.
\end{itemize}
\end{definition}

Note that in the definition of lattice-linearity and Betti-linearity it 
does not matter which multigraded basis we choose for $\mathbf{F}$. 
Hence, we can always use the homogenization of a Taylor basis 
of $\mathbf{F}$ as a multigraded basis. Indeed, it is easy to see that 
Betti-linearity can also be defined by a Taylor basis as follows: 
$M$ is called a \emph{Betti-linear monomial ideal} if $S/M$ 
has a minimal free resolution $(\mathbf{F}, \partial)$ 
with a Taylor basis $B$ such that the following condition is satisfied:
\[
\forall e\in B, \ \mbox{let} \  d(e)=\sum_{\widetilde{e}\in B} \lambda_{e, \widetilde{e}}\widetilde{e},
\]
then whenever $\lambda_{e, \widetilde{e}}\neq 0 \in k$ we have that 
$\mbox{mdeg}(\widetilde{e}) \lessdot \mbox{mdeg}(e)$ in $B_M$. 

The first part of the next proposition shows that the definition of a rigid 
monomial ideal is the same as the definition in \cite{B:CM1}. It is also proved in 
Proposition 4.3 of \cite{B:FMS}. Here we give it a different proof. 

\begin{proposition}\label{P:HM}
\begin{itemize}
\item[(1)] $M$ is homologically monotonic if and only if 
whenever $\widetilde{H}_i(\Delta_m; k)\neq 0$,  
$\widetilde{H}_i(\Delta_{\widetilde{m}}; k)\neq 0$ and $m\neq \widetilde{m}$, 
we have that $m$ and $\widetilde{m}$ are not comparable in $L_M$. 
\item[(2)] If $M$ is homologically monotonic, then for any $1\neq m \in B_M$ 
there exists $i$ such that $\widetilde{H}(\Delta_m; k)=\widetilde{H}_i(\Delta_m; k)\neq 0$.
\item[(2)] If $M$ is homologically monotonic,  then $M$ is Betti-linear. 
\end{itemize}
\end{proposition}

\begin{proof}
(1) \emph{if}: Assume that $M$ is not homologically monotonic, then there 
exist $\widetilde{m}<m \in L_M$ such that $\widetilde{H}_i(\Delta_{\widetilde{m}}; k)\neq 0$, 
$\widetilde{H}_j(\Delta_m; k)\neq 0$ and $i\geq j$.  If $i=j$ then we get 
a contradiction to the assumption that $\widetilde{m}$ and $m$ are not 
comparable. If $i>j$, then by Proposition \ref{P:homdelta} (2) there exists 
$\widehat{m}<\widetilde{m}<m \in L_M$ such that $\widetilde{H}_j(\Delta_{\widehat{m}}; k)\neq 0$, 
which contradicts to the assumption that $\widehat{m}$ and $m$ are not comparable. 

\emph{only if}: Assume that there exist $\widetilde{m}<m\in L_M$ such that 
$\widetilde{H}_i(\Delta_{\widetilde{m}}; k)\neq 0$ and 
$\widetilde{H}_i(\Delta_m; k)\neq 0$. This contradicts to the assumption 
that $M$ is homologically monotonic. 

(2) Assume that there exists $m\neq 1 \in B_M$ and $i>j$ such that 
$\widetilde{H}_i(\Delta_m; k)\neq 0$ and $\widetilde{H}_j(\Delta_m; k)\neq 0$. 
Then by Proposition \ref{P:homdelta} (2) there exists $\widetilde{m}<m\in L_M$ 
such that $\widetilde{H}_j(\Delta_{\widetilde{m}}; k)\neq 0$. By part (1) this 
is a contradiction to the assumption that $M$ is homologically monotonic. 

(3) Let $\mathbf{F}$ be a minimal free resolution of $S/M$ with a Taylor 
basis $B$. For any $e\in B$, let 
$d(e)=\sum\limits_{\widetilde{e}\in B}\lambda_{e, \widetilde{e}}\widetilde{e}$. 
Assume that there exist $e, \widetilde{e}\in B$ and $m\in B_M$ such that 
$\lambda_{e, \widetilde{e}}\neq 0 \in k$ and 
$\mbox{mdeg}(\widetilde{e}) < m< \mbox{mdeg}(e)$ in $B_M$. 
Let $e$ and $\widetilde{e}$ be in homological degrees $i$ and $i-1$, respectively. 
Then it is easy to see that $i\geq 2$, and by Theorem \ref{T:basis} we have that 
$\widetilde{H}_{i-2}(\Delta_{\mbox{mdeg}(e)}; k)\neq 0$ and 
$\widetilde{H}_{i-3}(\Delta_{\mbox{mdeg}(\widetilde{e})}; k)\neq 0$. 
Assume that $\widetilde{H}_j(\Delta_m; k)\neq 0$. Since $M$ is homologically 
monotonic, it follows that $i-3<j<i-2$, which is impossible. Thus, whenever 
$\lambda_{e, \widetilde{e}} \neq 0\in k$, we have that 
$\mbox{mdeg}(\widetilde{e}) \lessdot \mbox{mdeg}(e)$ in $B_M$. 
So $M$ is Betti-linear. 
\end{proof}

\begin{remark}\label{R:rigid}
\begin{itemize}
  \item[(1)] By part (1) of Proposition \ref{P:HM} we see that given a minimal 
free resolution $\mathbf{F}$ of a rigid monomial ideal with a multigraded basis, 
the only change of basis map of $\mathbf{F}$ is given by multiplying the 
basis elements by nonzero scalars. So, as in \cite{B:CM1} we may say that 
a rigid monomial ideal has a unique minimal free resolution up to rescaling. 
Note that the Taylor basis of $\mathbf{F}$ is not unique up to rescaling as is 
shown in Example \ref{E:strict1}. 
  \item[(2)] By Theorem 3.8 in \cite{B:FMS} we see that $M$ is homologically 
monotonic if and only if there exists a monomial ideal $N$ such that 
$L_M \cong L_N$ and $N$ has a pure resolution. So part (3) of proposition 
\ref{P:HM} is equivalent to Proposition 2.4 in \cite{B:Wo}. 
\end{itemize}
\end{remark}

\begin{remark}\label{R:diagram1}
We have the following diagram indicating the relations among some 
different classes of monomial ideals:
\[
\begin{array}{ccccccc}
\mbox{Scarf} & \subset & \mbox{rigid} & \subset & \mbox{homologically monotonic} 
& \subset & \mbox{Betti-linear} \\
 & & & & \cup &  &\cup \\
 & & & & \mbox{linear} & \subset & \mbox{lattice-linear} \\
 & & & & & & \cup\\
 & & & & & & \mbox{Scarf}
\end{array}
\]
Here, for example, Scarf means the calss of Scarf monomial ideals and linear 
means the class of monomial ideals having linear minimal free resolutions. 

If $M$ has a linear minimal free resolution $\mathbf{F}$, then $\mathbf{F}$ is 
pure, which implies that $M$ is homologically monotonic. Also, as in the definition of 
lattice-linearity, let $e$ be a basis element in homological degree greater than 
$1$. If $\lambda_{e, \widetilde{e}}\neq 0$ then we have that 
$\mbox{deg}(\mbox{mdeg}(e))=\mbox{deg}(\mbox{mdeg}(\widetilde{e}))+1$, 
which implies that $\mbox{mdeg}(\widetilde{e}) \lessdot \mbox{mdeg}(e)$ 
in $L_M$, and then $M$ is lattice-linear. This is also proved in Proposition 4.1 
of \cite{B:Cl}. 

 Let $M$ be the monomial ideal as in Example \ref{E:TBmany}, then 
$M$ is homologically monotonic but not rigid. Note that $M$ has a simplicial 
resolution, which is also linear. 

We will see from the next examples that all the inclusions in the above diagram 
are strict.
\end{remark}

\begin{example}\label{E:strict1}
Let $M$ be a monomial ideal in $S=k[a,b,c]$ minimally generated by monomials 
$m_1=a^2, m_2=ab, m_3=bc, m_4=c^2$. Then the lcm-lattice of $M$ is shown 
in Figure 3. 

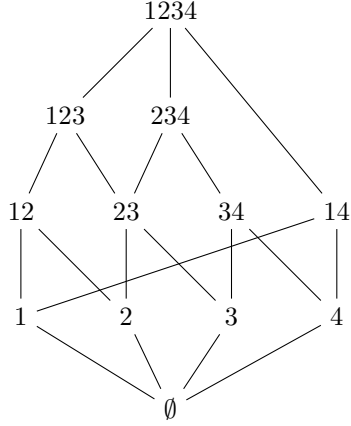
\begin{figure}
\begin{tikzpicture}[node distance=1.4cm]
\title{The strict1}
\node(1234)                                                                    {$1234$};
\node(234)    [below  of =1234]                                     {$234$};
\node(123)    [left of =234]                                            {$123$};
\node(23)      [below left =0.8cm and -0.1cm of 234]    {$23$};
\node(34)      [right of =23]                                            {$34$};
\node(14)      [right of =34]                                            {$14$};
\node(12)      [left of =23]                                               {$12$};
\node(1)        [below of =12]	                                    {$1$};
\node(2)        [below of =23]	                                    {$2$};
\node(3)        [below of =34]	                                    {$3$};
\node(4)        [below of =14]	                                    {$4$};
\node(0)        [below=4.8cm of 1234]                                 {$\emptyset$};

\draw(1234) --(234);
\draw(1234) --(123);
\draw(1234) --(14);
\draw(234) --(23);
\draw(234) --(34);
\draw(123) --(12);
\draw(123) --(23);
\draw(12) --(1);
\draw(12) --(2);
\draw(23) --(2);
\draw(23) --(3);
\draw(34) --(3);
\draw(34) --(4);
\draw(14) --(1);
\draw(14) --(4);
\draw(1) --(0);
\draw(2) --(0);
\draw(3) --(0);
\draw(4) --(0);

\end{tikzpicture}
\caption{The lcm-lattice $L_M$ of Example \ref{E:strict1}}
\end{figure}

Note that $\widetilde{H}(\Delta_{123}; k)=0$, $\widetilde{H}(\Delta_{234}; k)=0$ 
and $\widetilde{H}_1(\Delta_{1234}; k)\cong k$ with a basis $[12+23+34-14]$. 
Hence, by Theorem \ref{T:cbasis} we can take 
either $\emptyset, 1, 2, 3, 4, 12, 23$, $34$, $14$, $124+234$ 
or $\emptyset, 1, 2, 3, 4, 12, 23, 34, 14, 123+134$ as a Taylor basis. 
Note that $\mbox{mdeg}(12)=a^2b$ and $\mbox{mdeg}(14)=a^2c^2$. 
Then it is easy to see that $M$ is rigid and Betti-linear, but $M$ is not Scarf, 
or linear, or lattice-linear. Note that although $[d(124)]$ is a basis of 
$\widetilde{H}_1(\Delta_{1234}; k)$, $M$ does not have a simplicial resolution. 
However, $M$ do have a cellular resolution supported on a square 
with vertices $1, 2, 3, 4$. 
\end{example}

\begin{example}\label{E:strict2}
Let $\Delta=\langle 123, 34, 35, 45 \rangle$ be a simplicial complex. 
Let $M$ be the nearly Scarf monomial ideal defined by $\Delta$ as in \cite{B:PV}. 
Then the lcm-lattice of $M$ is shown in Figure 4. 

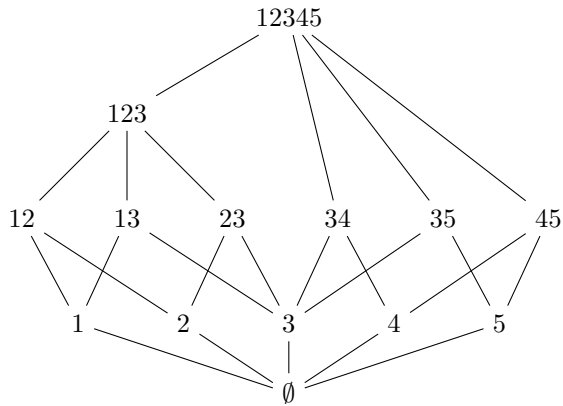
\begin{figure}
\begin{tikzpicture}[node distance=1.4cm]
\title{The strict2}
\node(12345)                                                                    {$12345$};
\node(123)    [below left =0.8cm and 1.2cm of 12345]  {$123$};
\node(13)      [below of =123]                                          {$13$};
\node(12)      [left of =13]                                                {$12$};
\node(23)      [right of =13]                                              {$23$};
\node(34)      [right of =23]                                              {$34$};
\node(35)      [right of =34]                                              {$35$};
\node(45)      [right of =35]                                              {$45$};
\node(3)        [below=3.6cm of 12345]                                 {$3$};
\node(2)        [left of =3]	                                    {$2$};
\node(1)        [left of =2]	                                    {$1$};
\node(4)        [right of =3]	                                    {$4$};
\node(5)        [right of =4]	                                    {$5$};
\node(0)        [below=4.5cm of 12345]                                 {$\emptyset$};

\draw(12345) --(123);
\draw(12345) --(35);
\draw(12345) --(34);
\draw(12345) --(45);
\draw(123) --(12);
\draw(123) --(13);
\draw(123) --(23);
\draw(12) --(1);
\draw(12) --(2);
\draw(13) --(1);
\draw(13) --(3);
\draw(23) --(2);
\draw(23) --(3);
\draw(34) --(3);
\draw(34) --(4);
\draw(35) --(3);
\draw(35) --(5);
\draw(45) --(4);
\draw(45) --(5);
\draw(1) --(0);
\draw(2) --(0);
\draw(3) --(0);
\draw(4) --(0);
\draw(5) --(0);

\end{tikzpicture}
\caption{The lcm-lattice $L_M$ of Example \ref{E:strict2}}
\end{figure}

Note that $\widetilde{H}_1(\Delta_{12345}; k)\cong k$ with a basis 
$[34-35+45]$. Hence, by Theorem \ref{T:cbasis} we can take 
$\emptyset$, $1, 2, 3, 4, 5$, $12$, $13$, $23$, $34$, $35$, $45$, 
$123$, $345$ as a Taylor basis. This Taylor basis satisfies the 
definition of lattice-linearity and then  $M$ 
is lattice-linear. Note that $\widetilde{H}_1(\Delta_{123}; k)\cong k$, 
so that $M$ is not homologically monotonic. 

Note that if we take $\emptyset$, $1, 2, 3, 4, 5$, $12$, $13$, $23$, 
$34$, $35$, $45$, $123$, $345+123$ as a Taylor basis, then the 
minimal free resolution of $S/M$ with this Taylor basis does not satisfy 
the definition of lattice-linearity. 

Similarly, let $\widetilde{\Delta}=\langle 123, 124, 134, 234, 45, 46, 56 \rangle$, 
and $\widetilde{M}$ the nearly Scarf monomial ideal defined by $\widetilde{\Delta}$ 
as in \cite{B:PV}. Then it is easy to see that 
$\widetilde{H}_2(\Delta_{123456}; k)\cong k$, 
$\widetilde{H}_1(\Delta_{123456}; k) \cong k$ 
and $M$ is lattice-linear. So Proposition \ref{P:HM} part (2) does not 
hold for lattice-linear monomial ideals. 
\end{example}

\begin{example}\label{E:notBetti}
Let $\Delta=\langle 12, 3 \rangle$ be a simplicial complex. Let $M$ be 
the nearly Scarf monomial ideal defined by $\Delta$ as in \cite{B:PV}. 
Then the lcm-lattice of $M$ is 
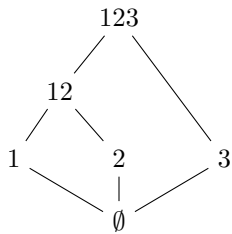
\[
\begin{tikzpicture}[node distance=1.4cm]
\title{not Betti-linear}
\node(123)                                                                    {$123$};
\node(12)      [below left =0.5cm and 0.1cm of 123]   {$12$};
\node(2)        [below=1.4cm of 123]                            {$2$};
\node(1)        [left of =2]	                                           {$1$};
\node(3)        [right of =2]	                                           {$3$};
\node(0)        [below=2.2cm of 123]                            {$\emptyset$};

\draw(123) --(12);
\draw(123) --(3);
\draw(12) --(1);
\draw(12) --(2);
\draw(1) --(0);
\draw(2) --(0);
\draw(3) --(0);
\end{tikzpicture}
\]
We can also use the labeling technique introduced in \cite{B:Ma} to get, for 
example, $M=(ab, ac, bcd)$. Indeed, many examples in this paper are 
obtained by first having the atomic lattice and then using the labeling 
technique in \cite{B:Ma} to get a monomial ideal. 

 Now we have that 
$\widetilde{H}_0(\Delta_{12}; k)\cong k$ with a basis $[-1+2]$ and 
$\widetilde{H}_0(\Delta_{123}; k) \cong k$ with a basis $[-1+3]=[-2+3]$. 
By Theorem \ref{T:basis} we see that $\Gamma_{123}=\{\lambda13+\mu23\}$, 
where $\lambda, \mu\in k$ and $\lambda+\mu\neq 0$. 
Since $d(\lambda13+\mu23)=-\lambda1-\mu2+(\lambda+\mu)3$ and both $1$ 
and $2$ are not covered by $123$ in $L_M=B_M$, it follows that $M$ is not 
Betti-linear. 
From this simple example we see that the class of  Betti-linear monomial ideals 
is not very large. 

With the above fixed basis of $\widetilde{H}_0(\Delta_{12}; k)$ and 
$\widetilde{H}_0(\Delta_{123}; k)$, by Definition \ref{D:posetres} we obtain
\[
\mathcal{D}(L_M): 0\to k^2 
\xrightarrow{\begin{pmatrix} -1& 0 \\ 1 &0 \\ 0 &1 \end{pmatrix}} 
k^3 \xrightarrow{\begin{pmatrix} 1 & 1 & 1 \end{pmatrix}} k, 
\]
which is not a complex. And then we obtain
\[
\mathcal{F}(L_M): 0\to S(-abc)\oplus S(-abcd) 
\xrightarrow{\begin{pmatrix} -c &0 \\ b & 0 \\ 0 & a \end{pmatrix} } 
S(-ab)\oplus S(-ac) \oplus S(-bcd) \ \ \ \ \ \ \ \ 
\]
\[
\ \ \ \ \ \ \ \ \ \ \ \ \ \ \ \ \ \ \ \ \ \ \ \ \ \ \ \ \ \ \ \ \ \ \ \ \ \ \ \ \ \ \ \ \ \ \ \ \ 
\ \ \ \ \ \ \ \ \ \ \ \ \ 
\xrightarrow{\begin{pmatrix} ab & ac & bcd \end{pmatrix} } S,
\]
which is not a minimal free resolution of $S/M$. 
\end{example}

Next we will use results about Taylor bases to prove that if $M$ is 
Betti-linear then $\mathcal{F}(L_M)$ as constructed in Definition \ref{D:posetres}, 
is a minimal free resolution of $S/M$. To make the proof easier to understand, 
we first prove the result for rigid monomial ideals. 

\begin{theorem}[\cite{B:CM2}]\label{T:rigid}
Let $M$ be a rigid monomial ideal minimally generated by $r$ monomials. 
Then the poset construction $\mathcal{D}(L_M)$ is an $r$-frame,  
$\mathcal{F}(L_M)$ equals the $M$-homogenization of $\mathcal{D}(L_M)$, 
and $\mathcal{F}(L_M)$ is a minimal free resolution of $S/M$. 
\end{theorem}

\begin{proof}
Since $M$ is rigid, it follows that $M$ is Betti-linear and every minimal free 
resolution of $S/M$ satisfies the definition of Betti-linearity. 
For any $m\neq 1 \in B_M$, let $[c_m]$ be a basis of 
$\widetilde{H}(\Delta_m; k)=\widetilde{H}_i(\Delta_m; k)\cong k$ for some 
$i\geq -1$; let $\mathbf{F}$ be a minimal free resolution of $S/M$ with a 
Taylor basis $B$ such that in homological degree $i+2$ at $S(-m)$, the 
Taylor basis element is $e_m$. By Theorem \ref{T:basis}, without the loss 
of generality, after a rescaling of the Taylor basis elements in $B$, we can 
assume that $[d(e_m)]=[c_m]$ in $\widetilde{H}(\Delta_m; k)$ for all 
$m\neq 1 \in B_M$, and that the frame of $\mathbf{F}$ 
in homological degrees $0$ and $1$ is given by 
$k^r \xrightarrow{(\begin{smallmatrix} 1& \cdots & 1 \end{smallmatrix} )} k$. 
Hence, there is a one-to-one correspondence between the basis elements 
$e_m$ of the frame of $\mathbf{F}$ and the basis elements $[c_m]$ of 
$\mathcal{D}(L_M)$. 

For any $m\neq 1 \in B_M$ such that $m$ is not an atom, we have that 
$\widetilde{H}_i(\Delta_m; k)\cong k$ for some $i\geq 0$. 
Let 
\[
d(e_m)=\lambda_1e_{\beta_1}+\cdots+\lambda_te_{\beta_t}, 
\]
where $\lambda_1, \ldots, \lambda_t$ are some nonzero scalars in $k$ 
and $e_{\beta_1}, \ldots, e_{\beta_t}\in B$ are some Taylor basis elements 
in homological degree $i+1$ with $\mbox{mdeg}(e_{\beta_j})=\beta_j$ for 
any $1\leq j \leq t$. By the definition of Betti-linearity we see that 
$\beta_1, \ldots, \beta_t$ are covered by $m$ in $B_M$, so that 
$\beta_1, \ldots, \beta_t$ are some maximal elements in $\mathcal{B}(m)$, 
which implies that the cycle $\lambda_1e_{\beta_1}+\cdots+\lambda_te_{\beta_t}$ 
is in $\widetilde{\Delta}_m$. Since for any $1\leq j\leq t$ we have that 
$\mbox{supp}(e_{\beta_j})\subseteq A_{\beta_j}$, it follows that 
\[
\varphi_{i+2, m}^{\beta_j}([\lambda_1e_{\beta_1}+\cdots+\lambda_te_{\beta_t}])
=[d(\lambda_je_{\beta_j})]=\lambda_j[d(e_{\beta_j})]=\lambda_j[c_{\beta_j}].
\]
Thus, we have that 
\begin{align*}
\varphi_{i+2}([c_m])&= \varphi_{i+2}([d(e_m)])\\
                                 &=\varphi_{i+2, m}([\lambda_1e_{\beta_1}+\cdots+\lambda_te_{\beta_t}])\\ 
                                 &=\sum_{j=1}^t\varphi_{i+2, m}^{\beta_j}([\lambda_1e_{\beta_1}+\cdots+\lambda_te_{\beta_t}])\\
                                 &=\lambda_1[c_{\beta_1}]+\cdots +\lambda_t[c_{\beta_t}].
\end{align*}
So $\mathcal{D}(L_M)$ is equal to the frame of $\mathbf{F}$, which is an $r$-frame. 
Note that the multigraded free $S$-module $\mathcal{F}_i$ in $\mathcal{F}(L_M)$ 
coincides with the multigraded free $S$-module $F_i$ in $\mathbf{F}$, and the 
differential maps in $\mathbf{F}$ coincide with the multigraded homomorphisms 
defined in $\mathcal{F}(L_M)$. Therefore, $\mathcal{F}(L_M)$ equals $\mathbf{F}$, 
and then $\mathcal{F}(L_M)$ is a minimal free  resolution of $S/M$. 
Since $\mathbf{F}$ equals the $M$-homogenization of the frame of $\mathbf{F}$, 
it follows that $\mathcal{F}(L_M)$ equals the $M$-homogenization of $\mathcal{D}(L_M)$. 
\end{proof}

Theorem \ref{T:rigid} is also proved in \cite{B:CM2}, where the proof takes over 
six pages. Here, by using results about Taylor bases our proof is much shorter. 

\begin{theorem}[\cite{B:Cl},\cite{B:Wo}]\label{T:Betti}
Let $M$ be a Betti-linear monomial ideal minimally generated by $r$ monomials. 
Then the poset construction $\mathcal{D}(L_M)$ is an $r$-frame,  
$\mathcal{F}(L_M)$ equals the $M$-homogenization of $\mathcal{D}(L_M)$, 
and $\mathcal{F}(L_M)$ is a minimal free resolution of $S/M$. 
\end{theorem}

\begin{proof}
For any $m\neq 1\in B_M$, suppose that $\widetilde{H}_i(\Delta_m; k)\cong k^{t_{m, i}}$ 
for some $t_{m, i}>0$. Let $[c_{m, i}^1], \ldots, [c_{m, i}^{t_{m,i}}]$ be a basis of 
$\widetilde{H}_i(\Delta_m; k)$. Let $\mathbf{F}$ be a minimal free resolution of $S/M$ 
with a Taylor basis $B$ such that $B$ satisfies the definition of Betti-linearity. 
Let $e_{m, i}^1, \ldots, e_{m, i}^{t_{m, i}} \in B$ be the Taylor basis elements in 
homological degree $i+2$ at $S(-m)^{t_{m, i}}$. By Theorem \ref{T:basis} we see that 
$[d(e_{m, i}^1)], \ldots, [d(e_{m, i}^{t_{m, i}})]$ is also a basis of 
$\widetilde{H}_i(\Delta_m; k)$.  Without the loss of  generality, after using an invertible 
$k$-linear map acting on the Taylor basis elements $e_{m, i}^1, \ldots, e_{m, i}^{t_{m, i}} \in B$, 
we can assume that $[d(e_{m, i}^j)]=[c_{m, i}^j]$ for any $1\leq j \leq t_{m, i}$. 
We can also assume that the frame of $\mathbf{F}$ in homological degrees $0$ and 
$1$ is given by $k^r \xrightarrow{(\begin{smallmatrix} 1 & \cdots & 1 \end{smallmatrix})} k$. 
Hence, there is a one-to-one correspondence between the basis elements $e_{m, i}^j$ 
of the frame of $\mathbf{F}$ and the basis elements $[c_{m, i}^j]$ of $\mathcal{D}(L_M)$. 

For any $m\neq 1\in B_M$ such that $m$ is not an atom, we have that 
$\widetilde{H}_i(\Delta_m; k)\neq 0$ for some $i\geq 0$. Let $e_m\in B$ be a Taylor 
basis element in homological degree $i+2$ with $\mbox{mdeg}(e_m)=m$. Let 
\[
d(e_m)=\lambda_1e_1+\cdots+\lambda_te_t, 
\]
where $\lambda_1, \ldots, \lambda_t$ are some nonzero scalars in $k$ 
and $e_1, \ldots, e_t \in B$ are some Taylor basis elements 
in homological degree $i+1$. Let $e_m, e_1, \ldots, e_t$ correspond to 
basis elements $[c_m], [c_1], \ldots, [c_t]$, respectively, in $\mathcal{D}(L_M)$. 
Next we want to show that 
\[
\varphi_{i+2}([c_m])=\lambda_1[c_1]+\cdots+\lambda_t[c_t].
\]

Indeed, since the Taylor basis $B$ satisfies the definition of Betti-linearity, it follows 
that $\mbox{mdeg}(e_1), \ldots, \mbox{mdeg}(e_t)$ are covered by $m$ in $B_M$, 
i.e., $\mbox{mdeg}(e_1), \ldots, \mbox{mdeg}(e_t)$ are some maximal elements 
in $\mathcal{B}(m)$. Let $\beta_1, \ldots, \beta_l$ be all the maximal elements in 
$\mathcal{B}(m)$, then for any $1\leq s \leq t$ there exists $1\leq j\leq l$ such that 
$\mbox{mdeg}(e_s)=\beta_j$, which implies that 
$\overline{\mbox{supp}(e_s)} = A_{\beta_j}$. Assume that there exist 
$1\leq j_1 \neq j_2 \leq l$ such that  $\mbox{supp}(e_s)\subseteq A_{\beta_{j_1}}$ 
and $\mbox{supp}(e_s) \subseteq A_{\beta_{j_2}}$. Then we have that 
$\mbox{supp}(e_s)\subseteq A_{\beta_{j_1}}\cap A_{\beta_{j_2}}
=A_{\beta_{j_1}\wedge \beta_{j_2}}$, so that 
$A_{\beta_j} =\overline{\mbox{supp}(e_s)}\subseteq A_{\beta_{j_1}\wedge \beta_{j_2}}$, 
and then $\beta_j\leq \beta_{j_1}\wedge \beta_{j_2}$, which contradicts to the 
assumption that $\beta_1, \ldots, \beta_l$ are maximal elements in $\mathcal{B}(m)$. 
Hence, for any $1\leq s \leq t$ there exists a unique $1\leq j\leq l$ such that 
$\mbox{supp}(e_s)\subseteq A_{\beta_j}$, $\mbox{mdeg}(e_s)=\beta_j$ and 
$[d(e_s)]=[c_s]$ is a basis element of $\widetilde{H}_{i-1}(\Delta_{\beta_j}; k)$. 
Thus, we have that 
\begin{align*}
\varphi_{i+2}([c_m])&= \varphi_{i+2}([d(e_m)])\\
                                 &=\varphi_{i+2, m}([\lambda_1e_1+\cdots+\lambda_te_t])\\ 
                                 &=\sum_{j=1}^l\varphi_{i+2, m}^{\beta_j}([\lambda_1e_1+\cdots+\lambda_te_t])\\
                                 &=[d(\lambda_1e_1)]+\cdots+[d(\lambda_te_t)]\\
                                 &=\lambda_1[d(e_1)]+\cdots+\lambda_t[d(e_t)]\\
                                 &=\lambda_1[c_1]+\cdots +\lambda_t[c_t].
\end{align*}

So $\mathcal{D}(L_M)$ is equal to the frame of $\mathbf{F}$, which is an $r$-frame. 
Similar to the argument in Theorem \ref{T:rigid}, we have that $\mathcal{F}(L_M)$ 
equals the $M$-homogenization of $\mathcal{D}(L_M)$ and is a minimal free resolution 
of $S/M$. 
\end{proof}

There is a proof of Theorem \ref{T:Betti} for lattice-linear monomial ideals in 
\cite{B:Cl}.  And in \cite{B:Wo} a similar proof is given for Betti-linear monomial 
ideals. Both proofs are very technical and not very easy to understand. 
Our proof here is different from theirs. We hope that this new proof can 
give us a better understanding about Betti-linear monomial ideals. 

\begin{remark}\label{R:Betti}
It is easy to see that if $\mathcal{F}(L_M)$ is a poset resolution of $S/M$, 
then $M$ is Betti-linear. So the class of Betti-linear monomial ideals is 
exactly the set of monomial ideals for which the poset construction induces 
a minimal free resolution. 
\end{remark}

\begin{remark}\label{R:pure}
By the proof of Proposition \ref{P:HM} (3), it is easy to see that if $M$ is 
homologically monotonic then every minimal free resolution of $S/M$ 
satisfies the definition of Betti-linearity. So every minimal free resolution 
of a homologically monotonic monomial ideal can be obtained as a poset 
resolution, which is not the case for Betti-linear monomial ideals, as is shown by 
Example \ref{E:strict2}. 
From Example \ref{E:notBetti} we see that the class of Betti-linear monomial ideals 
is not very large. However, it contains the class of homologically monotonic 
monomial ideals. In particular, for monomial ideals with pure minimal free 
resolutions, which play an important role in Boij-S\"{o}derberg theory, 
poset construction gives a formula for their minimal free resolutions. 
\end{remark}

\subsection{An Approximation Formula for Minimal Free Resolutions of a Monomial Ideal and Homology-linear Monomial Ideals}
In this subsection we will give a construction similar to the poset construction. 
With this new construction we obtain an approximation formula for minimal 
free resolutions of a monomial ideal. After that, we introduce a new class of 
monomial ideals, called homology-linear monomial ideals, which is similar to 
Betti-linear monomial ideals. 

Let $M$ be a monomial ideal minimally generated by monomials $m_1, \ldots, m_r$. 
For any $1\neq m \in B_M$ such that $m$ is not an atom, we have that 
$\widetilde{H}_i(\Delta_m; k)\neq 0$ for some $i\geq 0$. 
Let 
\[
\mathcal{B}_i(m)=\{\widetilde{m}\in B_M| 1\neq \widetilde{m}<m \ \mbox{and} \ 
\widetilde{H}_{i-1}(\Delta_{\widetilde{m}}; k) \neq 0\}, 
\]
then we have that 
$\mathcal{B}_i(m)\subseteq \mathcal{B}(m)$, and by Proposition \ref{P:homdelta} (2),  
it is easy to see that $\mathcal{B}_i(m)\neq \emptyset$. 
Let $\gamma_1, \ldots, \gamma_t$ be the maximal elements in $\mathcal{B}_i(m)$. 
Let 
\[
\Delta_m^{(i)}=\langle A_{\gamma_1}, \ldots, A_{\gamma_t} \rangle, 
\]
then $\Delta_m^{(i)}$ is a subcomplex of $\Delta_m$. These notations will be 
used throughout this subsection.

\begin{proposition}\label{P:surjective}
Let $1\neq m \in B_M$ such that $m$ is not an atom. Let 
$\sigma_m^{(i)}: \widetilde{H}_i(\Delta_m^{(i)}; k) \to \widetilde{H}_i(\Delta_m; k)$ 
be the $k$-linear map induced by the inclusion $\tau: \Delta_m^{(i)} \to \Delta_m$. 
Then $\sigma_m^{(i)}$ is surjective. Moreover, let $\mathcal{J}_i(m)$ be the set 
of all $\widetilde{m}< m \in L_M$ such that there does not exist $\gamma_j$ 
with $\widetilde{m}<\gamma_j$; if there does not exist $\widetilde{m}\in \mathcal{J}_i(m)$ 
such that $\widetilde{H}_i(\Delta_{\widetilde{m}}; k)\neq 0$, then $\sigma_m^{(i)}$ 
is an isomorphism. 
\end{proposition}

\begin{proof}
By the proof of Theorem \ref{T:main1} we have that 
\[
\widetilde{C}(\Delta_m; k)[-1]\cong (\mathbb{U}_m, d) \bigoplus (\bigoplus \limits_{1\neq \widetilde{m} <m} 
(\mathcal{E}_{\widetilde{m}}, d)).
\]
Let $\Pi=\{\widetilde{m} \in L_M| \widetilde{m} \leq \gamma_j  \mbox{ for some} \ \gamma_j\}$. 
Let $(\mathbb{U}_m^{(i)}; d)$ be the complex obtained by applying the boundary 
map $d$ to the Taylor basis elements in $\bigcup\limits_{\widetilde{m} \in \Pi} \Gamma_{\widetilde{m}}$.  
Then by Corollary \ref{C:main1c2} we have that 
\[
\widetilde{C}(\Delta_m^{(i)}; k)[-1]\cong (\mathbb{U}_m^{(i)}, d) \bigoplus (\bigoplus \limits_{1\neq \widetilde{m} \in \Pi} 
(\mathcal{E}_{\widetilde{m}}, d)).
\]
It is easy to see that $\widetilde{C}(\Delta_m^{(i)}; k)[-1]$, $(\mathbb{U}_m^{(i)}, d)$ , 
$\bigoplus\limits_{1\neq \widetilde{m} \in \Pi}(\mathcal{E}_{\widetilde{m}}, d)$ 
are subcomplexes of $\widetilde{C}(\Delta_m; k)[-1]$, $(\mathbb{U}_m, d)$ , 
$\bigoplus\limits_{1\neq \widetilde{m} <m}(\mathcal{E}_{\widetilde{m}}, d)$, 
respectively. Let $f: \widetilde{C}(\Delta_m^{(i)}; k)[-1] \to \widetilde{C}(\Delta_m; k)[-1]$, 
$g: (\mathbb{U}_m^{(i)}, d) \to (\mathbb{U}_m, d)$, 
$h: \bigoplus\limits_{1\neq \widetilde{m} \in \Pi}(\mathcal{E}_{\widetilde{m}}, d) 
\to \bigoplus\limits_{1\neq \widetilde{m} <m}(\mathcal{E}_{\widetilde{m}}, d)$ 
be the natural inclusion maps of complexes, then we have the following 
commutative diagram:
\[
\begin{tikzcd}
\widetilde{C}(\Delta_m^{(i)}; k)[-1] 
\arrow[r,"\cong "] \arrow[d,"f"]
& (\mathbb{U}_m^{(i)}, d) \bigoplus (\bigoplus \limits_{1\neq \widetilde{m} \in \Pi} (\mathcal{E}_{\widetilde{m}}, d)) 
\arrow[d,"(g\  h)"] \\
\widetilde{C}(\Delta_m; k)[-1] 
\arrow[r,"\cong"]
&(\mathbb{U}_m, d) \bigoplus (\bigoplus \limits_{1\neq \widetilde{m} <m} (\mathcal{E}_{\widetilde{m}}, d)).
\end{tikzcd}
\]
Note that $\bigoplus \limits_{1\neq \widetilde{m} <m} (\mathcal{E}_{\widetilde{m}}, d)$, 
$\bigoplus \limits_{1\neq \widetilde{m} \in \Pi} (\mathcal{E}_{\widetilde{m}}, d)$ 
are trivial complexes , and $\sigma_m^{(i)}$ is the map of homology induced by $f$. 
Hence, to show that $\sigma_m^{(i)}$ is surjective, it 
suffices to show that $g_{\ast}: \widetilde{H}_{i+1}(\mathbb{U}_m^{(i)}, d) 
\to \widetilde{H}_{i+1}(\mathbb{U}_m, d)$ is surjective. 

Next, for convenience, we use $\widetilde{d}$ to denote the boundary 
map $d$ in $(\mathbb{U}_m^{(i)}, d)$. By the definition of $\Pi$ it is easy 
to see that $\bigcup\limits_{\widetilde{m} \in \Pi}\Gamma_{\widetilde{m}}$ 
and $\bigcup\limits_{\widetilde{m} <m} \Gamma_{\widetilde{m}}$ have the 
same Taylor basis elements of dimension $i$, so that 
$(\mathbb{U}_m^{(i)})_{i+1}=(\mathbb{U}_m)_{i+1}$, which implies that 
$\mbox{Ker}\widetilde{d}_{i+1}=\mbox{Ker}d_{i+1}$. 
Since $(\mathbb{U}_m^{(i)})_{i+2}\subseteq (\mathbb{U}_m)_{i+2}$, it 
follows that $\mbox{Im}\widetilde{d}_{i+2} \subseteq \mbox{Im}d_{i+2}$. 
Thus, we have that 
\[
g_{\ast}: \widetilde{H}_{i+1}(\mathbb{U}_m^{(i)}, d) 
=\frac{\mbox{Ker}\widetilde{d}_{i+1}}{\mbox{Im}\widetilde{d}_{i+2}} 
\to \widetilde{H}_{i+1}(\mathbb{U}_m, d)
=\frac{\mbox{Ker} d_{i+1}}{\mbox{Im}d_{i+2}}
\]
is surjective, and then $\sigma_m^{(i)}$ is surjective. 

If there does not exist $\widetilde{m}\in \mathcal{J}_i(m)$ with 
$\widetilde{H}_i(\Delta_{\widetilde{m}}; k)\neq 0$, then it is easy to see that 
$(\mathbb{U}_m^{(i)})_{i+2}=(\mathbb{U}_m)_{i+2}$, which implies that 
$\mbox{Im}\widetilde{d}_{i+2}=\mbox{Im}d_{i+2}$, so that $g_{\ast}$ is 
the identity map, and then $\sigma_m^{(i)}$ is an isomorphism. 
\end{proof}

\begin{remark}\label{R:surjective}
Proposition \ref{P:surjective} can also be proved by using algebraic topology, 
but it is much easier to prove it by using the atomic lattice resolution theory. 
In this remark we give another proof of Proposition \ref{P:surjective}. 

Let $B$ be a Taylor basis of a minimal free resolution of $S/M$. 
Let $e_1, \ldots, e_p\in B$ be the Taylor basis elements of dimension $i+1$ 
and multidegree $m$, then by Theorem \ref{T:basis} we have that 
$[d(e_1)], \ldots, [d(e_p)]$ is a basis of $\widetilde{H}_i(\Delta_m; k)$. 
Let $f_1, \ldots, f_q \in B$ be the Taylor basis elements of dimension $i$ 
and multidegree less than $m$. 
Then for any $1\leq l \leq q$ there exists $1\leq s \leq t$ such that 
$\mbox{supp}(f_l)\subseteq A_{\gamma_s}$. 
Note that for any $1\leq j \leq p$ there exist $\lambda_{j1}, \ldots, \lambda_{jq} \in k$ 
such that 
\[
d(e_j)=\lambda_{j1}f_1+\cdots+\lambda_{jq}f_q.
\]
Hence, $d(e_1), \ldots, d(e_p)$ are cycles in $\Delta_m^{(i)}$, which implies that 
$[d(e_1)], \ldots, [d(e_p)]$ are in $\widetilde{H}_i(\Delta_m^{(i)}; k)$. 
So $\sigma_m^{(i)}$ is surjective and $\widetilde{H}_i(\Delta_m; k)$ can be 
naturally embedded as a subspace of $\widetilde{H}_i(\Delta_m^{(i)}; k)$.  

Let $g_1, \ldots, g_s\in (\bigcup\limits_{\widetilde{m}\in \mathcal{J}_i(m)}\Gamma_{\widetilde{m}})
\bigcup \Gamma_m \subseteq B$ be the Taylor basis elements of dimension $i+1$. 
Let $V$ be the $k$-vector space with basis $g_1, \ldots, g_s$. 
Then by Theorem \ref{T:main2} and the proof of Theorem \ref{T:main1}, 
we have that $(\mathbb{V}_m)_{i+2}=(\mathbb{U}_m^{(i)})_{i+2}\oplus V$. 
Since $(\mathbb{V}_m; d)$ is exact, it follows that $[d(g_1)], \ldots, [d(g_s)]$ 
is a basis of $\widetilde{H}_{i+1}(\mathbb{U}_m^{(i)}, d)\cong \widetilde{H}_i(\Delta_m^{(i)}; k)$. 
Note that $d(g_1), \ldots, d(g_s)$ are cycles in $\Delta_m^{(i)}$, so that 
for any $\widetilde{m}\in \mathcal{J}_i(m)$, $\widetilde{H}_i(\Delta_{\widetilde{m}}; k)$ 
can also be embedded as a subspace of $\widetilde{H}_i(\Delta_m^{(i)}; k)$. 
With this embedding, we have the following formula:
\[
\widetilde{H}_i(\Delta_m^{(i)}; k) \cong (\bigoplus\limits_{\widetilde{m}\in \mathcal{J}_i(m)}
\widetilde{H}_i(\Delta_{\widetilde{m}}; k))\bigoplus \widetilde{H}_i(\Delta_m; k).
\]
So if there does not exist $\widetilde{m}\in \mathcal{J}_i(m)$ with 
$\widetilde{H}_i(\Delta_{\widetilde{m}}; k)\neq 0$, then $\sigma_m^{(i)}$ is an isomorphism. 
\end{remark}

\begin{example}\label{E:surjective}
Let $M=(ab, ac, b^2c)$ be a monomial ideal in $S=k[a, b, c]$. 
Then the lcm-lattice of $M$ is shown in Example \ref{E:notBetti}. 
We have that $\widetilde{H}_0(\Delta_{12}; k)\cong k$ with a basis $[-1+2]$ and 
$\widetilde{H}_0(\Delta_{123}; k) \cong k$ with a basis $[-1+3]=[-2+3]$. 
Since $\Delta_{123}^{(0)}=\langle 1, 2, 3 \rangle$, it follows that 
$\widetilde{H}_0(\Delta_{123}^{(0)}; k)\cong k^2$ with a basis $[-1+2], [-1+3]$. 
Then $\sigma_{123}^{(0)}: \widetilde{H}_0(\Delta_{123}^{(0)}; k) 
\to \widetilde{H}_0(\Delta_{123}; k)$ is surjective and we have that
$\widetilde{H}_0(\Delta_{123}^{(0)}; k) \cong 
\widetilde{H}_0(\Delta_{12}; k) \oplus \widetilde{H}_0(\Delta_{123}; k)$.
\end{example}

Next, similar to Definition \ref{D:posetres} we define a sequence of $k$-vector 
spaces and $k$-linear maps from the lcm-lattice $L_M$: 
\[
\mathcal{R}(L_M): \cdots \to \mathcal{R}_i 
\xrightarrow{\psi_i} \mathcal{R}_{i-1} 
\to \cdots \to \mathcal{R}_1 \xrightarrow{\psi_1} \mathcal{R}_0. 
\]

\begin{definition}\label{D:RLM}
For any $i\geq 0$ let $\mathcal{R}_i=\mathcal{D}_i$ as in Definition \ref{D:posetres}. 
Also, let $\psi_1=\varphi_1$. 

For any $i\geq 2$ we define 
\[
\psi_i: \mathcal{R}_i=\bigoplus\limits_{1\neq m \in L_M}\widetilde{H}_{i-2}(\Delta_m; k) 
\to \mathcal{R}_{i-1}=\bigoplus\limits_{1\neq m \in L_M}\widetilde{H}_{i-3}(\Delta_m; k)
\]
componentwise as follows. 
First for any $m\neq 1 \in B_M$ we fix a basis for $\widetilde{H}(\Delta_m; k)$. 
Suppose that $\widetilde{H}_{i-2}(\Delta_m; k) \neq 0$ with $i\geq 2$, then we have that 
$\mbox{rk}(m)\geq 2$.  Next we componentwise define
\[
\psi_{i,m}: \widetilde{H}_{i-2}(\Delta_m; k) 
\to \mathcal{R}_{i-1}=\bigoplus\limits_{1\neq m \in L_M}\widetilde{H}_{i-3}(\Delta_m; k).
\] 
Let $\Delta_m^{(i-2)}=\langle A_{\gamma_1}, \ldots, A_{\gamma_t} \rangle$.
By Proposition \ref{P:surjective} we have that 
\[
\sigma_m^{(i-2)}: \widetilde{H}_{i-2}(\Delta_m^{(i-2)}; k) \to \widetilde{H}_{i-2}(\Delta_m; k)
\]
is surjective, so that $\widetilde{H}_{i-2}(\Delta_m^{(i-2)}; k) \neq 0$, 
which implies that $t\geq 2$. 

First we define $\psi_{i,m}^{\gamma_1}: \widetilde{H}_{i-2}(\Delta_m; k) 
\to \widetilde{H}_{i-3}(\Delta_{\gamma_1}; k)$. Similar to Definition \ref{D:posetres}, 
let $\Delta_1=\langle A_{\gamma_1} \rangle$ and 
$\Delta_2=\langle A_{\gamma_2}, \ldots, A_{\gamma_t} \rangle$, then we have that 
$\Delta_1 \cup \Delta_2=\Delta_m^{(i-2)}$ and 
$\Delta_1 \cap \Delta_2=\langle A_{\gamma_1\wedge \gamma_2}, 
\ldots, A_{\gamma_1\wedge \gamma_t} \rangle$, which is a subcomplex of 
$\Delta_{\gamma_1}$. 
From the Mayer-Vietoris sequence we have the connecting 
homomorphism 
\[
\delta: \widetilde{H}_{i-2}(\Delta_m^{(i-2)}; k) \to  
\widetilde{H}_{i-3}(\Delta_1 \cap \Delta_2; k). 
\]
 Let $\iota: \widetilde{H}_{i-3}(\Delta_1 \cap \Delta_2; k) \to 
\widetilde{H}_{i-3}(\Delta_{\gamma_1}; k)$ 
be the $k$-linear map induced by the inclusion map
$\Delta_1 \cap \Delta_2 \to \Delta_{\gamma_1}$. 
Let $[e_1], \ldots, [e_p]$ be the fixed basis of $\widetilde{H}_{i-2}(\Delta_m; k)$. 
Let $[f_1], \ldots, [f_q]$ be the fixed basis of $\widetilde{H}_{i-3}(\Delta_{\gamma_1}; k)$. 
Since $\sigma_m^{(i-2)}$ is surjective, there exist $[g_1], \ldots, [g_p] 
\in \widetilde{H}_{i-2}(\Delta_m^{(i-2)}; k)$ such that 
$\sigma_m^{(i-2)}([g_j])=[e_j]$ for any $1\leq j \leq p$, i.e., 
$[g_j]=[e_j] \in \widetilde{H}_{i-2}(\Delta_m; k)$ and $g_j$ is 
a chain in $\Delta_m^{(i-2)}$. 
For any $1\leq j \leq p$ there exist scalars $\lambda_{j1}, \ldots, \lambda_{jq} \in k$ 
such that 
\[
(\iota \circ\delta)([g_j])=\lambda_{j1}[f_1]+\cdots+\lambda_{jq}[f_q].
\]
Then for any $1\leq j\leq p$ we define
\[
\psi_{i,m}^{\gamma_1}([e_j])=\lambda_{j1}[f_1]+\cdots+\lambda_{jq}[f_q],
\]
which determines a $k$-linear map 
$\psi_{i,m}^{\gamma_1}: \widetilde{H}_{i-2}(\Delta_m; k) 
\to \widetilde{H}_{i-3}(\Delta_{\gamma_1}; k)$.  

Similarly, for any $2\leq s\leq t$,  we can define 
a $k$-linear map $\psi_{i,m}^{\gamma_s}: \widetilde{H}_{i-2}(\Delta_m; k) 
\to \widetilde{H}_{i-3}(\Delta_{\gamma_s}; k)$. 
And for any $1\neq \widetilde{m} \in L_M$ with $\widetilde{m} \notin 
\{\gamma_1, \ldots, \gamma_t\}$, we define 
$\psi_{i,m}^{\widetilde{m}}: \widetilde{H}_{i-2}(\Delta_m; k) 
\to \widetilde{H}_{i-3}(\Delta_{\widetilde{m}}; k)$ to be the zero map. 

Thus, we can define
\[
\psi_{i,m}=\sum_{1\neq \widetilde{m} \in L_M} \psi_{i, m}^{\widetilde{m}} 
=\psi_{i,m}^{\gamma_1} +\cdots +\psi_{i,m}^{\gamma_t}, 
\]
and then $\psi_i$ is defined componentwise. 
So we obtain a sequence of $k$-vector spaces and $k$-linear maps: 
\[
\mathcal{R}(L_M): \cdots \to \mathcal{R}_i 
\xrightarrow{\psi_i} \mathcal{R}_{i-1} 
\to \cdots \to \mathcal{R}_1 \xrightarrow{\psi_1} \mathcal{R}_0. 
\]

For any $i\geq 0$ let $\mathcal{G}_i=\mathcal{R}_i\otimes S$. 
Let $f_0$ be the basis element of $\mathcal{R}_0$. For any $1\leq j \leq r$, 
let $f_j$ be the basis element of $\widetilde{H}_{-1}(\Delta_{m_j}; k)$. 
We still use $f_0$ as the basis of $\mathcal{G}_0$ and use $f_1, \ldots, f_r$ 
as the basis of $\mathcal{G}_1$. To make $\mathcal{G}_0$ and 
$\mathcal{G}_1$ multigraded free $S$-modules, we define 
$\mbox{mdeg}(f_0)=1$ and $\mbox{mdeg}(f_j)=m_j$ for $1\leq j\leq r$. 
For any $i\geq 2$, let $[e_1], \ldots, [e_p]$ be a basis of $\widetilde{H}_{i-2}(\Delta_m; k)$. 
We still use $[e_1], \ldots, [e_p]$ as a basis  of $\widetilde{H}_{i-2}(\Delta_m; k)\otimes S$, 
and we define $\mbox{mdeg}([e_j])=\mbox{mdeg}(e_j)=m$ for any $1\leq j\leq p$.  
Thus, $\mathcal{G}_i$ is a multidgraded free $S$-module with a multigraded basis. 

We define $\eth_1: \mathcal{G}_1 \to \mathcal{G}_0$ by 
$S^r \xrightarrow{(\begin{smallmatrix} m_1 & \cdots & m_r \end{smallmatrix})} S$. 
For any $i\geq 2$ and $1\leq s \leq t$, we can homogenize the map $\psi_{i,m}^{\gamma_s}$ 
to get 
\[
\eth_{i,m}^{\gamma_s}: \widetilde{H}_{i-2}(\Delta_m; k)\otimes S \to 
\widetilde{H}_{i-3}(\Delta_{\gamma_s}; k) \otimes S 
\]
with $\eth_{i, m}^{\gamma_s}=\frac{m}{\gamma_s}\otimes \psi_{i,m}^{\gamma_s}$. 
And then  we define $\eth_i: \mathcal{G}_i \to \mathcal{G}_{i-1}$ 
componentwise by $\eth_{i, m}$, where 
$\eth_{i, m}=\eth_{i, m}^{\gamma_1} +\cdots +\eth_{i,m}^{\gamma_t}$. 
Thus, for any $i\geq 1$, $\eth_i$ is multigraded. 

So we obtain a sequence of multigraded free $S$-modules and multigraded homomorphisms:
\[
\mathcal{G}(L_M):\  \cdots \to \mathcal{G}_i \xrightarrow{\eth_i} \mathcal{G}_{i-1} 
\to \cdots \to \mathcal{G}_1 \xrightarrow{\eth_1} \mathcal{G}_0.
\]
\end{definition}

\begin{remark}\label{R:RLM}
If we fix a basis for all $\widetilde{H}(\Delta_m; k)$, then in Definition \ref{D:posetres} 
the maps $\varphi_i:\mathcal{D}_i \to \mathcal{D}_{i-1}$ are uniquely defined, but 
in Definition \ref{D:RLM} the maps $\psi_i: \mathcal{R}_i \to \mathcal{R}_{i-1}$ may 
not be uniquely defined. Indeed, if $\sigma_m^{(i-2)}$ is an isomorphism, 
then $\psi_{i,m}$ is uniquely defined; otherwise, 
if $\sigma_m^{(i-2)}$ is not an isomorphism, then $\psi_{i,m}$ may 
depend on the preimages $[g_1], \ldots, [g_p]\in \widetilde{H}_{i-2}(\Delta_m^{(i-2)}; k)$ 
we choose for $[e_1], \ldots, [e_p]\in \widetilde{H}_{i-2}(\Delta_m; k)$,  and then 
$\psi_{i,m}$ may not be uniquely defined. This is illustrated in  example \ref{E:RLM}. 
Note that if we also fix the preimages $[g_1], \ldots, [g_p]$, 
then $\psi_{i,m}$ is uniquely defined.
\end{remark}

\begin{example}\label{E:RLM}
Let $M=(ab, ac, b^2c)$ be as in Example \ref{E:surjective} and the lcm-lattice 
of $M$ is shown in Example \ref{E:notBetti}. We know from Example \ref{E:notBetti} 
that $M$ is not Betti-linear. From Example \ref{E:surjective} we see that 
$\sigma_{123}^{(0)}: \widetilde{H}_0(\Delta_{123}^{(0)}; k) 
\to \widetilde{H}_0(\Delta_{123}; k)$ is not an isomorphism. 
Let $e=[-1+3]=[-2+3]$ be a basis of $\widetilde{H}_0(\Delta_{123}; k)$. 
Note that $[-1+3]\neq [-2+3]$ in $\widetilde{H}_0(\Delta_{123}^{(0)}; k)$. 
If we choose $[-1+3]$ as the preimage of $e$ in $\widetilde{H}_0(\Delta_{123}^{(0)}; k)$, 
then by Definition \ref{D:RLM} we get
\[
\mathcal{R}(L_M): 0\to k^2 
\xrightarrow{\begin{pmatrix} -1& -1\\ 1 &0 \\ 0 &1 \end{pmatrix}} 
k^3 \xrightarrow{\begin{pmatrix} 1 & 1 & 1 \end{pmatrix}} k, 
\]
and then we have
\[
\mathcal{G}(L_M): 0\to S(-abc)\oplus S(-ab^2c) 
\xrightarrow{\begin{pmatrix} -c &-bc\\ b & 0 \\ 0 & a \end{pmatrix} } 
S(-ab)\oplus S(-ac) \oplus S(-b^2c) \ \ \ \ \ \ \ \ 
\]
\[
\ \ \ \ \ \ \ \ \ \ \ \ \ \ \ \ \ \ \ \ \ \ \ \ \ \ \ \ \ \ \ \ \ \ \ \ \ \ \ \ \ \ \ \ \ \ \ \ \ 
\ \ \ \ \ \ \ \ \ \ \ \ \ 
\xrightarrow{\begin{pmatrix} ab & ac & b^2c \end{pmatrix} } S,
\]
which is a minimal free resolution of $S/M$. 
If we choose $[-2+3]$ as the preimage of $e$ in $\widetilde{H}_0(\Delta_{123}^{(0)}; k)$, 
then by Definition \ref{D:RLM} we get
\[
\mathcal{R}(L_M): 0\to k^2 
\xrightarrow{\begin{pmatrix} -1& 0\\ 1 &-1\\ 0 &1 \end{pmatrix}} 
k^3 \xrightarrow{\begin{pmatrix} 1 & 1 & 1 \end{pmatrix}} k, 
\]
and then we have
\[
\mathcal{G}(L_M): 0\to S(-abc)\oplus S(-ab^2c) 
\xrightarrow{\begin{pmatrix} -c &0\\ b & -b^2 \\ 0 & a \end{pmatrix} } 
S(-ab)\oplus S(-ac) \oplus S(-b^2c) \ \ \ \ \ \ \ \ 
\]
\[
\ \ \ \ \ \ \ \ \ \ \ \ \ \ \ \ \ \ \ \ \ \ \ \ \ \ \ \ \ \ \ \ \ \ \ \ \ \ \ \ \ \ \ \ \ \ \ \ \ 
\ \ \ \ \ \ \ \ \ \ \ \ \ 
\xrightarrow{\begin{pmatrix} ab & ac & b^2c \end{pmatrix} } S,
\]
which is also a minimal free resolution of $S/M$. 
\end{example}

From Example \ref{E:RLM} one might wonder if $\mathcal{G}(L_M)$ in 
Definition \ref{D:RLM} is always a minimal free resolution of $S/M$. 
As illustrated by Example \ref{E:notRLM} this is not true. 

\begin{example}\label{E:notRLM}
Let $M=(a^2b, ac, ad, bcd)$ and its lcm-lattice $L_M$ be as in Example 
\ref{E:cbasis}. By looking at the sublattice of $L_M$ with top element 
$234$, we see that $M$ is not Betti-linear. As in Example \ref{E:cbasis} 
we have that $\widetilde{H}_0(\Delta_{12}; k)\cong k$ with a basis $[-1+2]$, 
$\widetilde{H}_0(\Delta_{13}; k)\cong k$ with a basis $[-1+3]$, 
$\widetilde{H}_0(\Delta_{23}; k)\cong k$ with a basis $[-2+3]$, 
$\widetilde{H}_0(\Delta_{234}; k)\cong k$ with a basis 
$[-2+4]=[-3+4]$, and $\widetilde{H}_1(\Delta_{1234}; k)\cong k$ with 
a basis $[12-13+23]$. Note that $\Delta_{1234}^{(1)}=\Delta_{1234}
=\langle 12, 13, 234 \rangle$, so that $\sigma_{1234}^{(1)}$ is the 
identity map, and then under the above bases the $k$-linear map
\[
\psi_3: \widetilde{H}_1(\Delta_{1234}; k)\to \widetilde{H}_0(\Delta_{12}; k) 
\oplus \widetilde{H}_0(\Delta_{13}; k) \oplus \widetilde{H}_0(\Delta_{23}; k) 
\oplus \widetilde{H}_0(\Delta_{234}; k)
\]
is uniquely defined. By Definition \ref{D:RLM} we get 
$\psi_3: k\xrightarrow{\left(\begin{smallmatrix} 1\\ -1\\ 0 \\ 0 \end{smallmatrix}\right)} k^4$. 
Since in the matrix of $\psi_2$ under the given basis, the first two columns 
are $\left( \begin{smallmatrix} -1 \\ 1\\ 0 \\0 \end{smallmatrix} \right)$ and 
 $\left( \begin{smallmatrix} -1 \\ 0\\ 1 \\0 \end{smallmatrix} \right)$, and we have 
that $\left(\begin{smallmatrix} -1 & -1 \\ 1 & 0\\  0& 1 \\0 &0 \end{smallmatrix} \right)
\left(\begin{smallmatrix} 1  \\-1 \end{smallmatrix} \right) \neq 0$, it follows that 
$\psi_2 \circ \psi_3\neq 0$, so that $\mathcal{R}(L_M)$ is not a complex. 
So  $\mathcal{G}(L_M)$ is not a minimal free resolution of $S/M$. 

From the minimal free resolution of $S/M$ given in Example \ref{E:cbasis} it is easy to 
see that if $\mathbf{F}$ is a minimal free resolution of $S/M$, then in the frame of 
$\mathbf{F}$, the corresponding component map from $\widetilde{H}_1(\Delta_{1234}; k)$ 
to $\widetilde{H}_0(\Delta_{234}; k)$ must be zero, while the component map 
from $\widetilde{H}_1(\Delta_{1234}; k)$ to $\widetilde{H}_0(\Delta_{23}; k)$ 
can not be zero. This is illustrated by the fact that $[d(23)]=0$ in $\widetilde{H}_0(\Delta_{234}; k)$, 
while $[d(23)]\neq 0$ in $\widetilde{H}_0(\Delta_{23}; k)$. 
However, by Definition \ref{D:RLM}, since $23$ is not a maximal element in 
$\mathcal{B}_1(1234)=\{12, 13, 23, 234 \}$, the component map 
$\psi_{3,1234}^{23}: \widetilde{H}_1(\Delta_{1234}; k) \to \widetilde{H}_0(\Delta_{23}; k)$ 
is always defined to be the zero map. 
So $\mathcal{G}(L_M)$ can never be a minimal free resolution 
of $S/M$. In situations like this, we can never use Definition \ref{D:RLM} to 
obtain a minimal free resolution of $S/M$. 
\end{example}

Later we will prove that $\mathcal{G}(L_M)$ is a minimal free resolution for 
some special classes of monomial ideals. But before that we will show that 
given any monomial ideal $M$, $\mathcal{G}(L_M)$ provides an approximation 
formula for minimal free resolutions of $S/M$. To make the word 
``approximation" precise, we have the following definition. 

\begin{definition}\label{D:approximation}
Let $(\mathbf{F}, \partial)$ be a minimal free resolution of $S/M$ with a 
multigraded basis $B$. Let $e\in B$ be a basis element in $F_i$. Let 
$e_1, \ldots, e_p\in B$ be all the basis elements in $F_{i-1}$. Let 
$\partial(e)=\lambda_1e_1+\cdots+\lambda_pe_p$ with 
$\lambda_1, \ldots, \lambda_p \in S$. 

We define $\partial^{\approx}(e)=\mu_1e_1+\cdots+\mu_pe_p$ as follows:
for any $1\leq j \leq p$, if there exists $1\leq l \leq p$ such that 
$\mbox{mdeg}(e_j)<\mbox{mdeg}(e_l) \in L_M$, then we set $\mu_j=0$; 
otherwise, we set $\mu_j=\lambda_j$. 

After defining $\partial^{\approx}(e)$ for all $e\in B$, we obtain 
$(\mathbf{F}, \partial^{\approx})$, which is a sequence of multigraded 
$S$-modules and multigraded homomorphisms. We call
$(\mathbf{F}, \partial^{\approx})$ the \emph{maximal approximation} of 
$(\mathbf{F}, \partial)$. $(\mathbf{F}, \partial^{\approx})$ can also be 
denoted by $\mathbf{F}^{\approx}$. Note that although 
$\mathbf{F}^{\approx}$ is often not a complex, we can still 
dehomogenize $\mathbf{F}^{\approx}$ to get a sequence of $k$-vector 
spaces and $k$-linear maps, which is also called the \emph{frame} of 
$\mathbf{F}^{\approx}$. 
\end{definition}

\begin{theorem}\label{T:approximation}
Let $M$ be a monomial ideal in $S$ minimally generated by monomials $m_1, \ldots, m_r$. 
\begin{itemize}
\item[(1)] Let $(\mathbf{F}, \partial)$ be a minimal free resolution of $S/M$ such 
that $F_1 \xrightarrow{\partial_1} F_0$ is given by 
$S(-m_1)\oplus \cdots \oplus S(-m_r) \xrightarrow
{(\begin{smallmatrix} m_1 & \cdots & m_r \end{smallmatrix})} S$, 
then $\mathbf{F}^{\approx}$ can be obtained by Definition \ref{D:RLM} 
as $\mathcal{G}(L_M)$. 
\item[(2)] Let $\mathcal{G}(L_M)$ be a sequence of multigraded $S$-modules 
and multigraded homomorphisms obtained by Definition \ref{D:RLM}, 
then there exists a minimal free resolution $(\mathbf{F}, \partial)$ of 
$S/M$ such that $\mathcal{G}(L_M)=\mathbf{F}^{\approx}$. 
\end{itemize}
\end{theorem}

\begin{proof}
(1) Let $B$ be a Taylor basis of $\mathbf{F}$. For any $m\neq 1\in B_M$ with 
$\widetilde{H}_{i-2}(\Delta_m; k)\neq 0$, let 
$e_{m,i-2}^1, \ldots, e_{m, i-2}^{t_{m, i-2}}\in B$ be all the Taylor basis 
elements in homological degree $i$ with multidegree $m$.  
Then by Theorem \ref{T:basis}, $[d(e_{m, i-2}^1)], \ldots, [d(e_{m,i-2}^{t_{m, i-2}})]$ 
is a basis of $\widetilde{H}_{i-2}(\Delta_m; k)$. We use this basis for 
$\widetilde{H}_{i-2}(\Delta_m; k)$ in Definition \ref{D:RLM}. If $i\geq 2$, then 
it is easy to see that $d(e_{m, i-2}^j)$ is a cycle in $\Delta_m^{(i-2)}$, so that 
in Definition \ref{D:RLM} we can always choose $[d(e_{m, i-2}^j)]$ as the 
preimage of $[d(e_{m, i-2}^j)]$ in $\widetilde{H}_{i-2}(\Delta_m^{(i-2)}; k)$. 

Let $m\neq 1 \in B_M$ with $\widetilde{H}_{i-2}(\Delta_m; k)\neq 0$ and $i\geq 2$. 
Let $\Delta_m^{(i-2)}=\langle A_{\gamma_1}, \ldots, A_{\gamma_t} \rangle$. 
Let $e_m \in B$ be a Taylor basis element in homological degree $i$ with 
multidegree $m$, then $[d(e_m)]$ is a basis element of $\widetilde{H}_{i-2}(\Delta_m; k)$. 
Let
\[
d(e_m)=\lambda_1f_1+\cdots+\lambda_pf_p+\mu_1g_1+\cdots+\mu_qg_q,
\]
where $\lambda_1, \ldots, \lambda_p, \mu_1, \ldots, \mu_q$ are some nonzero 
scalar in $k$ and $f_1, \ldots, f_p, g_1, \ldots, g_q\in B$ are some Taylor basis 
elements in homological degree $i-1$ such that for any $1\leq j \leq p$ there 
exists $1\leq l \leq t$ with $\mbox{mdeg}(f_j)=\gamma_l$ and for any 
$1\leq a \leq q$ there exists $1 \leq b \leq t$ with $\mbox{mdeg}(g_a)<\gamma_b$. 
Assume that there exist $1\leq l_1 \neq l_2 \leq t$ such that  
$\mbox{supp}(f_j)\subseteq A_{\gamma_{l_1}}$ 
and $\mbox{supp}(f_j) \subseteq A_{\gamma_{l_2}}$. Then we have that 
$\mbox{supp}(f_j)\subseteq A_{\gamma_{l_1}}\cap A_{\gamma_{l_2}}
=A_{\gamma_{l_1}\wedge \gamma_{l_2}}$, so that 
$A_{\gamma_l} =\overline{\mbox{supp}(f_j)}\subseteq A_{\gamma_{l_1}\wedge \gamma_{l_2}}$, 
and then $\gamma_l\leq \gamma_{l_1}\wedge \gamma_{l_2}$, which contradicts to the 
assumption that $\gamma_1, \ldots, \gamma_t$ are maximal elements in $\mathcal{B}_{i-2}(m)$. 
Hence, for any $1\leq j \leq p$ there exists a unique $1\leq l\leq t$ such that 
$\mbox{supp}(f_j)\subseteq A_{\gamma_l}$, $\mbox{mdeg}(f_j)=\gamma_l$ and 
$[d(f_j)]$ is a basis element of $\widetilde{H}_{i-3}(\Delta_{\gamma_l}; k)$. 
Also, let $\widetilde{m}=\mbox{mdeg}(g_a)<\gamma_b$, then we have that 
$\mbox{supp}(g_a)\subseteq A_{\widetilde{m}}$ and $A_{\widetilde{m}}$ is 
a face of $\Delta_{\gamma_b}$, so that $[d(g_a)]=0$ in $\widetilde{H}_{i-3}(\Delta_{\gamma_b}; k)$. 
Thus, in Definition \ref{D:RLM} we have that 
\begin{align*}
\psi_i([d(e_m)])&=\psi_i([\lambda_1f_1+\cdots+\lambda_pf_p+\mu_1g_1+\cdots+\mu_qg_q])\\
                         &=\sum_{l=1}^t\psi_{i, m}^{\gamma_l}([\lambda_1f_1+\cdots+\lambda_pf_p+\mu_1g_1+\cdots+\mu_qg_q])\\
                         &=[d(\lambda_1f_1)]+\cdots+[d(\lambda_pf_p)]\\
                         &=\lambda_1[d(f_1)]+\cdots+\lambda_p[d(f_p)].
\end{align*}
Let $d^{\approx}$ be the map in the frame of $\mathbf{F}^{\approx}$, then 
it is easy to see that 
\[
d^{\approx}(e_m)=\lambda_1f_1+\cdots+\lambda_pf_p.
\]

So $\mathcal{R}(L_M)$ obtained in Definition \ref{D:RLM} is equal to the 
frame of $\mathbf{F}^{\approx}$, and then it is easy to see that 
$\mathcal{G}(L_M)$ is equal to $\mathbf{F}^{\approx}$. 

(2) By the proof of part (1) we see that it suffices to prove the following claim: 
if for any $m\neq 1 \in B_M$ with $\widetilde{H}_{i-2}(\Delta_m; k)\neq 0$ and 
$i\geq 2$, we have fixed a basis $[c_{m, i-2}^1], \ldots, [c_{m, i-2}^{t_{m, i-2}}]$ 
for $\widetilde{H}_{i-2}(\Delta_m; k)$, and for any $1\leq j \leq t_{m, i-2}$ we have 
fixed $[\widetilde{c}_{m, i-2}^{t_{m, i-2}}]\in \widetilde{H}_{i-2}(\Delta_m^{(i-2)}; k)$ 
with $\sigma_m^{(i-2)}([\widetilde{c}_{m, i-2}^j])=[c_{m, i-2}^j] \in \widetilde{H}_{i-2}(\Delta_m; k)$, 
then there exists a Taylor basis $B$ for a minimal free resolution of $S/M$ such 
that $B\supseteq \{\emptyset, \{1\}, \ldots, \{r\} \}$, and let 
$e_{m,i-2}^1, \ldots, e_{m, i-2}^{t_{m, i-2}}\in B$ be the Taylor basis elements 
in homological degree $i$ with multidegree $m$, then for any $1\leq j \leq t_{m, i-2}$ 
we have that $[d(e_{m, i-2}^j)]=[\widetilde{c}_{m, i-2}^j] \in \widetilde{H}_{i-2}(\Delta_m^{(i-2)}; k)$. 

Next, following Construction \ref{C:VL}, we will use strong induction in $L_M$ 
to obtain a Taylor basis $B$ satisfying the claim. 

Base case: Let $\Gamma_1=\{\emptyset \}$, and for any $1\leq j\leq r$, 
let $\Gamma_{m_j}=\{\{j\}\}$. 

Inductive step: Let $m\in L_M$ with $\mbox{rk}(m)\geq 2$. Assume that for any 
$\widetilde{m}<m \in L_M$ we have found $\Gamma_{\widetilde{m}}$ and the 
claim holds for the Taylor basis elements in $\Gamma_{\widetilde{m}}$ whenever 
$\widetilde{m}\in B_M$ and $\mbox{rk}(\widetilde{m})\geq 2$. If $m\notin B_M$ 
then $\Gamma_m=\emptyset$ and there is nothing to prove. Next we assume 
that $\widetilde{H}_{i-2}(\Delta_m; k)\neq 0$ for some $i\geq 2$. 
Let $\Delta_m^{(i-2)}=\langle A_{\gamma_1}, \ldots, A_{\gamma_t} \rangle$ and 
$\Pi=\{\widetilde{m} \in L_M| \widetilde{m} \leq \gamma_j  \mbox{ for some} \ \gamma_j\}$. 
Let $(\mathbb{U}_m^{(i-2)}; d)$ be the complex obtained by applying the boundary 
map $d$ to the Taylor basis elements in $\bigcup\limits_{\widetilde{m} \in \Pi} \Gamma_{\widetilde{m}}$.  
Then by Corollary \ref{C:main1c2} we have that 
\[
\widetilde{C}(\Delta_m^{(i-2)}; k)[-1]\cong (\mathbb{U}_m^{(i-2)}, d) 
\bigoplus (\bigoplus \limits_{1\neq \widetilde{m} \in \Pi} 
(\mathcal{E}_{\widetilde{m}}, d)),
\]
where $(\mathcal{E}_{\widetilde{m}}, d)$ are trivial complexes. Hence, 
given any $[\widetilde{c}_{m, i-2}^j]\in \widetilde{H}_{i-2}(\Delta_m^{(i-2)}; k)$, 
there exists a cycle $z_{m, i-2}^j$ in $\mathbb{U}_m^{(i-2)}$ such that 
$[z_{m, i-2}^j]=[\widetilde{c}_{m,i-2}^j]\in \widetilde{H}_{i-2}(\Delta_m^{(i-2)}; k)$. 
Note that $\mathbb{U}_m^{(i-2)}$ is a subcomplex of $\mathbb{U}_m$, so 
that $z_{m, i-2}^j$ is a cycle in $\mathbb{U}_m$. Since
\[
\sigma_m^{(i-2)}([z_{m, i-2}^j])=\sigma_{m}^{(i-2)}([\widetilde{c}_{m, i-2}^j])
=[c_{m, i-2}^j] \in \widetilde{H}_{i-2}(\Delta_m; k), 
\]
it follows that $[z_{m, i-2}^1], \ldots, [z_{m, i-2}^{t_{m, i-2}}]$ is a basis 
of $ \widetilde{H}_{i-2}(\Delta_m; k)$. Thus, by Corollary \ref{C:main1c1}, 
we have that $[z_{m, i-2}^1], \ldots, [z_{m, i-2}^{t_{m, i-2}}]$ is a basis 
of $ H_{i-1}(\mathbb{U}_m, d)$. So by Construction \ref{C:VL} we can find 
Taylor basis elements $e_{m, i-2}^1, \ldots, e_{m, i-2}^{t_{m, i-2}}$ at $m$ 
in homological degree $i$ such that for any $1\leq j\leq t_{m, i-2}$, 
$d(e_{m, i-2}^j)=z_{m, i-2}^j$, and then 
\[
[d(e_{m, i-2}^j)]=[z_{m, i-2}^j]=[\widetilde{c}_{m, i-2}^j] \in \widetilde{H}_{i-2}(\Delta_m^{(i-2)}; k).
\]
Therefore, we can find $\Gamma_m$ and the Taylor basis elements in $\Gamma_m$ 
satisfy the claim. 
\end{proof}

\begin{remark}\label{R:approximation}
In the proof of Theorem \ref{T:approximation} (2) we actually didn't use the 
inductive assumption that the Taylor basis elements in $\Gamma_{\widetilde{m}}$ 
satisfy the claim. Indeed, if $\Gamma_{\widetilde{m}}$ is given for all 
$\widetilde{m}<m \in L_M$, then we can obtain $e_{m, i-2}^j$ such that 
$d(e_{m, i-2}^j)$ equals the given $[\widetilde{c}_{m, i-2}^j]$ in 
$\widetilde{H}_{i-2}(\Delta_m^{(i-2)}; k)$. 

By Example \ref{E:approximation}, we see that given two different minimal free resolutions 
$\mathbf{F}$ and $\mathbf{G}$ of $S/M$, it may happen that 
$\mathbf{F}^{\approx}=\mathbf{G}^{\approx}$. Hence, in 
Theorem \ref{T:approximation} (2), the minimal free resolution 
$\mathbf{F}$ may not be unique. 

Although it seems impossible to get a formula for minimal free resolutions 
of \emph{all} monomial ideals, by Theorem \ref{T:approximation} we do 
have an approximation formula, namely, $\mathcal{G}(L_M)$,  for minimal 
free resolutions of \emph{all} monomial ideals, which is somewhat unexpected. 

By Construction \ref{C:VL} or by Theorem \ref{T:cbasis}, we can inductively 
construct a minimal free resolution of $S/M$ step by step. In contrast, 
Theorem \ref{T:approximation} tells us that if we want to understand the differential 
map at homological degree $i$  in  minimal free resolutions of $S/M$, 
we can use Definition \ref{D:RLM} to get an approximation of the 
differential map and we do not need to care about the differential maps at 
other homological degrees. 
\end{remark}

\begin{example}\label{E:approximation}
Let $M$ be a monomial ideal whose lcm-lattice is shown in Figure 5. 
Then we have that $\Delta_{12345}=\langle 12, 13, 2345 \rangle$ and 
$\Delta_{2345}=\langle 23, 24, 34, 5 \rangle$, which implies that 
$\widetilde{H}_1(\Delta_{12345}; k)\cong k$ with a basis $[12-13+23]$, 
$\widetilde{H}_0(\Delta_{2345}; k)\cong k$ with a basis $[-2+5]$, 
$\widetilde{H}_1(\Delta_{2345}; k)\cong k$ with a basis $[23-24+34]$. 
Note that $\widetilde{H}_0(\Delta_{12}; k)\cong k$ with a basis $[-1+2]$, 
and we have similar basis  for $\widetilde{H}_0(\Delta_{13}; k)$, 
$\widetilde{H}_0(\Delta_{23}; k)$, $\widetilde{H}_0(\Delta_{24}; k)$ 
and $\widetilde{H}_0(\Delta_{34}; k)$. 
Hence, we have that 
$\Delta_{12345}^{(1)}=\Delta_{12345}$ and $\sigma_{12345}^{(1)}$ 
is the identity map; 
$\Delta_{2345}^{(0)}=\langle 2, 3, 4, 5 \rangle \subsetneq \Delta_{2345}$ 
and $\sigma_{2345}^{(0)}$ is not an isomorphism; 
$\Delta_{2345}^{(1)}=\langle 23, 24, 34 \rangle \subsetneq \Delta_{2345}$ 
and $\sigma_{2345}^{(1)}$ is an isomorphism. 
It is interesting to notice that the vertex set of $\Delta_{2345}^{(1)}$ is 
$\{2, 3, 4\}$ instead of $\{2, 3, 4, 5\}$. Thus, in general, if 
$\Delta_m^{(i)}=\langle A_{\gamma_1}, \ldots, A_{\gamma_t} \rangle$, 
then $\gamma_1, \ldots, \gamma_t$ may not form a crosscut of the 
atomic lattice $L_M(\leq m)=\{\widetilde{m}\in L_M| \widetilde{m}<m\}$. 
Let 
\begin{align*}
\mathcal{R}_3&=\widetilde{H}_1(\Delta_{2345}; k)\oplus \widetilde{H}_1(\Delta_{12345}; k) \cong k^2\\
\mathcal{R}_2&=\widetilde{H}_0(\Delta_{12}; k)\oplus \cdots \oplus 
\widetilde{H}_0(\Delta_{34}; k) \oplus \widetilde{H}_0(\Delta_{2345}; k) \cong k^6\\
\mathcal{R}_1&=\widetilde{H}_{-1}(\Delta_{1}; k)\oplus \cdots \oplus \widetilde{H}_{-1}(\Delta_{5}; k)\cong k^5
\end{align*}
If we choose $[-2+5]$ to be the preimage of $[-2+5]$ in 
$\widetilde{H}_0(\Delta_{2345}^{(0)}; k)$, then by Definition \ref{D:RLM} we get
\[
\mathcal{R}(L_M): 0\to k^2 \xrightarrow{\left( \begin{smallmatrix} 0 & 1 \\ 0 & -1 \\ 1& 0 \\
-1 & 0\\ 1 & 0 \\ 0 & 0 \end{smallmatrix} \right)} k^6 \xrightarrow{\left( \begin{smallmatrix} 
-1 &-1 & 0 & 0 & 0 & 0 \\ 1 & 0 & -1& -1 & 0 & -1\\ 0 &1 & 1 & 0 & -1 & 0 \\ 0 & 0 & 0 & 1 & 1 & 0 \\ 0 & 0 & 0 & 0 & 0 & 1
\end{smallmatrix} \right)} k^5 \xrightarrow{(\begin{smallmatrix} 1& 1 & 1 & 1 & 1 \end{smallmatrix})} k. 
\]
Similar to Example \ref{E:cbasis} we take $\Gamma_{2345}=\{ 25, 234\}$. 
If we take $\Gamma_{12345}=\{123\}$ then we get a minimal free resolution 
$\mathbf{F}$ of $S/M$ whose frame between homological degree $2$ and 
homological degree $3$ is given by 
\[
k234\oplus k123 \xrightarrow{\left( \begin{smallmatrix} 0 & 1 \\0 & -1 \\ 1 & 1 \\ -1 & 0 \\ 1 & 0 \\ 0 & 0 
\end{smallmatrix}\right) } k12\oplus k13 \oplus k23 \oplus k24 \oplus k34 \oplus k25. 
\]
If we take $\Gamma_{12345}=\{123-234\}$ then we get a minimal free resolution 
$\mathbf{G}$ of $S/M$ whose frame between homological degree $2$ and 
homological degree $3$ is given by 
\[
k234\oplus k(123-234) \xrightarrow{\left( \begin{smallmatrix} 0 & 1 \\0 & -1 \\ 1 & 0 \\ -1 & 1 \\ 1 & -1 \\ 0 & 0 
\end{smallmatrix}\right) } k12\oplus k13 \oplus k23 \oplus k24 \oplus k34 \oplus k25. 
\]
It is easy to see that $\mathbf{F}^{\approx}=\mathbf{G}^{\approx}$, 
and the frame of $\mathbf{F}^{\approx}$ equals $\mathcal{R}(L_M)$,  but $\mathbf{F}\neq \mathbf{G}$.

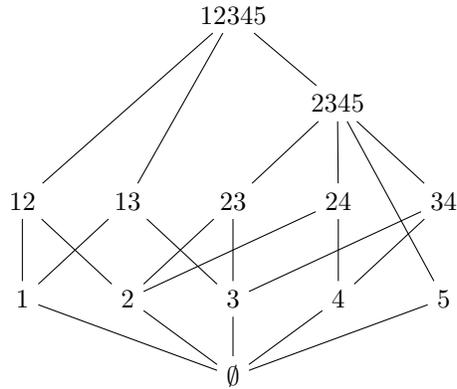
\begin{figure}
\begin{tikzpicture}[node distance=1.4cm]
\title{The approximation}
\node(12345)                                                                    {$12345$};
\node(2345)    [below right=0.7cm and 0.35cm of 12345]  {$2345$};
\node(23)      [below=2.0cm of 12345]                                          {$23$};
\node(24)      [right of =23]                                              {$24$};
\node(34)      [right of =24]                                              {$34$};
\node(13)      [left of =23]                                              {$13$};
\node(12)      [left of =13]                                              {$12$};
\node(3)        [below=3.3cm of 12345]                                 {$3$};
\node(2)        [left of =3]	                                    {$2$};
\node(1)        [left of =2]	                                    {$1$};
\node(4)        [right of =3]	                                    {$4$};
\node(5)        [right of =4]	                                    {$5$};
\node(0)        [below=4.3cm of 12345]                                 {$\emptyset$};

\draw(12345) --(2345);
\draw(12345) --(12);
\draw(12345) --(13);
\draw(2345) --(23);
\draw(2345) --(24);
\draw(2345) --(34);
\draw(2345) --(5);
\draw(12) --(1);
\draw(12) --(2);
\draw(13) --(1);
\draw(13) --(3);
\draw(23) --(2);
\draw(23) --(3);
\draw(24) --(2);
\draw(24) --(4);
\draw(34) --(3);
\draw(34) --(4);
\draw(1) --(0);
\draw(2) --(0);
\draw(3) --(0);
\draw(4) --(0);
\draw(5) --(0);

\end{tikzpicture}
\caption{The lcm-lattice $L_M$ of Example \ref{E:approximation}}
\end{figure}

\end{example}

The next example shows that even if $\mathcal{G}(L_M)$ is a 
complex, it may not be a minimal free resolution of $S/M$. 

\begin{example}\label{E:approximation2}
Let $M$ be a monomial ideal whose lcm-lattice is shown in Figure 6. 
Similar to Example \ref{E:cbasis} we can take 
$\emptyset, 1, 2, 3, 4, 5, 6, 12, 13, 23, 14, 15, 26, 123$ as a 
Taylor basis of a minimal free resolution $\mathbf{F}$ of $S/M$. 
It is easy to see that the frame of $\mathbf{F}^{\approx}$ is 
given by
\[
0\to k \xrightarrow{0} k^6 \xrightarrow{\left( \begin{smallmatrix} 
-1 &-1 & 0 & -1 & -1 & 0 \\ 1 & 0 & -1& 0 & 0 & -1\\ 0 &1 & 1 & 0 & 0 & 0 \\ 0 & 0 & 0 & 1 & 0 & 0 \\ 0 & 0 & 0 & 0 & 1 & 0
\\ 0 & 0 & 0 & 0 & 0 &1 
\end{smallmatrix} \right)} k^6 \xrightarrow{(\begin{smallmatrix} 1& 1 & 1 & 1 & 1 & 1\end{smallmatrix})} k, 
\]
which is a complex but not exact, so that $\mathbf{F}^{\approx}$ is a complex, 
but $\mathbf{F}^{\approx}\neq \mathbf{F}$. 
So by Theorem \ref{T:approximation} we see that even if $\mathcal{G}(L_M)=\mathbf{F}^{\approx}$ 
is a complex, it may not be a minimal free resolution of $S/M$.

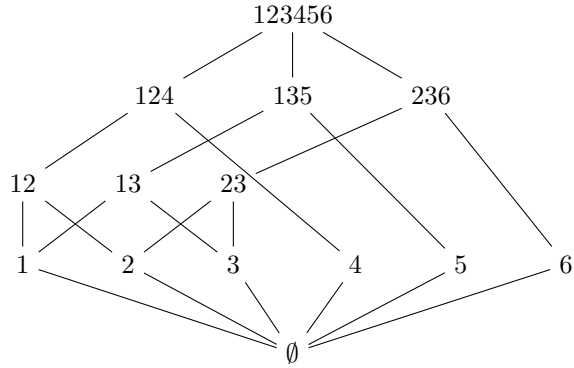
\begin{figure}
\begin{tikzpicture}[node distance=1.4cm]
\title{The approximation2}
\node(123456)                                                                    {$123456$};
\node(135)    [below=0.6cm   of 123456]  {$135$};
\node(124)       [below left=0.6cm and 0.8cm of 123456]                          {$124$};
\node(236)      [below right=0.6cm and 0.8cm of 123456]                        {$236$};
\node(13)      [below left=0.7cm and 1.5cm of 135]                                              {$13$};
\node(12)      [left of =13]                                              {$12$};
\node(23)      [right of =13]                                              {$23$};
\node(2)        [below=0.6cm of 13]	                                    {$2$};
\node(3)        [right of =2]                                 {$3$};
\node(4)        [right=1.2cm of 3]	                                    {$4$};
\node(5)        [right of =4]	                                    {$5$};
\node(6)        [right of =5]                                     {$6$};
\node(1)        [left of =2]	                                    {$1$};
\node(0)        [below=4.0cm of 12345]                                 {$\emptyset$};

\draw(123456) --(124);
\draw(123456) --(135);
\draw(123456) --(236);
\draw(124) --(12);
\draw(124) --(4);
\draw(135) --(13);
\draw(135) --(5);
\draw(236) --(23);
\draw(236) --(6);
\draw(12) --(1);
\draw(12) --(2);
\draw(13) --(1);
\draw(13) --(3);
\draw(23) --(2);
\draw(23) --(3);
\draw(1) --(0);
\draw(2) --(0);
\draw(3) --(0);
\draw(4) --(0);
\draw(5) --(0);
\draw(6) --(0);

\end{tikzpicture}
\caption{The lcm-lattice $L_M$ of Example \ref{E:approximation2}}
\end{figure}
\end{example}

Similar to the definition of  Betti-linear monomial ideals,  we have the following definition of 
homology-linear monomial ideals.

\begin{definition}\label{D:HL}
Let $M$ be a monomial ideal in $S$. If there exists a minimal free resolution 
$(\mathbf{F}, \partial)$ of $S/M$ such that $\mathbf{F}^{\approx}=\mathbf{F}$, 
then $M$ is called a \emph{homology-linear monomial ideal}. 
\end{definition}

\begin{remark}\label{R:HL}
By Theorem \ref{T:approximation} it is easy to see that $M$ is homology-linear 
if and only if there exists a $\mathcal{G}(L_M)$ such that $\mathcal{G}(L_M)$ 
is a minimal free resolution of $S/M$. 
Let $M$ be the monomial ideal as in Example \ref{E:cbasis}, then by 
Example \ref{E:notRLM} we see that although $M$ has a simplicial resolution, 
$M$ is not homology-linear. 
When $M$ is homology-linear, we will see in Example \ref{E:big2} that not all 
$\mathcal{G}(L_M)$ obtained by Definition \ref{D:RLM} are minimal free 
resolutions of $S/M$. In contrast, if $M$ is a Betti-linear monomial ideal, by 
Remark \ref{R:Betti} we see that every poset construction induces a minimal 
free resolution $\mathcal{F}(L_M)$ of $S/M$. 

In Proposition \ref{P:BLHL}, we will show that Betti-linear monomial ideals are homology-linear, 
so that the concept of homology-linearity is a generalization of Betti-linearity, 
and Definition \ref{D:RLM} can be viewed as a genealization of Definition \ref{D:posetres}. 
So the theory in this subsection is not just parallel to the theory in Subsection 4.1; it is 
also a generalization. 
\end{remark}

\begin{proposition}\label{P:BLHL}
Let $M$ be a monomial ideal. If $M$ is Betti-linear then $M$ is homology-linear. 
\end{proposition}

\begin{proof}
Since $M$ is Betti-linear, it follows that there exists a minimal free resolution 
$(\mathbf{F}, \partial)$ of $S/M$ such that $\mathbf{F}$ satisfies the 
definition of Betti-linearity. By Definition \ref{D:ideals} (4), 
it is easy to see that $\mathbf{F}^{\approx}=\mathbf{F}$, so that $M$ 
is homology-linear. 
\end{proof}

By Theorem \ref{T:approximation} it is easy to see that every $\mathcal{G}(L_M)$ 
is a minimal free resolution of $S/M$ if and only if for every minimal free resolution 
$(\mathbf{F}, \partial)$ of $S/M$, we have that $\mathbf{F}^{\approx}=\mathbf{F}$. 
Thus, we have the following definition. 

\begin{definition}\label{D:NHM}
Let $M$ be a monomial ideal. 
\begin{itemize}
 \item[(1)] $M$ is called a \emph{strongly homology-linear monomial ideal} if 
 for every minimal free resolution $(\mathbf{F}, \partial)$ of $S/M$, we have 
 that $\mathbf{F}^{\approx}=\mathbf{F}$.
 \item[(2)] $M$ is called a \emph{nearly homologically monotonic monomial 
 ideal} if the following condition is satisfied: for any $1< \widetilde{m} <m \in L_M$, 
 if $\widetilde{H}_i(\Delta_{\widetilde{m}}; k)\neq 0$ and 
 $\widetilde{H}_i(\Delta_m; k) \neq 0$, then there does not exist $\widehat{m} \in L_M$ 
 such that $\widetilde{m}<m<\widehat{m}$ and $\widetilde{H}_{i+1}(\Delta_{\widehat{m}}; k)\neq 0$. 
\end{itemize}
\end{definition}

\begin{remark}\label{R:NHM}
The class  of strongly homology-linear monomial ideals is exactly the set of 
monomial ideals for which every $\mathcal{G}(L_M)$ obtained by Definition 
\ref{D:RLM} is a minimal free resolution. 
Every strongly homology-linear monomial ideal is homology-linear. 
In Example \ref{E:big2} we will give a monomial ideal which is 
homology-linear but not strongly homology-linear. 

Note that to check if a monomial ideal $M$ is Betti-linear or homology-linear, we need 
to look at minimal free resolutions of $S/M$; in contrast, to check if $M$ is rigid or (nearly)
homologically monotonic, we only need to calculate $\widetilde{H}(\Delta_m; k)$ for all $m\in L_M$. 

Obviously, every homologically monotonic monomial ideal is nearly homologically monotonic. 
Let $M$ be the monomial ideal as in Example \ref{E:notBetti}, where we have proved that 
$M$ is not Betti-linear and then not homologically monotonic, but it is easy to see that 
$M$ is nearly homologically monotonic. 

By Definition \ref{D:approximation} it is easy to see that every nearly homologically 
monotonic monomial ideal is strongly homology-linear. In Example \ref{E:SHL} we give 
a monomial ideal which is strongly homology-linear but not nearly homologically monotonic 
or Betti-linear. 
\end{remark}

\begin{example}\label{E:big2}
This example is related to Example \ref{E:strict2} where the monomial ideal is both 
Betti-linear and nearly homologically monotonic. 

Let $M$ be a monomial ideal whose lcm-lattice is shown in Figure 7. 
Then we have that $\Delta_{123456}=\langle 1234, 346, 356, 456 \rangle$ and 
$\Delta_{12345}=\langle 123, 34, 35, 45 \rangle$, which implies that 
$\widetilde{H}_2(\Delta_{123456}; k)\cong k$ with a basis $[345-346+356-456]$, 
and $\widetilde{H}_1(\Delta_{12345}; k)\cong k$ with a basis $[34-35+45]$. 
Note that $\widetilde{H}_1(\Delta_{123}; k)\cong k$ with a basis $[12-13+23]$, 
so that $M$ is not nearly homologically monotonic.  

Similar to Example \ref{E:cbasis} and  by Theorem \ref{T:cbasis} we take 
$\Gamma_{12345}=\{ 345\}$,  $\Gamma_{123456}=\{ 3456 \}$ and 
$\Gamma_{A_m}=\{A_m\}$ for every Scarf multidegree $m\in L_M$. 
Hence, $\emptyset$, $1, \ldots, 6$, $12, 13, 23$, $34, \ldots, 56$, $123, 
345, 346, 356, 456, 3456$ is a Taylor basis of a minimal free resolution 
$\mathbf{F}$ of $S/M$. Obviously, $\mathbf{F}$ satisfies the definition 
of Betti-linearity, so that $M$ is Betti-linear and then homology-linear. 

However, if we take $\Gamma_{12345}=\{345+123\}$, then by Theorem \ref{T:cbasis},
$\emptyset$, $1, \ldots, 6$, $12, 13, 23$, $34, \ldots, 56$, $123, 
345+123, 346, 356, 456, 3456$ is a also Taylor basis of a minimal free resolution 
$\mathbf{G}$ of $S/M$. Note that 
\[
d(3456)=(345+123)-346+356-456-123,
\]
from which we see that $\mathbf{G}$ does not satisfy the definition of Betti-linearity. 
Let $d^{\approx}$ be the differential map in the frame of $\mathbf{G}^{\approx}$, 
then we have that 
\[
d^{\approx}(3456)=(345+123)-346+356-456, 
\]
which implies that $\mathbf{G}^{\approx}\neq \mathbf{G}$, so that 
$M$ is not strongly homology-linear.

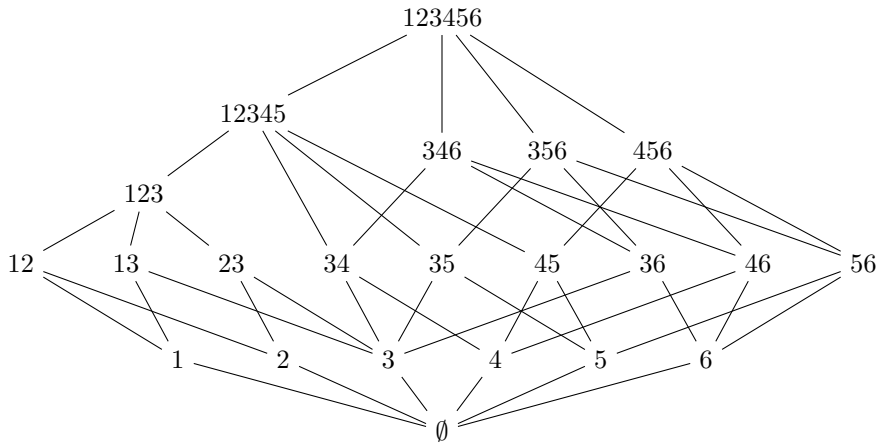
\begin{figure}
\begin{tikzpicture}[node distance=1.4cm]
\title{The lcm-lattice}
\node(123456)                                                                  {$123456$};
\node(346)      [below=1.3cm of 123456]                        {$346$};
\node(356)      [right of =346]                                          {$356$};
\node(456)      [right of =356]                                          {$456$};
\node(12345)  [below left =0.8cm and 1.3cm of 123456]   {$12345$};
\node(123)      [below left=0.6cm and 0.5cm of 12345]   {$123$};
\node(35)        [below=2.8cm of 123456]                            {$35$};
\node(34)        [left of =35]                                              {$34$};
\node(23)        [left of =34]                                              {$23$};
\node(13)        [left of =23]                                              {$13$};
\node(12)        [left of =13]                                              {$12$};
\node(45)        [right of =35]                                            {$45$};
\node(36)        [right of =45]                                            {$36$};
\node(46)        [right of =36]                                            {$46$};
\node(56)        [right of =46]                                            {$56$};
\node(3)          [below left=0.8cm and 0.2cm of 35]            {$3$};
\node(2)          [left of =3]                                                {$2$};
\node(1)          [left of =2]                                                {$1$};
\node(4)          [below right=0.8cm and 0.2cm of 35]          {$4$};
\node(5)          [right of =4]                                              {$5$};
\node(6)          [right of =5]                                              {$6$};
\node(0)          [below=5cm of 123456]                           {$\emptyset$};

\draw(123456) --(12345);
\draw(123456) --(346);
\draw(123456) --(356);
\draw(123456) --(456);
\draw(12345) --(123);
\draw(12345) --(34);
\draw(12345) --(35);
\draw(12345) --(45);
\draw(123) --(12);
\draw(123) --(13);
\draw(123) --(23);
\draw(346) --(34);
\draw(346) --(36);
\draw(346) --(46);
\draw(356) --(35);
\draw(356) --(36);
\draw(356) --(56);
\draw(456) --(45);
\draw(456) --(46);
\draw(456) --(56);
\draw(12) --(1);
\draw(12) --(2);
\draw(13) --(1);
\draw(13) --(3);
\draw(23) --(2);
\draw(23) --(3);
\draw(35) --(3);
\draw(35) --(5);
\draw(34) --(3);
\draw(34) --(4);
\draw(36) --(3);
\draw(36) --(6);
\draw(45) --(4);
\draw(45) --(5);
\draw(56) --(5);
\draw(56) --(6);
\draw(46) --(4);
\draw(46) --(6);
\draw(1) --(0);
\draw(2) --(0);
\draw(3) --(0);
\draw(4) --(0);
\draw(5) --(0);
\draw(6) --(0);

\end{tikzpicture}
\caption{The lcm-lattice $L_M$ of Example \ref{E:big2}}
\end{figure}

\end{example}

\begin{example}\label{E:SHL}
Let $M$ be a monomial ideal whose lcm-lattice is shown in Figure 8. 
Then similar to Example \ref{E:notBetti},  it is easy to see that $M$ is not Betti-linear. Since 
$\Delta_{123456}=\langle 123, 45, 46, 56 \rangle$, it follows that 
$\widetilde{H}_1(\Delta_{123456}; k)\cong k$ with a basis $[45-46+56]$, 
and $\widetilde{H}_0(\Delta_{123456}; k)\cong k$ with a basis $[-1+4]$. 
Note that $\widetilde{H}_0(\Delta_{123}; k)\cong k$ with a basis $[-1+3]$ 
and $\widetilde{H}_0(\Delta_{12}; k)\cong k$ with a basis $[-1+2]$, 
so that $M$ is not nearly homologically monotonic.  Next we use Definition 
\ref{D:RLM} to calculate $\mathcal{G}(L_M)$ and prove that 
$M$ is strongly homology-linear. 

Note that any basis element of $\widetilde{H}_1(\Delta_{123456}; k)$ 
can be written as $[\lambda(12-13+23)+\mu(45-46+56)]$, where 
$\lambda, \mu\in k$ and $\mu \neq 0$. Without the loss of generality 
we assume that $\lambda=\mu=1$. Note that 
$\Delta_{123456}^{(1)}=\Delta_{123456}$ and $\sigma_{123456}^{(1)}$ 
is the identity map. Similar to Example \ref{E:RLM} we choose $[-1+3]$, $[-1+4]$  
as the preimages of $[-1+3]$, $[-1+4]$ in $\widetilde{H}_0(\Delta_{123}^{(0)}; k)$,  
$\widetilde{H}_0(\Delta_{123456}^{(0)}; k)$, respectively. And we fix bases 
$[-1+2]$, $[-4+5]$, $[-4+6]$, $[-5+6]$ for $\widetilde{H}_0(\Delta_{12}; k)$, 
$\widetilde{H}_0(\Delta_{45}; k)$, $\widetilde{H}_0(\Delta_{46}; k)$, 
$\widetilde{H}_0(\Delta_{56}; k)$, respectively. Then by Definition \ref{D:RLM} 
we have that 
\[
\psi_{3, 123456}^{123}: \widetilde{H}_1(\Delta_{123456}; k) \to \widetilde{H}_0(\Delta_{123}; k)
\]
is a zero map. Indeed, from $\Delta_1=\langle 123 \rangle$ and 
$\Delta_2=\langle 45, 46, 56 \rangle$, we get $\Delta_1 \cap \Delta_2=\{\emptyset\}$, 
so that $\widetilde{H}_0(\Delta_1\cap \Delta_2; k)=0$, which implies that 
$\psi_{3, 123456}^{123}=0$; or, from $d(12-13+23)=0$ we see that 
$\psi_{3, 123456}^{123}=0$. Then by Definition \ref{D:RLM} we get 
\[
\mathcal{R}(L_M): 0\to k \xrightarrow{\left( \begin{smallmatrix} 0  \\ 0  \\ 1 \\
-1 \\ 1  \\ 0  \end{smallmatrix} \right)} k^6 \xrightarrow{\left( \begin{smallmatrix} 
-1 &-1 & 0 & 0 & 0 & -1 \\ 1 & 0 & 0& 0 & 0 & 0\\ 0 &1 & 0 & 0 & 0 & 0 \\ 
0 & 0 & -1 & -1 & 0& 1 \\ 0 & 0 & 1 & 0 & -1& 0 \\ 0 & 0 & 0 & 1 & 1& 0
\end{smallmatrix} \right)} k^6 \xrightarrow{(\begin{smallmatrix} 1& 1 & 1 & 1 & 1 &1 \end{smallmatrix})} k. 
\]
Hence, $\mathcal{G}(L_M)$ is a minimal free resolution of 
$S/M$ with a Taylor basis $\emptyset, 1, \ldots, 6$, $12$, $13$, $45$, $46$, $56, 14, 456$. 
And it is easy to see that every $\mathcal{G}(L_M)$ obtained by Definition \ref{D:RLM} 
is a minimal free resolution of $S/M$. So $M$ is strongly homology-linear. 

Next we use Taylor basis to give another proof of $M$ being strongly homology-linear. 
Let $\mathbf{F}$ be a minimal free resolution of $S/M$ with a Taylor basis 
$\emptyset, 1, \ldots, 6$, $12$, $13$, $45$, $46$, $56, 14, 456$. Then we 
have that $\mathbf{F}^{\approx}=\mathbf{F}$. Let $\mathbf{G}$ be another 
minimal free resolution of $S/M$, then $\mathbf{G}$ has a Taylor basis which can 
be obtained from the Taylor basis of $\mathbf{F}$ by a change of basis map 
as in Subsection 2.1. Note that $456$ is the only $2$-dimensional Taylor basis 
element, so that $456$ can only be changed to $\lambda456$ for some 
$\lambda \neq 0 \in k$. And for $1$-dimensional Taylor basis elements, one can, 
for example, change $13$ to $13+12$, or change $14$ to $-14+45$. Anyway, 
we still have $\mathbf{G}^{\approx}=\mathbf{G}$. So $M$ is strongly homology-linear.

\begin{figure}
\begin{tikzpicture}[node distance=1.4cm]
\title{The lcm-lattice}
\node(123456)                                                                      {$123456$};
\node(45)        [below right=1.6cm and -0.1cm of 123456]  {$45$};
\node(46)        [right of =45]                                                {$46$};
\node(56)        [right of =46]                                                {$56$};
\node(123)      [below left =0.6cm and 0.5cm of 123456]   {$123$};
\node(12)        [left=2.8cm of 45]          {$12$};

\node(4)          [below =0.7cm  of 45]                              {$4$};
\node(5)          [below =0.7cm  of 46]                              {$5$};
\node(6)          [below =0.7cm  of 56]                              {$6$};
\node(3)          [left=1.2cm of 4]                                              {$3$};
\node(2)          [left of =3]                            {$2$};
\node(1)          [left of =2]                                                {$1$};
\node(0)          [below=3.8cm of 123456]                           {$\emptyset$};

\draw(123456) --(123);
\draw(123456) --(45);
\draw(123456) --(46);
\draw(123456) --(56);
\draw(123) --(12);
\draw(123) --(3);
\draw(12) --(1);
\draw(12) --(2);
\draw(45) --(4);
\draw(45) --(5);
\draw(56) --(5);
\draw(56) --(6);
\draw(46) --(4);
\draw(46) --(6);
\draw(1) --(0);
\draw(2) --(0);
\draw(3) --(0);
\draw(4) --(0);
\draw(5) --(0);
\draw(6) --(0);

\end{tikzpicture}
\caption{The lcm-lattice $L_M$ of Example \ref{E:SHL}}
\end{figure}
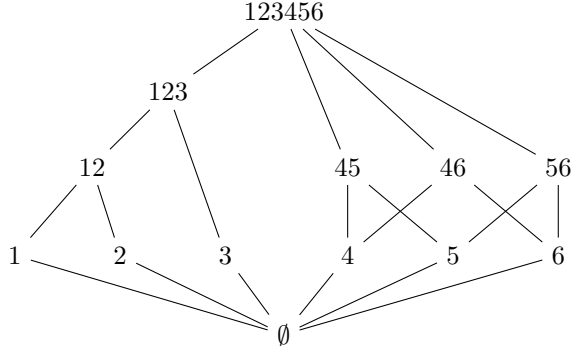

\end{example}

In many examples of Subsection 4.1 we have considered nearly Scarf monomial ideals. 
Next we prove that a nearly Scarf monomial ideal is nearly homologically monotonic. 

\begin{proposition}\label{P:NS}
Let $M$ be a nearly Scarf monomial ideal, then $M$ is nearly homologically monotonic.
\end{proposition}

\begin{proof}
Let $1<\widetilde{m}<m \in L_M$ such that $\widetilde{H}_i(\Delta_{\widetilde{m}}; k)\neq 0$ 
and $\widetilde{H}_i(\Delta_m; k)\neq 0$. Since $M$ is nearly Scarf, it follows that 
$\widetilde{m}$ is a Scarf multidegree and $m$ is the top element of $L_M$, so that 
there does not exist $\widehat{m}\in L_M$ such that $m<\widehat{m}$. So $M$ is 
nearly homologically monotonic. 
\end{proof}

As is shown by the next example, a nearly homologically monotonic monomial ideal 
may not be nearly Scarf. 

\begin{example}\label{E:NNS}
Let $M$ be a monomial ideal whose lcm-lattice is shown in Figure 9. 
Then $\widetilde{H}_0(\Delta_{1234}, k)\cong k$ and $\widetilde{H}_0(\Delta_{123}; k)\cong k^2$. 
So it is easy to see that $M$ is nearly homologically monotonic, but $M$ is not nearly Scarf or Betti-linear.

\begin{figure}
\begin{tikzpicture}[node distance=1.4cm]
\title{The lcm-lattice}
\node(1234)                                                                      {$1234$};
\node(123)      [below left =0.6cm and 0.1cm of 1234]   {$123$};
\node(2)          [below =0.6cm  of 123]                              {$2$};
\node(1)          [left of =2]                                                {$1$};
\node(3)          [right of =2]                              {$3$};
\node(4)          [right of =3]                              {$4$};
\node(0)          [below=2.5cm of 1234]                           {$\emptyset$};

\draw(1234) --(123);
\draw(1234) --(4);
\draw(123) --(1);
\draw(123) --(2);
\draw(123) --(3);
\draw(1) --(0);
\draw(2) --(0);
\draw(3) --(0);
\draw(4) --(0);

\end{tikzpicture}
\caption{The lcm-lattice $L_M$ of Example \ref{E:NNS}}
\end{figure}
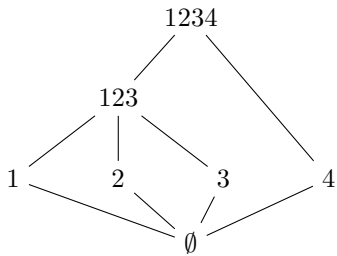

\end{example}

\begin{remark}\label{R:diagram2}
By the above results, similar to Remark \ref{R:diagram1}, 
we have the following diagram indicating the relations among some 
different classes of monomial ideals:
\[
\begin{array}{ccccccc}
\begin{matrix} \mbox{nearly}\\ \mbox{Scarf}\end{matrix} & \subset
& \begin{matrix} \mbox{nearly}\\ \mbox{HM}\end{matrix} & \subset 
& \begin{matrix} \mbox{strongly}\\ \mbox{homology-linear}\end{matrix}& \subset
& \mbox{homology-linear} \\
 & &\cup & &  &  &\cup \\
 & &\mbox{HM} & &\subset  &  & \mbox{Betti-linear} 
\end{array}
\]
where HM is the abbreviation for homologically monotonic, and 
all the inclusions in the above diagram are strict. This diagram can 
be viewed as a continuation of the diagram in Remark \ref{R:diagram1}. 

If $M$ is homologically monotonic, then Definition \ref{D:posetres} and 
Definition \ref{D:RLM} coincide, and then we have that 
$\mathcal{F}(L_M)=\mathcal{G}(L_M)$, which is a minimal free 
resolution of $S/M$. 
\end{remark}

Given a monomial ideal $M$ we can always use Contruction \ref{C:VL} 
to obtain a minimal free resolution of $S/M$. 
However, given a class of monomial ideals, it may not be easy to 
characterize the lcm-lattices of these monomial ideals. 
For example, although stable ideals have nice combinatorial structures and we have 
the Eliahou-Kervaire resolutions for stable ideals,  it seems hard to 
describe the lcm-lattices of stable ideals. 
Next we give an example of a stable ideal which is not homology linear.

\begin{example}\label{E:stable}
Let $M$ be a monomial ideal in $S=k[x, y, z]$ minimally generated by 
$m_1=x^2$, $m_2=xy$, $m_3=xz$,  $m_4=y^3$,  $m_5=y^2z$, 
$m_6=yz^2$,  $m_7=z^3$. Then $M$ is a stable ideal. 
Let $\mathbf{F}$ be the Eliahou-Kervaire resolution of $S/M$. 
Then by Construction \ref{C:basis3} and Theorem \ref{T:basis3} we 
get the following Taylor basis of $\mathbf{F}$: $\emptyset$, $1$, $\ldots$, $7$, 
$12$, $13$, $23$, $24$, $25$, $45$, $26$, $56$, $37$,  $67$, $123$, 
$245$, $256$, $367-236$. 
Let $\mathbf{G}$ be a minimal free resolution of $S/M$, 
then $\mathbf{G}$ has a Taylor basis which can 
be obtained from the Taylor basis of $\mathbf{F}$ by 
a change of basis map as in Subsection 2.1. 
Since we have that $\mbox{mdeg}(23)<\mbox{mdeg}(25)$ 
and $\mbox{mdeg}(23)<\mbox{mdeg}(26)$ in $L_M$, 
by using a rescaling of the basis elements, without the loss of generality  
we can assume that $\mathbf{G}$ has a Taylor basis: 
$\emptyset$, $1$, $\ldots$, $7$, $12$, $13$, $23$, $24$, 
$25+\lambda23$, $45$, $26+\mu23$, $56$, $37$,  $67$, $123$, 
$245$, $256$, $367-236$, where $\lambda, \mu$ are some scalars in $k$. 
Note that 
\begin{align*}
d(245)&=24-(25+\lambda23)+45+\lambda23 \\
d(256)&=(25+\lambda23)-(26+\mu23)+56+(-\lambda+\mu)23 \\
d(367-236)&=(26+\mu23)-37+67-(\mu+1)23. 
\end{align*}
Assume that $\mathbf{G}^{\approx}=\mathbf{G}$, then we have that 
$\lambda=-\lambda+\mu=\mu+1=0$, which is impossible. 
Thus, $\mathbf{G}^{\approx}\neq \mathbf{G}$ for any minimal free 
resolution $\mathbf{G}$ of $S/M$. So $M$ is not homology-linear. 
\end{example}

If we can find interesting classes of monomial ideals whose lcm-lattices 
have nice structures, then we may use their lcm-lattices to obtain 
formulas for their minimal free resolutions. Not much is known in this 
direction. 

In this paper, for convenience, we call the results in Subsection 3.2 
the atomic lattice resolution theory. We suggest calling all the results invloving 
the lcm-lattice of a monomial ideal the lcm-lattice resolution theory. 
Thus, this paper and most of the reference papers of this paper are about 
the lcm-lattice resolution theory. 

Example \ref{E:stable} gives us a glimpse that the study of monomial 
resolutions is a very diverse area of research. 
Besides the Stanley-Reisner theory, the cellular resolution theory, 
the lcm-lattice resolution theory, etc, there must exist other theories 
yet to be found.

\end{document}